\documentclass[reqno, 11pt]{amsart}

\usepackage{bbm}
\usepackage[top=2.6cm,left=2.7cm,right=2.7cm,bottom=2.6cm]{geometry}
\usepackage{amsfonts}
\usepackage{amsmath,amsthm,amscd,mathrsfs,amssymb,graphicx,enumerate,
	dsfont}

\usepackage{xypic}
\usepackage{hyperref}
\usepackage[usenames,dvipsnames]{xcolor}
\hypersetup{colorlinks=true,citecolor=NavyBlue,linkcolor=Brown,urlcolor=Orange}
\usepackage{tikz-cd}

\newcounter{1}

\newtheorem*{claim}{Claim}

\newtheorem{thm2}{Theorem}

\newtheorem{thm}{Theorem}[section]
\newtheorem{cor}[thm]{Corollary}
\newtheorem{lem}[thm]{Lemma}
\newtheorem{prop}[thm]{Proposition}
\newtheorem{conj}[thm]{Conjecture}

\theoremstyle{definition}
\newtheorem{defn}[thm]{Definition}

\newtheorem{ex}[thm]{Example}
\newtheorem{rmk}[thm]{Remark}

 \DeclareMathOperator{\Spec}{Spec}
\DeclareMathOperator{\Gal}{Gal} 
 
\DeclareMathOperator{\End}{End} 
\DeclareMathOperator{\Hom}{Hom} 
 \DeclareMathOperator{\rk}{rk}

\newcommand{\C}{\ensuremath\mathds{C}}
\newcommand{\R}{\ensuremath\mathds{R}}
\newcommand{\Z}{\ensuremath\mathds{Z}}
\newcommand{\Q}{\ensuremath\mathds{Q}}

\newcommand{\HH}{\ensuremath\mathrm{H}}

\newcommand{\CH}{\ensuremath\mathrm{CH}}
\newcommand{\h}{\ensuremath\mathfrak{h}}

\newcommand{\op}{\ensuremath\mathrm{op}}
\newcommand{\B}{\ensuremath\mathrm{B}}
\newcommand{\dR}{\ensuremath\mathrm{dR}}
\newcommand{\dRB}{\ensuremath\mathrm{dRB}}

\newcommand{\mot}{\ensuremath\mathrm{mot}}
\renewcommand{\And}{\ensuremath\mathrm{And}}
\newcommand{\an}{\ensuremath\mathrm{an}}

\newif\ifHideFoot
\HideFoottrue  
\HideFootfalse  

\ifHideFoot

\newcommand{\Charles}[1]{}
\newcommand{\Mingmin}[1]{}
\newcommand{\Tobias}[1]{}

\else

\newcommand{\marg}[1]{\normalsize{{
			\color{red}\footnote{{\color{blue}#1}}}{\marginpar[\vskip
			-.25cm{\color{red}\hfill$\Rightarrow$\tiny\thefootnote}]{\vskip
				-.2cm{\color{red}$\Leftarrow$\tiny\thefootnote}}}}}
\newcommand{\Mingmin}[1]{\marg{(Mingmin) #1}}
\newcommand{\Charles}[1]{\marg{(Charles) #1}}
\newcommand{\Tobias}[1]{\marg{(Tobias) #1}}

\fi

\AtBeginDocument{%
	\def\MR#1{}
}

\begin{document}
	
	\title[$L$-de\,Rham--Betti classes]{De\,Rham--Betti classes with coefficients}
    \author{Tobias Kreutz}
	\author{Mingmin Shen}
	\author{Charles Vial}
	
	
	
	\thanks{The research of M.Sh.\ was supported by the NWO under grant number 016.Vidi.189.015. The research of Ch.V.\ was funded by the Deutsche Forschungsgemeinschaft (DFG, German Research Foundation) -- Project-ID 491392403 - TRR 358.}

	
%
%

	\date{\today}
	
	\begin{abstract} Let $K$ and $L$ be algebraic extensions of the rational numbers inside the field of complex numbers.
An $L$-\emph{de\,Rham--Betti class} on a smooth projective variety $X$ over 
$K$ 
 is a 
class in the Betti cohomology with $L$-coefficients of the analytification of~$X$ that descends to a class in the algebraic de Rham cohomology of $X$ via the period comparison isomorphism.
The period conjecture of Grothendieck implies that $L$-de\,Rham--Betti classes should be $L$-linear combinations of algebraic cycle classes.
We prove that $L$-de\,Rham--Betti classes on products of elliptic curves are $L$-linear combinations of algebraic classes, provided $L$ contains at most one of the CM 
fields associated with the CM elliptic curves involved in the product. 
A key step consists in establishing a version of the analytic subgroup theorem with $L$-coefficients.
Moreover, building on results of Deligne and Andr\'e regarding the Kuga--Satake correspondence, 
 we show that codimension-2 $L$-de\,Rham--Betti classes on hyper-K\"ahler varieties of known deformation type are $L$-linear combinations of
motivated cycles, and we obtain a global de\,Rham--Betti Torelli theorem for K3 surfaces defined over the algebraic numbers.
	\end{abstract}
	\maketitle

\vspace{-25pt}
	\section{Introduction}

Throughout this paper, $\overline \Q$ denotes the algebraic closure of $\Q$ inside $\C$ and we let $K$ and $L$ be subfields of $\overline \Q$. 
\medskip

\paragraph{\textbf{$L$-de\,Rham--Betti classes.}}
Let $X$ be a smooth projective variety defined over $K$.
Serre's GAGA and the analytic Poincar\'e lemma provide a canonical isomorphism
\begin{equation}\label{E:comparison}
c^n_X : \HH^{n}_\dR(X/K) \otimes_K \C \stackrel{\simeq}{\longrightarrow} \HH^{n}_\B(X_\C^\an,\Q)\otimes_\Q \C,
\end{equation}
called the \emph{period comparison isomorphism}, between the algebraic de\,Rham cohomology of $X_\C$ and the Betti cohomology with $\C$-coefficients of the analytification $X_\C^\an$ of $X_\C$. Writing as is usual $\Q(k) =_{\mathrm{def}} (2\pi i)^{k}\Q \subset \C$, the \emph{de\,Rham--Betti cohomology} of $X$ is defined as the triple
$$\HH^n_\dRB(X,\Q(k)) =_{\mathrm{def}} \big( \HH^{n}_\dR(X/K),  \HH^{n}_\B(X_\C^\an,\Q(k)),  c_X^n\big).$$
 If $Z$ denotes an algebraic cycle of codimension~$k$ on~$X$, then by e.g.\ \cite[Prop.~1.1]{BostCharles} its classes in de\,Rham cohomology $ \HH^{2k}_\dR(X/K)$ and in Betti cohomology $\HH^{2k}_\B(X_\C^\an,\Q(k))$ are related, up to a sign, by 
\begin{equation}\label{E:cycleclass}
	c^{2k}_X\big(\mathrm{cl}_\dR(Z)\otimes_K 1_\C \big) = \mathrm{cl}_\B(Z) \otimes_{\Q} 1_\C.
\end{equation}
An $L$-\emph{de\,Rham--Betti class} in $\HH^n_\dRB(X,\Q(k))$
consists of a pair $(\alpha_\dR,\alpha_\B)$ with $\alpha_\dR\in \HH^{n}_\dR(X/K)$  and $\alpha_B \in\HH^{n}_\B(X_\C^\an,\Q(k))\otimes_\Q L$ 
whose complexifications correspond to one another under $c_X$\,; see Definitions~\ref{D:dRBobject} and~\ref{D:dRB}.  
By \eqref{E:cycleclass}, the Betti class and the de~Rham class of an algebraic cycle of codimension $k$ define a $\Q$-de\,Rham--Betti class in $\HH^{2k}_\dRB(X,\Q(k))$.

A consequence of the Grothendieck Period Conjecture is that $L$-de\,Rham--Betti classes in $\HH_{\dRB}^j(X,\Q(k))$ are 
$L$-linear combinations of $\Q$-de\,Rham--Betti classes\,; see Proposition~\ref{P:GPC-LdRB}. It turns out that even this statement is out of reach in general. It cannot be formal, as it fails for general de\,Rham--Betti objects\,; see Example~\ref{Ex:dRB-Qbar}.
Phenomena for which fields of coefficients play a non-trivial role are not uncommon.
For instance, 
it is expected from the standard conjectures, but still unclear in positive characteristic, that the $\Q_\ell$-span of algebraic classes in $\ell$-adic cohomology is isomorphic to the $\Q$-span of algebraic classes in $\ell$-adic cohomology tensored with~$\Q_\ell$\,; see \cite{Clozel}, \cite[Rmk.~6.9]{Ancona} and \cite{Ancona-Fratila}.
The subtlety of working with de\,Rham--Betti classes with coefficients also arises in the theory of Shimura periods for CM abelian varieties\,:
monomial relations among Shimura periods on CM abelian varieties are governed exactly by $\overline \Q$-de\,Rham--Betti classes,
so that the $\overline \Q$-algebraicity of $\overline{\Q}$-de\,Rham--Betti classes on products of CM abelian varieties implies that all monomial relations among Shimura periods are generated by the so-called monomial Shimura relations\,; see \S \ref{SS:Shimura}.
All these considerations underscore the importance of a systematic study of 
de\,Rham--Betti classes with coefficients and their behavior under extensions of scalar coefficients.

 The Grothendieck Period Conjecture, as formulated in Conjecture~\ref{C:GPC}, implies in fact that  $L$-de\,Rham--Betti classes in $\HH_{\dRB}^j(X,\Q(k))$ are 
\emph{$L$-algebraic}, i.e., are $L$-linear combinations of algebraic classes. 
It remains wide open\,: 
it is essentially known in the case of varieties whose motives are direct sums of Lefschetz motives (transcendence of $\pi$), and in the case of CM elliptic curves due to Chudnovsky~\cite{Chudnovsky}.
Nonetheless, the algebraicity of  $\Q$-de\,Rham--Betti classes could be established
  in some cases without establishing the Grothendieck Period Conjecture.
 Here is an exhaustive list\,:  
 for  $\Q$-de\,Rham--Betti classes in $\HH_{\dRB}^2(X,\Q(1))$ for $X$ an abelian variety~\cite{AndreBook, Bost, BostCharles} 
 as an application of W\"ustholz's analytic subgroup theorem~\cite{Wuestholz} and, 
 via the Kuga--Satake correspondence of Andr\'e~\cite{AndreK3}, 
 for  $X$ a hyper-K\"ahler variety~\cite{BostCharles}\,; 
 see Theorem~\ref{T:dRB1A} and Proposition~\ref{P:BC}.
The aim of this work is to extend the above results to $L$-de\,Rham--Betti classes and to higher codimension classes.
\medskip

\paragraph{\textbf{$L$-de\,Rham--Betti classes of codimension~1 on abelian varieties}}
As a first result, which will be instrumental for understanding $L$-de\,Rham--Betti classes on abelian varieties and hyper-K\"ahler varieties,
we establish a version with $L$-coefficients of W\"ustholz's analytic subgroup theorem (Proposition~\ref{Qbarsubgroupmotives}) and as a consequence show\,:
\begin{thm2}[Corollary~\ref{cor:LdRB-abelian}]\label{T:LdRB-abelian}
	Let $A$ be an abelian variety over $K$. Then\,:
		\begin{enumerate}[(i)]
		\setcounter{enumi}{\value{1}}
		\item any $L$-de\,Rham--Betti class in $\HH^{2}_{\dRB}(A,\Q(1))$ is $L$-algebraic\,;
		\item the $L$-de\,Rham--Betti object $\HH^{1}_{\dRB}(A,L(k))$ is semi-simple for all $k\in \Z$. 
	\end{enumerate}
\end{thm2}
We refer to 
\S \ref{SS:dRBdef} for the notion of $L$-de\,Rham--Betti object. 
\medskip

\paragraph{\textbf{$L$-De\,Rham--Betti classes on products of elliptic curves}} 
Our next result, parts of which seem to be known to experts (see, e.g., \cite[footnote 12]{Andre} 
where Andr\'e mentions that Theorem~\ref{T:main-elliptic} can be proved for $\Q$-de\,Rham--Betti classes on powers of a non-CM elliptic curve), 
is the following\,:

\begin{thm2}[Theorem~\ref{T:dRB-elliptic}]\label{T:main-elliptic}
	Let $X$ be an abelian variety over $K$ such that $X_{\overline \Q}$ is isogenous to a product of elliptic curves $E_i$.
	Assume that $L$ contains at most one of the CM fields associated with the CM elliptic curves among the elliptic curves $E_i$.
	Then\,:
	\begin{enumerate}[(i)]
		\setcounter{enumi}{\value{1}}
		\item any $L$-de\,Rham--Betti class in $\HH^{2k}_{\dRB}(X,\Q(k))$ is $L$-algebraic for all $k$\,;
		\item any $L$-de\,Rham--Betti class in $\HH^{j}_{\dRB}(X,\Q(k))$ is zero for $j\neq 2k$.
	\end{enumerate}
\end{thm2}

We refer also to  Theorem~\ref{T:abelian-surfaces} for the case of powers of abelian surfaces whose base-change to~$\overline \Q$ have non-trivial endomorphism ring. 
The reason we have to set restrictions on the field of coefficients $L$ is outlined in Remark~\ref{rmk:elliptic}.
Theorem~\ref{T:main-elliptic} is the analogue in the de\,Rham--Betti setting of the following results.
\begin{enumerate}[(a)]
	\item The Hodge conjecture holds for products of complex elliptic curves. Its proof essentially goes back to Tate (unpublished)\,; see \cite[\S 3]{Hodge-survey} for a proof and further references.
	\item The Tate conjecture holds for products of elliptic curves over a finitely generated extension of $\Q$. This follows from the validity of the Mumford--Tate conjecture for products of elliptic curves. The latter is established in \cite[Cor.~1.2]{Lombardo} and builds on the validity of the Mumford--Tate conjecture for elliptic curves due to Serre~\cite{Serre}.
	\item The Tate conjecture holds for products of elliptic curves over a finite field. This is due to Spie\ss~\cite{Spiess}.
\end{enumerate}

The proof of Theorem~\ref{T:main-elliptic} uses the Tannakian formalism\,:
 it consists in showing the stronger statement that 
the \emph{$L$-de\,Rham--Betti torsor of periods} agrees with the \emph{motivic torsor of periods}, i.e., 
that the de\,Rham--Betti conjecture holds for $X_{\overline \Q}$ with $L$-coefficients (Conjecture~\ref{C:dRB}).
Crucial ingredients include the connectedness of the $L$-de\,Rham--Betti torsor of periods (Theorem~\ref{T:connectedness}) and the aforementioned W\"ustholz's analytic subgroup theorem with $L$-coefficients.
A difficulty is that, in contrast to the Hodge setting and the Mumford--Tate group, 
there is \emph{a priori} no theory of weights in the de\,Rham--Betti setting, i.e., it is not clear that the Tannakian $L$-de\,Rham--Betti group contains the scalar matrices.
For instance, for a smooth projective variety $X$ over~$K$, it is not known in general that   $\HH^{j}_{\dRB}(X,\Q(k))$ does not support any non-zero $L$-de\,Rham--Betti class for $j\neq 2k$. 
 Theorem~\ref{T:main-elliptic}$(ii)$ confirms that this is indeed the case for products of elliptic curves in case $L=\Q$. 
The only general results that were known so far in that direction are the following\,:
\begin{enumerate}[(a)]
	\item Any $\Q$-de\,Rham--Betti class in   $\HH^{0}_{\dRB}(X,\Q(k))$ is zero for $k\neq 0$\,; this is equivalent to the transcendence of $\pi$.
	\item Any $\Q$-de\,Rham--Betti class in   $\HH^{1}_{\dRB}(X,\Q(k))$ is zero for
	$k=0$ and $k=1$ by \cite[Thms.~4.1 \& 4.2]{BostCharles}.
	We observe in Theorem~\ref{T:H1} that W\"ustholz's analytic subgroup theorem in fact implies that any de\,Rham--Betti class in   $\HH^{1}_{\dRB}(X,\Q(k))$ is zero for any $k\in \Z$. 
\end{enumerate}
\medskip

\paragraph{\textbf{$L$-de\,Rham--Betti classes on hyper-K\"ahler varieties}}
Let $X$ be a polarized {hyper-K\"ahler variety} over~$K$.
Here a \emph{hyper-K\"ahler variety} over~$K$ means a variety over~$K$ whose base-change to $\C$ is projective, irreducible holomorphic symplectic and, deviating from the usual definition, is such that its second Betti number is at least~3.
We denote $ \mathrm{P}^2_\dRB(X,\Q)$ its de\,Rham--Betti \emph{primitive cohomology} -- it is the orthogonal complement to the class of the polarization -- and $ \mathrm{P}^2_\dRB(X,L)$ its base-change on the Betti side to $L$. 
Theorem~\ref{T:LdRB-abelian}, together with Andr\'e's result \cite{AndreK3} that the Kuga--Satake correspondence is motivated and defined over~$\overline \Q$, implies easily that $L$-de\,Rham--Betti classes in $\HH_\dRB^2(X,\Q(1))$ are $L$-algebraic\,; see Proposition~\ref{P:BC}. 
Our aim is to extend such results to higher-codimensional $L$-de\,Rham--Betti classes.

\begin{thm2}[special instance of Theorem~\ref{P:GdRBisometry}]
	\label{T:main-dRBisometry}
	Let $X$ and $X'$ be hyper-K\"ahler varieties over $K$.
	Any $L$-de\,Rham--Betti isometry 
	$$ \mathrm{P}^2_\dRB(X,L) \stackrel{\sim}{\longrightarrow} \mathrm{P}^2_{\dRB}(X',L)$$
	is  $L$-motivated.
\end{thm2}

Combining Theorem~\ref{T:main-dRBisometry} with the Hodge-theoretic Torelli theorem for K3 surfaces, we obtain\,:

\begin{thm2}[Global de\,Rham--Betti Torelli theorem for K3 surfaces over $\overline \Q$\,; see Theorem~\ref{T2:dRB-Torelli}]
	\label{T:dRB-Torelli}
	Let $S$ and $S'$ be two K3 surfaces over $\overline{\Q}$. If there exists  an integral de\,Rham--Betti class in $\HH^4_\dRB(S \times S',\Z(2))$ inducing an isometry
	$$\HH^2_\dRB(S,\Z) \stackrel{\sim}{\longrightarrow} \HH^2_\dRB(S',\Z),$$
	then $S$ and $S'$ are isomorphic.
\end{thm2}

A hyper-K\"ahler variety will be said to be \emph{of known deformation type} if its analytification is deformation equivalent to the Hilbert scheme of points on a K3 surface, a generalized Kummer variety, or one of O'Grady's two sporadic examples (all of these have second Betti number larger than 3).
Building on Theorem~\ref{T:main-dRBisometry}, we obtain the following.

\begin{thm2}[Proposition~\ref{P:dRB-tensor2} and  Theorem~\ref{T:dRB-cod2}] 
	\label{T:main-dRB}
	Let $X$ be a hyper-K\"ahler variety over $K$
	  and let $n\in \Z_{>0}$. 
	 Then\,:
	\begin{enumerate}[(i)]
		\item any $L$-de\,Rham--Betti class in $\HH^2_\dRB(X,\Q(1))\otimes \HH^2_\dRB(X, \Q(1))$ is $L$-motivated.
			\setcounter{1}{\value{enumi}}
	\end{enumerate}
If in addition $X$ is of known deformation type, then\,:
\begin{enumerate}[(i)]
\setcounter{enumi}{\value{1}}
		\item  any $L$-de\,Rham--Betti class in $\HH^4_\dRB(X,\Q(2))$ is  $L$-motivated.
	\end{enumerate}
\end{thm2}

As a first consequence of Theorem~\ref{T:main-dRB}$(i)$, we obtain in Corollary~\ref{C:T2} that, for any $k\in \Z$, any $\Q$-de\,Rham--Betti class in $\mathrm{H}^2_\dRB(X, \Q(k))$ is algebraic and thus zero for $k\neq 1$. 
Furthermore, 
we can derive from Theorem~\ref{T:main-dRB} that de\,Rham--Betti classes of any codimension on hyper-K\"ahler varieties of large Picard rank are motivated\,:

\begin{thm2}\label{T:main-GPC-CM}
Let $X$ be a hyper-K\"ahler variety over $K$ of known deformation type. Denote by $\rho^c(X)$ the \emph{Picard corank} of $X$, that is, $\rho^c(X) =_{\mathrm{def}} h^{1,1}(X_\C^\an) - \rho(X_\C)$ where $\rho(X_\C)$ is the Picard rank of $X_\C$.
\begin{enumerate}[(i)]
	\item  If $\rho^c(X)\leq 1$, then any $L$-de\,Rham--Betti class in $\HH^{j}_{\dRB}(X^n,\Q(k))$ is $L$-motivated.
	\item  If $\rho^c(X)=2$, then any $\Q$-de\,Rham--Betti class in $\HH^{j}_{\dRB}(X^n, \Q(k))$ is motivated.
\end{enumerate}
\end{thm2}

We refer to the more general (due to Proposition~\ref{P:torsor-fullyfaithful}) Theorem~\ref{T:main-GPC-maxPic}.
\medskip

\paragraph{\textbf{Outline}} In an effort to distinguish formal properties of de\,Rham--Betti objects with properties of geometric de\,Rham--Betti objects, we will focus first in \S \ref{S:dRB} and~ \S\ref{S:dRB-barQ} exclusively on general de\,Rham--Betti objects (with various fields of coefficients, both on the de\,Rham and the Betti side), and only starting with \S\ref{S:GPC} outline which properties are expected, or have been established, in the geometric setting. 
We then explore the algebraicity of $L$-de\,Rham--Betti classes for abelian varieties in \S \ref{S:ab} and for hyper-K\"ahler varieties in \S \ref{S:cod2}.
\medskip

\paragraph{\textbf{Acknowledgments}} Thanks to Giuseppe Ancona, Yves Andr\'e, Joseph Ayoub and Francis Brown for useful comments and/or discussions. Ch.V.\ is grateful to All Souls College, Oxford, for a Visiting Fellowship in 2025 providing excellent working conditions.
\medskip

\paragraph{\textbf{Note}} After the first version of this work appeared on the arXiv, Bruno Kahn uploaded to the arXiv a paper containing another proof of Theorem~\ref{T:main-elliptic}$(i)$ in the special case $L=\Q$\,; this work is now published as \cite{kahn}.
\medskip

\paragraph{\textbf{Dedication}} Tobias Kreutz passed away in August 2024. 
M.Sh.\ and Ch.V.\ have fond memories of the times spent together with Tobias in Amsterdam and Bielefeld collaborating on this work. His insights and passion for mathematics are profoundly missed.

\section{De\,Rham--Betti objects}\label{S:dRB}

\subsection{The Tannakian category of de\,Rham--Betti objects} 
\label{SS:dRBdef}

\begin{defn}[De\,Rham--Betti objects]
	\label{D:dRBobject}
The category of \emph{$L$-de\,Rham--Betti objects over $K$} is the $(K\cap L)$-linear category $\mathscr C_{\dRB, K_\dR, L_\B}$ whose objects $M$
are triples of the form $$M = (M_{\dR}, M_\B, c_M),$$ where $M_{\dR}$ is a finite-dimensional $K$-vector space, $M_\B$ is a finite-dimensional $L$-vector space and $c_M : M_{\dR}\otimes_K \C \to M_\B \otimes_{L} \C$ is a
$\C$-linear isomorphism. 
A \emph{de\,Rham--Betti homomorphism} $f\in \Hom_{\dRB}(M,N)$ between $L$-de\,Rham--Betti objects over $K$ consists of a $K$-linear map $f_\dR : M_\dR \to N_\dR$ together with an $L$-linear map $f_\B : M_\B \to N_\B$ such that the diagram
$$\xymatrix{M_\dR \otimes_K \C \ar[d]_{c_M} \ar[rr]^{f_\dR \otimes_K \mathrm{id}_\C} && N_\dR \otimes_K \C \ar[d]^{c_N} \\
M_\B \otimes_L \C \ar[rr]^{f_\B\otimes_{L}\mathrm{id}_\C} & & N_\B \otimes_L \C}$$
commutes.

For any $k\in \Z$, we denote $\mathds 1(k)$ the object of $\mathscr C_{\dRB,K_\dR,L_\B}$ defined by
$$\mathds{1}(k)_\dR := K, \quad \mathds{1}(k)_\B := (2\pi i)^{k}L, \quad \mbox{and} \quad c_{\mathds{1}(k)} : \C \to \C, z\mapsto  z$$
or equivalently by
$$\mathds{1}(k)_\dR := K, \quad \mathds{1}(k)_\B := L, \quad \mbox{and} \quad c_{\mathds{1}(k)} : \C \to \C, z\mapsto (2\pi i)^{-k} z.$$
The category $\mathscr C_{\dRB, K_\dR, L_\B}$ is then naturally an abelian rigid $\otimes$-category, with unit object $\mathds 1 := \mathds 1(0)$. 
The two natural forgetful functors
$$\omega_\dR \colon\mathscr C_{\dRB, K_\dR, L_\B} \to \operatorname{Vec}_{K}\qquad \mathrm{and} \qquad
\omega_\B \colon \mathscr C_{\dRB, K_\dR, L_\B} \to \operatorname{Vec}_{L} $$ 
define fiber functors, each thereby endowing $\mathscr C_{\dRB, K_\dR, L_\B}$ with the structure of a Tannakian category.
If $L\subseteq K$ (resp.\ $K\subseteq L$), then $\mathscr C_{\dRB, K_\dR, L_\B}$ is neutralized by $\omega_\B$ (resp.\ $\omega_{\dR}$) and is thus a \emph{neutral} Tannakian category.
\end{defn}

\begin{defn}[Base-change of de\,Rham--Betti objects]
Let   $M = (M_\dR,M_\B,c_M)\in \mathscr C_{\dRB, K_\dR, L_\B}$ be an $L$-de\,Rham--Betti object  over $K$,  and let $K\subseteq K'\subseteq \overline \Q$ and $L\subseteq L' \subseteq \overline \Q$ be field extensions. We denote
$$M_{K'}\otimes L' =_{\mathrm{def}} (M_\dR \otimes_K K', M_\B\otimes_L L', c_M)$$ the object in $\mathscr C_{\dRB, K'_\dR, L'_\B}$ obtained from base-change of $M$. 
\end{defn}

\subsection{De\,Rham--Betti classes with coefficients}
\begin{defn}[$L'$-de\,Rham--Betti classes] 
	\label{D:dRB}
	Let $M$ be a de\,Rham--Betti object over $K$ with $L$-coefficients, and let $L\subseteq L' \subseteq \overline{\Q}$ be a field extension. 
	An~ $L'$-\emph{de\,Rham--Betti class} on $M$ is an element of $\Hom_{\mathscr C_{\dRB,K, L'}}(\mathds 1, M\otimes_L L')$. 
	Equivalently, it consists of a pair of elements  $\alpha_{\dR} \in M_\dR$ and $\alpha_B \in M_\B\otimes_L L'$ with $c_M(\alpha_\dR \otimes_K 1_\C) = \alpha_{\B}\otimes_{L'} 1_\C$.
	\\
	If $L = L'$, unless it seems necessary for clarity, we will more succinctly say \emph{de\,Rham--Betti class}. Hence an $L'$-{de\,Rham--Betti class} on $M$ is exactly a {de\,Rham--Betti class} on $M\otimes_L L'$.
\end{defn}

As we will see in Example~\ref{Ex:dRB-Qbar}, and this is one of the crucial basic observations of this work, an $L'$-de\,Rham--Betti class on $M$ is not necessarily an $L'$-linear combination of de\,Rham--Betti classes on $M$.

\subsection{The de\,Rham--Betti group and the de\,Rham--Betti torsor of periods}
\label{SS:dRB-torsor}
In this work we will be mostly interested in the case where the fiber functor $\omega_\B$ is neutral, i.e., in the case $L \subseteq K$. There are two reasons for this\,: first, the Tannakian formalism is simpler in presence of a neutral fiber functor; second and most importantly, this ensures the existence of a \emph{de\,Rham--Betti realization functor} from the category of Andr\'e motives over $K$ with $L$-coefficients, see \S\ref{SS:dRB-Andre}.

 In fact, although we will aim for the most general situation, the following three cases are essentially the ones we have in mind\,:
\begin{enumerate}[(a)]
	\item $K\subseteq \overline{\Q}$ is arbitrary and $L=\Q$. 
	\label{(a)}
	\item $K=\overline \Q$ and $L=\Q$. This is a special instance of (a). 
Objects of  $\mathscr C_{\dRB} := \mathscr C_{\dRB, \overline \Q_\dR, \Q_\B}$ will be simply called \emph{de\,Rham--Betti objects} if no confusion arises.
	\label{(b)}
	\item $K=\overline \Q$ and $L=\overline \Q$. 
	Object of $\mathscr{C}_{\dRB, \overline \Q,\overline \Q}$ will be sometimes called \emph{$\overline \Q$-de\,Rham--Betti objects}. These appear for instance in \cite{GUY} in the study of Shimura periods of CM abelian varieties over $\overline \Q$ under the name of bi-$\overline \Q$-structures.
	\label{(c)}
\end{enumerate}
Note however that we will also consider $L$-de\,Rham--Betti classes on de\,Rham--Betti objects over~$K$ for any subfields $K,L \subseteq \overline \Q$, for instance in Lemma~\ref{L:descent} or Corollary~\ref{cor:LdRB-abelian}, but such considerations will not involve the Tannakian setting.
\medskip

\begin{defn}[The de\,Rham--Betti group and the de\,Rham--Betti torsor of periods]
Assume $L\subseteq K$.
 The \emph{de\,Rham--Betti group} $G_\dRB$ of $\mathscr C_{\dRB,K, L}$ is the Tannakian fundamental group of the neutral Tannakian category $\mathscr C_{\dRB,K, L}$,
  that is, $$G_{\dRB,K,L}=_{\mathrm{def}} \operatorname{Aut}^\otimes(\omega_B : \mathscr C_{\dRB,K, L} \to \mathrm{Vec}_L).$$
The \emph{de\,Rham--Betti torsor of periods} is the Tannakian torsor $$\Omega^\dRB_{K,L} =_{\mathrm{def}} \operatorname{Iso}^{\otimes}(\omega_{\dR},\omega_{\B}\otimes_L K)\,;$$
 it is a torsor under $\mathrm{Aut}^\otimes(\omega_{\B} \otimes_L K)$, 
 which coincides with $(G_{\dRB,K,L})_K=_{\mathrm{def}} G_{\dRB,K,L} \times_{L} K$ by \cite[Rmk.~3.12]{DeligneMilne}.
\end{defn}

For a de\,Rham--Betti object $M \in \mathscr C_{\dRB,K, L}$, 
we denote 
$$G_\dRB(M) =_{\mathrm{def}} \operatorname{Aut}^\otimes(\omega_B|_{\langle M\rangle})
\quad \text{and} \quad
\Omega^\dRB_M =_{\mathrm{def}} \operatorname{Iso}^{\otimes}(\omega_{\dR}|_{\langle M\rangle},\omega_{\B}|_{\langle M\rangle}\otimes_L K).$$
The comparison isomorphism $c_M\colon M_\dR \otimes_{K} \C \stackrel{\sim}{\longrightarrow} M_\B \otimes_{L} \C$ defines a complex-valued point 
$$c_M \in \Omega_M^\dRB(\C).$$
From the general formalism of neutral Tannakian categories, for any object $M \in \mathscr C_{\dRB,K, L}$, there is a natural epimorphism $G_{\dRB,K,L} \twoheadrightarrow G_\dRB(M)$ and for any object $X \in \langle M\rangle$ the action of $G_{\dRB,K,L}$ on $\omega_\B(X)$ factors through $G_\dRB(M)$. Moreover, an element of $M_\B^{\otimes n} \otimes_L (M_\B^\vee)^{\otimes m}$ extends to a de\,Rham--Betti class in $M^{\otimes n} \otimes (M^\vee)^{\otimes m}$ if and only if it is fixed by $G_{\dRB}(M)$.
As an example, for $\mathds{1}(k) \in  \mathscr C_{\dRB,K, L}$,
we have $$G_\dRB(\mathds 1(k)) = \begin{cases}
\{1\}, \quad\text{if } k=0\\
\mathds{G}_m, \quad \text{if } k\neq 0.
\end{cases}$$
Indeed, $G_\dRB(\mathds 1(k))$ is a $L$-subgroup of $\mathds{G}_m$.
Suppose it is finite, say of order $n$. Then $G_\dRB(\mathds 1(k))$ acts trivially on $(\mathds 1(k)^{\otimes n})_\B$ so that $\mathds 1(nk) \simeq \mathds{1}(k)^{\otimes n} \simeq \mathds 1$ in the category $\mathscr C_{\dRB,K, L}$. 
This implies that $(2\pi i)^{nk}$ lies in $K$ and hence that $k=0$ by the transcendence of~$\pi$.  
(Alternately, we will show in Theorem~\ref{T:connectedness} that $G_\dRB$ is connected.)
Conversely, it is clear that $G_\dRB(\mathds 1) = \{1\}$.

\subsection{The de\,Rham--Betti torsor of periods is connected}\label{SS:connected}
In characteristic zero, we have the following criterion for the connectedness of the Tannakian fundamental group of a neutral Tannakian category.
\begin{prop}\label{P:connected}
Let $\mathscr T$ be a neutral Tannakian category with neutral fiber functor $\omega : \mathscr T \to \operatorname{Vec}_F$, with $\operatorname{char}(F)=0$. 
 The following statements are equivalent.
	\begin{enumerate}[(i)]
		\item The Tannakian fundamental group $G=_{\mathrm{def}}\operatorname{Aut}^\otimes(\omega)$ is connected.
		\item For any object $X$ of $\mathscr T$, if its Tannakian fundamental group $G_X$ is finite, then it is trivial.
	\end{enumerate}
\end{prop}
\begin{proof}
	This is essentially the criterion in \cite[Cor.~2.22]{DeligneMilne}. In details, since $\mathrm{char}(F)=0$, the group $G$ is connected if and only if 
	there is no non-trivial epimorphism from $G$ to a finite group scheme. Thus, by \cite[Prop.~2.21(a)]{DeligneMilne}, 
	$G$ is connected if and only if $\mathscr T$ does not have any Tannakian subcategory with non-trivial finite Tannakian fundamental group. By \cite[Prop.~2.20(b)]{DeligneMilne}, such a Tannakian subcategory is generated by an object $X$.
\end{proof}

We assume $K=\overline \Q$.
The following observation is purely formal, 
in the sense that it applies to any de\,Rham--Betti object
  and not only to those coming from geometry (that is, those that are the realization of Andr\'e motives\,; see \S\ref{SS:dRB-Andre} below).

\begin{thm}\label{T:connectedness}
	Both the de\,Rham--Betti group $G_\dRB$ and the de\,Rham--Betti torsor of periods $\Omega^\dRB$ associated with $\mathscr C_{\dRB,\overline{\Q}, L}$ are connected. 
	In particular, for any de\,Rham--Betti object $M$ over~$\overline{\Q}$ with $L$-coefficients, 
	both $G_\dRB(M)$ and $\Omega_M^\dRB$ are connected.
\end{thm}
\begin{proof} It suffices to show that $G_{\dRB}$ is connected. 
	According to Proposition~\ref{P:connected}, it suffices to show that, for any object $M \in \mathscr C_\dRB$ with finite de\,Rham--Betti group, $G_\dRB(M)$ is trivial.
	Since the de\,Rham--Betti torsor of periods  $\Omega_M^\dRB$ of such an object $M$ is then finite, 
	the comparison isomorphism $c_M : M_{\dR}\otimes_{\overline \Q} \C \to M_\B\otimes_{L} \C$ is defined over $\overline{\Q}$.
	Choosing an $L$-basis $(e_i)$ of $M_\B$ and letting $(c_M^{-1}(e_i \otimes_{\Q} 1_{\overline{\Q}}))$ be the corresponding $\overline \Q$-basis of $M_\dR$, we see that $M \simeq \mathds 1^{\oplus \dim M_\B}$ in $\mathscr C_{\dRB,\overline{\Q}, L}$. 
	This implies  $G_\dRB(M) = G_\dRB(\mathds 1) = 1$.
\end{proof}

\begin{rmk}[Connectedness fails when $L=K=\Q$]
 Consider for instance the $\Q$-de\,Rham--Betti object over $\Q$ given by $M=(\Q,\Q,e^{\frac{2\pi i}{p}})$,
where $p$ is an odd prime. The de\,Rham--Betti group $G_\dRB(M)$ is then $\mu_p$ and $\Omega_M^\dRB$ is the trivial torsor under $\mu_p$, so that both are not connected.
\end{rmk}

\subsection{The de\,Rham Betti group is not pro-reductive}
The Tannakian category $\mathscr C_{\dRB, K_\dR, L_\B}$ is not semi-simple, as can be seen from considering the object
\[
M:=\left(M_\dR=K^{\oplus 2}, M_\B = L^{\oplus 2}, c_M\right), \quad \text{with }c_M = \begin{pmatrix}
\alpha & \beta \\ 0 &\gamma
\end{pmatrix},
\quad \mathrm{degtr}_\Q \Q(\alpha,\beta, \gamma) = 3.
\]
Equivalently,
if $L\subseteq K$, $G_{\dRB,K,L}$ has a non-reductive quotient, i.e., is not pro-reductive.
We will see however that in the geometric setting, the de\,Rham--Betti group associated with a smooth projective variety over $K$, or more generally with an Andr\'e motive, is expected from the Grothendieck Period Conjecture to be reductive.

Moreover, as we will sometimes assume that an object $M \in \mathscr C_{\dRB, K_\dR, L_\B}$ and its base-change $M\otimes_LL'$ for some $L'/L$ are both semi-simple, we note that indeed the semi-simplicity of $M$ and of $M\otimes_LL'$ are in general not correlated.
For instance, the de\,Rham--Betti object 
\[
M:=\left(M_\dR=\overline \Q^{\oplus 2}, M_\B = \Q^{\oplus 2}, c_M\right), \quad \text{with }c_M = \begin{pmatrix}
\alpha & \beta \\ i\alpha &\gamma
\end{pmatrix},
\quad \mathrm{degtr}_\Q \Q(\alpha,\beta, \gamma) = 3.
\] is simple but its base-change $M\otimes \overline \Q$ is not semi-simple.
On the other hand, the de\,Rham--Betti object
\[
M:=\left(M_\dR=\overline \Q^{\oplus 2}, M_\B = \Q^{\oplus 2}, c_M\right), \quad \text{with }c_M = \begin{pmatrix} 1 & i\pi \\ 0 &\pi \end{pmatrix}.
\] 
is not semi-simple but its base-change $M\otimes \overline \Q$ is semi-simple.

\section{The de\,Rham--Betti torsor of periods and change of coefficients}\label{S:dRB-barQ}

 We assume $L\subseteq K$. Let $M = (M_\dR,M_\B,c_M)$ be a de\,Rham--Betti object over $K$ with $L$-coefficients, that is, an object of the neutral $L$-linear Tannakian category $\mathscr{C}_{\dRB, K, L}$.
 Given a field extension $L\subseteq L' \subseteq K$, the base-change functor 
given by $$M \mapsto M\otimes_L L' =_{\mathrm{def}} (M_\dR,M_\B\otimes_L L',c_M)$$
 provides by \cite[Prop.~2.21(b)]{DeligneMilne} a
closed embedding $\Omega_{M\otimes_L L'}^{\dRB}   \subseteq    \Omega_M^{\dRB}$, and we have
a chain of natural closed inclusions
\begin{equation*}
Z_M \ \subseteq \ \Omega_M \ \subseteq \ \Omega_{M\otimes_L L'}^{\dRB} \  \subseteq  \  \Omega_M^{\dRB} \ \subseteq  \
\mathrm{Iso}_{K}(M_{\dR}, M_{\B}\otimes_L K),
\end{equation*} 
where $Z_M$ is the $K$-Zariski closure of $c_M$ 
and $\Omega_M$ is the smallest $K$-subtorsor containing $c_M$, 
i.e., the intersection of all $K$-subtorsor containing $c_M$.
Here, by a \emph{$K$-subtorsor} of $ \mathrm{Iso}_{K}(M_{\dR}, M_{\B}\otimes_L K)$ we mean a closed $K$-subscheme $V $ 
such that $f\circ g^{-1}\circ h$ belongs to $V$ for all points $f,g,h \in V$\,;
 it is naturally a torsor under a closed $K$-algebraic subgroup $H$ of $\mathrm{GL}(M_{\B}\otimes_L K)$.
Since the complex-valued point $c_M$ agrees with the complex-valued point $c_{M\otimes_L L'}$, we have equalities $Z_M = Z_{M\otimes_L L'}$ of closed $K$-subschemes and  $\Omega_M = \Omega_{M\otimes_L L'}$ of $K$-subtorsors of  $\mathrm{Iso}_{K}(M_{\dR}, M_{\B}\otimes_L K)$. We also note that for a field extension $K\subseteq K'\subseteq \overline{\Q}$,  $Z_{M_{K'}}$ identifies with an irreducible component of $(Z_M)_{K'}$ and that there are natural closed embeddings $\Omega_{M_{K'}} \hookrightarrow (\Omega_M)_{K'}$ and $\Omega^\dRB_{M_{K'}} \hookrightarrow (\Omega_M^\dRB)_{K'}$. 
For $L'\subseteq K'\subseteq \overline \Q$ respective field extensions of $L\subseteq K \subseteq \overline{\Q}$, we therefore have chains of closed embeddings as pictured in the 3 left columns of \eqref{eq:base-change} below.
\medskip

Assuming further that $L' = K= \overline{\Q}$, we have for a de\,Rham--Betti object $M$ over $\overline{\Q}$ with $L$-coefficients a chain of natural closed inclusions
\begin{equation} \label{E:torsor-inclusions}
Z_M \subseteq \Omega_M \subseteq \Omega_{M\otimes_L \overline \Q}^{\dRB} \subseteq \Omega_M^\dRB.
\end{equation}
In this section, we show that the middle inclusion in \eqref{E:torsor-inclusions} is an equality, while the left-most and right-most inclusions are not equalities in general.
The observation that the inclusion $ \Omega_{M\otimes_\Q \overline \Q}^{\dRB} \subseteq \Omega_M^\dRB$ is not an equality for a general de\,Rham--Betti object over $\overline \Q$ with $\Q$-coefficients is the motivation(!) for considering de\,Rham--Betti classes with coefficients, as well as Andr\'e motives with coefficients and their de\,Rham--Betti realizations.

\subsection{On the inclusion 	$\Omega_M \subseteq  \Omega^{\dRB}_{M\otimes {\overline \Q}}$}
The middle  inclusion of \eqref{E:torsor-inclusions} is an equality\,:

\begin{prop}	\label{P:barQ-torsor}
	For any de\,Rham--Betti object~$M\in \mathscr C_{\dRB,\overline{\Q}_\dR,L_\B}$, we have
	$$\Omega_M = \Omega^{\dRB}_{M\otimes_L {\overline \Q}}.$$
\end{prop}
\begin{proof} Since $\Omega_M = \Omega_{M\otimes_L \overline{\Q}}$, we may assume $L=\overline{\Q}$.
	From the inclusion $\Omega_M \subseteq \Omega^{\dRB}_{M}$,
	it follows that $\Omega_M$ is a torsor under a $\overline \Q$-subgroup $ G \subseteq G_{\dRB}(M)$.
	The group $G \subseteq \operatorname{GL}(M_{\B})$ is thus the stabilizer of a $\overline \Q$-line $L_{\B}$ in some tensor space
	$\bigoplus_{\mathrm{finite}} M_{\B}^{\otimes n_i} \otimes M_{\B}^{\vee \otimes m_i}$.
	We will prove that $G_{\dRB}(M)$ stabilizes the line $L_{\B}$. 
	This will show that $G = G_{\dRB}(M)$ and finish the proof that the two torsors are equal.
	If we choose a point $x \in \Omega_M (\overline \Q)$, then $c_M \circ x^{-1} \in G(\mathbb{C})$.
	As $G(\C)$ stabilizes $L_{\B} \otimes \C$ inside $\bigoplus_{\mathrm{finite}} M_{\B}^{\otimes n_i} \otimes M_{\B}^{\vee \otimes m_i}$, this means that
	$$ c_M \circ x^{-1}(L_{\B} \otimes \C) = L_{\B} \otimes \C.$$
	Since $\Omega_M \subseteq \mathrm{Iso}_{\overline \Q}(M_{\dR}, M_{\B})$, the line $L_{\dR}:= x^{-1}(L_B) \subseteq \bigoplus_{\mathrm{finite}}  M_{\dR}^{\otimes n_i} \otimes M_{\dR}^{\vee \otimes m_i}$ is defined over $\overline{\Q}$, 
	and
	$c_M(L_{\dR}\otimes \C) = L_{\B} \otimes \C$.
	Hence $(L_{\dR}, L_{\B}, c_M|_{L_{\dR}\otimes \C})$ is a subobject of $\bigoplus_{\mathrm{finite}} M^{\otimes n_i} \otimes M^{\vee \otimes m_i}$ in~$\mathscr{C}_{\dRB,\overline{\Q}_\dR, \overline{\Q}_\B}$.
	We conclude that $G_{\dRB}(N) $ stabilizes $L_{\B}$. 
\end{proof}

\begin{rmk}
 Proposition~\ref{P:barQ-torsor} does not hold in general without assuming that $K=\overline{\Q}$. For an example of a de\,Rham--Betti object $M \in \mathscr C_{\dRB,K_\dR,\Q_\B}$ such that $\Omega_M \subsetneq \Omega^\dRB_{M\otimes K}$, see
 Remark~\ref{rmk:torsor_notalgclosed} below.
\end{rmk}

\subsection{On the inclusion $ \Omega_{M\otimes_L L'}^{\dRB} \subseteq \Omega_M^\dRB$}
Consider field extensions $L\subseteq L' \subseteq K \subseteq \overline \Q$. 
The inclusion $ \Omega_{M\otimes_L L'}^{\dRB} \subseteq \Omega_M^\dRB$ is an equality if and only if the functor $\langle M \rangle \otimes_L L'\to \langle M\otimes_L L' \rangle$ is an equivalence of Tannakian categories. 
  This is related to the question of whether $L'$-de\,Rham--Betti classes on $M$ are $L'$-linear combinations of de\,Rham--Betti classes on $M$, 
  but also to the question of whether the $\overline \Q$-torsor $ \Omega_{M\otimes_L L'}^{\dRB}$ is a torsor under a $\overline \Q$-group defined over $L$\,:

\begin{prop}\label{P:barQ-dRB}
	Let $M$ be a semi-simple de\,Rham--Betti object over $K$
	with $L$-coefficients. The following statements are equivalent.
	\begin{enumerate}[(i)]
\item 	The inclusion  $ \Omega_{M\otimes_L L'}^{\dRB} \subseteq \Omega_M^\dRB$ is an equality.
\item $M\otimes_LL'$ is a semi-simple de\,Rham--Betti object and, for any $N \in \langle M\rangle$, any  $L'$-de\,Rham--Betti class on $N$ is an  $L'$-linear combination of de\,Rham--Betti classes on $N$.
\item $G_{\dRB}(M\otimes_LL')$ is reductive and defines an $L$-subgroup of $G_\dRB(M)$.
	\end{enumerate}
\end{prop}
\begin{proof} 
 Since $M$ is semi-simple,
	 the de\,Rham--Betti group $G_{\dRB}(M)$ is reductive.
By definition, the torsor  $\Omega_{M\otimes_LL'}^{\dRB}$ is included in the intersection of the torsors $\Omega_{\alpha}$ whose $\overline \Q$-points
	are  given by 
	\begin{equation}\label{eq:Omega_alpha}
\Omega_\alpha(\overline \Q) = \{ f\in \operatorname{Iso}_{\overline \Q}(M_\dR\otimes_K \overline \Q, M_\B \otimes_{L} \overline \Q) \ | \ f(\alpha_\dR\otimes_K 1_{\overline \Q}) = \alpha_\B \otimes_{L'} 1_{\overline \Q} \},
	\end{equation}
	for $\alpha = (\alpha_{\dR},\alpha_{\B})$ running through the de\,Rham--Betti classes in the $L'$-base change of  the various tensor spaces $M^{\otimes n} \otimes (M^\vee)^{\otimes m}$, and this inclusion is an equality if $M\otimes_L L'$ is semi-simple as a de\,Rham--Betti object with $L'$-coefficients. 
	(Under this semi-simplicity assumption, one recovers the definition of 
	the de\,Rham--Betti torsor of periods $\Omega_M^\dRB$ in terms of invariants given in  \cite[Def.~2.4]{BostCharles}.)
	On the other hand, by the semi-simplicity of $M$,
	the torsor $\Omega_M^\dRB$ is the intersection of the torsors $\Omega_{\alpha}$ for $\alpha$ running through the de\,Rham--Betti classes in the  various tensor spaces $M^{\otimes n} \otimes (M^\vee)^{\otimes m}$.
	 This establishes the equivalence of $(i)$ and $(ii)$. The implication $(i)\Rightarrow (iii)$ is clear\,; indeed $(i)$ is equivalent to the inclusion $G_{\dRB}(M\otimes_LL') \subseteq G_{\dRB}(M)_{L'}$ being an equality.
	 Finally, assume both groups $G_{\dRB}(M\otimes_LL')$ and $G_{\dRB}(M)$ are reductive and assume that $G_{\dRB}(M\otimes_LL') = G_{L'}$ 
	 where $G\subseteq G_{\dRB}(M)$ is an $L$-subgroup. 
	 Then, for  any element $\alpha$ in a tensor space $T(M_\B)$ of $M_\B$ such that $\alpha$ is $G$-invariant, 
	 the class $\alpha\otimes_L 1_{L'}$, being $G_{L'}$-invariant, defines a de\,Rham--Betti class on $T(M\otimes_L L')$. 
	 In particular $\alpha$ is a de\,Rham--Betti class on $T(M)$ and hence is $G_{\dRB}(M)$-invariant. 
	 Therefore the reductive groups $G$ and $G_{\dRB}(M)$ share the same invariants and thus coincide. 
	 This establishes the implication $(iii) \Rightarrow (i)$.
\end{proof}

\begin{ex}[The inclusion $ \Omega_{M\otimes_LL'}^{\dRB} \subseteq \Omega_M^\dRB$ can be strict]
	\label{Ex:dRB-Qbar}
Consider the de\,Rham--Betti object
\[
M:=\left(M_\dR=\overline{\Q}^{\oplus 2}, M_\B = \Q^{\oplus 2}, c_M\right), \quad \text{with }c_M = \begin{pmatrix}
\pi & a \\ b\pi &c
\end{pmatrix},
\quad a,b,c\in \overline{\Q}, \; c-ab\neq 0.
\]
For a general choice of $a$, $b$ and $c$ in $\overline \Q$,  $M\otimes \overline \Q$ has a nonzero de\,Rham--Betti class as we have $M\otimes\overline{\Q} \simeq \mathds{1}(-1)\oplus \mathds{1}$, while 
$M$ is simple and in particular does not have any nonzero de\, Rham--Betti class. 
In view of Proposition~\ref{P:barQ-dRB}, the inclusion $ \Omega_{M\otimes \overline \Q}^{\dRB} \subseteq \Omega_M^\dRB$ is then strict.
(In relation to \cite[Rmk.~2.6]{BostCharles}, one also notes that the Zariski closure $Z_M$ of $c_M$ is a torsor under $G_{\dRB}(M_{\overline{\Q}}) = \mathds{G}_{m,\overline{\Q}}$ but that the latter does not descend to a $\Q$-subgroup of $\mathrm{GL}_2$.)
\end{ex}

\subsection{On the inclusion $Z_M \subseteq \Omega_M$}
For an object $N \in \mathscr{C}_{\dRB, K_\dR, K_\B}$,
 the comparison isomorphism $c_N: N_{\dR} \otimes_{K} \C \stackrel{\simeq}{\longrightarrow} N_{\B} \otimes_{K} \C $ defines a $K$-bilinear map, called the \emph{period pairing},
$$\int \colon N_\B^\vee \otimes_{K} N_{\dR} \to \mathbb{C}, \,\,\, \gamma \otimes \omega \mapsto \int_{\gamma} \omega =_{\mathrm{def}} \gamma_{\C}(c_N(\omega_\C)).$$
Fixing $\gamma \in N_{\B}^{\vee}$ and $\omega \in N_{\dR}$, respectively, we thereby get $K$-linear maps
$$\int_{\gamma}\colon N_{\dR} \to \mathbb{C}, \,\,\, \omega \mapsto \gamma_{\C}(c_N(\omega_\C)) \quad \mbox{and} \quad \int {\omega}\colon N_{\B}^{\vee} \to \mathbb{C}, \,\,\, \gamma \mapsto \gamma_{\C}(c_N(\omega_\C)).$$

\begin{defn}\label{D:ann}
	For $\gamma \in N_{\B}^{\vee}$ we define the \emph{annihilator} $\mathrm{Ann}(\gamma) =_{\mathrm{def}} \ker \int_{\gamma} \subseteq N_{\dR}$.
	Similarly, for $\omega \in N_{\dR}$ we define $\mathrm{Ann}(\omega) =_{\mathrm{def}} \ker \int \omega\subseteq N_{\B}^{\vee}$.
\end{defn}

The following proposition gives criteria for the inclusion $Z_N \subseteq \Omega_N$ to be an equality in case $K=\overline{\Q}$. (Criteria similar to $(iii)$ and $(iv)$ appear in \cite{Hoermann, HW}).
We leave it to the reader to state and prove a similar statement for the inclusion $Z_M \subseteq \Omega^\dRB_M$ to be an equality for a de\,Rham--Betti object~$M \in \mathcal{C}_{\dRB, \overline \Q_{\dR}, \Q_{\B}}$.

\begin{prop}
	Let $M \in  \mathscr{C}_{\dRB, \overline{\Q}_\dR, \overline{\Q}_\B}$. The following statements are equivalent\,:
	\begin{enumerate}[(i)]
		\item
		$Z_M \subseteq \mathrm{Iso}_{\overline{\Q}}(M_{\dR}, M_{\B} )$ is a torsor\,;
		\item $Z_M = \Omega_M$\,;
		\item
		For every $N \in \langle M  \rangle \subset  \mathscr{C}_{\dRB, \overline{\Q}_\dR, \overline{\Q}_\B}$ and every $\omega \in N_\dR$, there exists a short exact sequence
		$$0 \to N' \to N \to N'' \to 0$$
		in $ \mathscr{C}_{\dRB, \overline{\Q}_\dR, \overline{\Q}_\B}$
		such that $\omega \in N'_{\dR}$ and $\mathrm{Ann}(\omega) = (N''_\B)^{\vee}$\,;
		\item
		For every $N \in \langle M   \rangle \subset  \mathscr{C}_{\dRB, \overline{\Q}_\dR, \overline{\Q}_\B}$ and every $\gamma \in N_\B^{\vee}$, there exists a short exact sequence
		$$0 \to N' \to N\to N'' \to 0$$
		in $ \mathscr{C}_{\dRB, \overline{\Q}_\dR, \overline{\Q}_\B}$
		such that $\gamma \in (N''_{\B})^{\vee}$ and $\mathrm{Ann}(\gamma) = N'_\dR$.
	\end{enumerate}
\end{prop}
\begin{proof}
	The equivalence $(i) \iff (ii)$ follows from the definition of $\Omega_M$.
	We first introduce another condition $(ii')$ which is equivalent to $(ii)$.
The direct sum of the period pairings associated with all objects  $N \in \langle M   \rangle$ provides
	 a pairing $$p: \bigoplus_{N \in \langle M   \rangle } N_\B^{\vee} \otimes N_{\dR} \to \C.$$
	We define the \emph{ring of periods} $\mathcal{P}(M) \subseteq \C$ of $M$ to be the image of $p$.
	Note that $G_{\dRB}(M)$ naturally acts on all $N_\B^{\vee}$ and hence on $\bigoplus_{N \in \langle M   \rangle } N_\B^{\vee} \otimes N_{\dR}$.
	
	\begin{claim}
		Statement $(ii)$ is equivalent to the assertion
		\begin{enumerate}[(i')] \setcounter{enumi}{1}
			\item $\ker(p) \subseteq \bigoplus_{N \in \langle M   \rangle } N_\B^{\vee} \otimes N_{\dR}$ is stable under the action of $G_{\dRB}(M)$.
			
	\end{enumerate} \end{claim}
	\begin{proof}[Proof of Claim]
		By Proposition \ref{P:barQ-torsor}, the torsor $\Omega_M = \Omega_{M  }^{\dRB}$ is the Tannakian torsor for $\langle M   \rangle$.
		It then follows from the construction (cf.~\cite[proof of Thm.~3.2]{DeligneMilne}) that $\Omega_M = \Spec R$ is the spectrum of an algebra $R$ which fits into the diagram
		$$\xymatrix{\bigoplus_{N \in \langle M   \rangle } N_\B^{\vee} \otimes N_{\dR} \ar@{->>}[r] \ar@{->>}[rd]_{p} & R \ar@{->>}[d] \\
			&\qquad  \mathcal{P}(M) \subset \C.}$$
		Note that $Z_M = \Spec \mathcal{P}(M)$.
		Since $\Omega_M$ is a torsor under $G_{\dRB}(M)$, the equality $Z_M = \Omega_M$ holds exactly if the action of $G_{\dRB}(M)$ on $\bigoplus_{N \in \langle M   \rangle } N_\B^{\vee} \otimes N_{\dR}$ passes down to an action of $G_{\dRB}(M)$ on the quotient $\mathcal{P}(M)$.
		This is equivalent to saying that the kernel $\ker(p) \subseteq \bigoplus_{N \in \langle M   \rangle } N_\B^{\vee} \otimes N_{\dR} $ is stable under $G_{\dRB}(M)$.
	\end{proof}

	We prove the equivalence $(ii') \iff (iii)$.
	Suppose $(ii')$ holds true and let $\omega \in N_\dR$.
	Note that $\gamma \in \mathrm{Ann}(\omega) $ if and only if $\gamma \otimes \omega \in \ker(p)$.
	It follows that $\mathrm{Ann}(\omega) \subseteq N_B^{\vee} $ is stable under the action of $G_{\dRB}(M)$ and is therefore the realization of a subobject in $ \mathscr{C}_{\dRB, \overline{\Q}_\dR, \overline{\Q}_\B}$. 
	The dual object gives the desired $N''$.
	Conversely, suppose that $(iii)$ is true.
	An element of $ \bigoplus_{N \in \langle M   \rangle } N_\B^{\vee} \otimes N_{\dR}$ is given by a collection $\left(\sum_{i=1}^{m_N} \gamma_{i,N} \otimes\omega_{i,N} \right)_N$, where only finitely many $N_1, ..., N_n$ contribute.
	Consider the object $$\tilde{N} := \bigoplus_{k=1}^n N_k^{\oplus m_{N_k}} \in  \mathscr{C}_{\dRB, \overline{\Q}_\dR, \overline{\Q}_\B}.$$
	Let $\omega := (\omega_{1,N_k}, ..., \omega_{m_{N_k},N_k} )_k \in \tilde{N}_{\dR}$.
	Then $\left(\sum_{i=1}^{m_N} \gamma_{i,N} \otimes\omega_{i,N} \right)_N$ lies in $ \ker(p)$ if and only if
	$\gamma:= (\gamma_{1,N_k}, ..., \gamma_{m_{N_k},N_k})_k \in \tilde{N}_{\B}^{\vee}$ lies in $\mathrm{Ann}(\omega)$.
	Since $\mathrm{Ann}(\omega)$ is stable under $G_{\dRB}(M)$ by assumption, we see that $\ker(p)$ is stable under $G_{\dRB}(M)$.
	
	The situation is symmetric in the fiber functors $\omega_\dR$ and $\omega_\B$, and therefore a similar proof shows $(ii') \iff (iv)$.
	\end{proof}

\begin{ex}[The inclusion $Z_M \subseteq \Omega_M$ can be strict] \label{ex:Z-Omega}
	The object
	\[
	M:=\left(M_\dR=\overline \Q^{\oplus 2}, M_\B = \overline \Q^{\oplus 2}, c_M\right), \quad \text{with }c_M = \begin{pmatrix}
	\alpha & \beta \\ a &\gamma
	\end{pmatrix},
	\quad \mathrm{degtr}_\Q \Q(\alpha,\beta, \gamma) = 3, \quad a \in \overline \Q \setminus \{0\}
	\]
	defines a simple de\,Rham--Betti object with $\overline \Q$-coefficients such that $Z_M$ is not a torsor.
 Indeed, $\dim Z_M=3$ and $\Omega_M$ is a torsor under a connected and reductive subgroup of $\mathrm{GL}_{2,\overline \Q}$ of dimension $\geq3$ and hence must be a torsor under $\mathrm{GL}_{2,\overline \Q}$.
\end{ex}

\section{The Grothendieck period conjecture and the de\,Rham--Betti conjecture}\label{S:GPC}

\subsection{ Andr\'e motives with coefficients and their de\,Rham--Betti realization}
\label{SS:dRB-Andre}

To a smooth projective variety $X$ defined over $K$, one associates its \emph{de\,Rham--Betti cohomology groups}
$$\HH^n_\dRB(X,\Q(k)) =_{\mathrm{def}} \big(\HH^n_{\dR}(X/K), \HH_{\B}^n(X_\C^\an,\Q(k)),  c_X\big),$$
where 
$c_X : \HH^{n}_\dR(X/K) \otimes_K \C \stackrel{\simeq}{\longrightarrow} \HH^{n}_\B(X_\C^\an,\Q)\otimes_\Q \C$ is Grothendieck's period comparison isomorphism~\eqref{E:comparison}. These are objects in $\mathscr{C}_{\dRB,K_\dR,\Q_\B}$.

Let  $\mathsf{M}^\And_K$ be the category of Andr\'e motives over $K$. By  \cite[\S 4]{AndreIHES},
the category  $\mathsf{M}^\And_K$ is graded Tannakian semi-simple over~$\Q$, neutralized by the fiber functor given by the Betti cohomology realization functor 
$$\omega_B :  \mathsf{M}_K^\And \to \mathrm{Vec}_\Q,\quad M:=(X,p,n) \mapsto \HH_{\B}^*(M) := p_*\HH_\B^*(X^\an_\C,\Q(n)).$$
Under the de\,Rham realization functor $\omega_{\dR} : \mathsf{M}_K^\And \to \mathrm{Vec}_K$, motivated cycles are mapped to classes lying in $F^0$ for the Hodge filtration. In addition these are compatible with the canonical comparison isomorphisms, so that motivated cycles are \emph{absolute Hodge} in the sense of Deligne~\cite{Deligne} and in particular they are also \emph{de\,Rham--Betti}.
 As such, there is a well-defined faithful realization functor 
$$\rho_{\dRB} : \mathsf M_K^\And \to \mathscr C_{\dRB,K,\Q}, \quad
M=(X,p,n) \mapsto (p_*\HH_\dR^*(X/K), p_*\HH^*_\B(X^\an_\C,\Q(n)),p_*\circ c_X \circ p_*),$$ and we may speak of de\,Rham--Betti classes on an Andr\'e motive $M$.
\medskip

More generally, 
	let $\mathsf{M}^{\And}_{K} \otimes L$ be the $L$-linear category of Andr\'e motives over $K$ with $L$-coefficients\,; 
it is the pseudo-abelian envelope of the base-change to $L$ of~$\mathsf{M}^{\And}_{K}$. 
An object  of $\mathsf{M}^{\And}_{K} \otimes L$ is a triple $(X,p,n)$ with $X$ a smooth projective variety over $K$, $p$ a motivated  idempotent correspondence with $L$-coefficients, and $n\in \Z$.
The category  $\mathsf{M}^{\And}_{K} \otimes L$ is graded Tannakian semi-simple over~$L$, neutralized by the fiber functor given by the Betti cohomology realization functor 
$$\omega_B :  \mathsf{M}_K^\And \otimes L \to \mathrm{Vec}_L,\quad M:=(X,p,n) \mapsto \HH_{\B}^*(M) := p_*\HH_\B^*(X^\an_\C,L(n)).$$

Given field extensions $K\subseteq K'$ and $L \subseteq L'$, 
the base-change of $M$ to $\mathsf{M}^{\And}_{K'} \otimes L'$ is denoted by $M_{K'}\otimes_L L'$.
Motivated classes on $M\otimes_L L'$, also called \emph{$L'$-motivated classes} on $M$, are by definition $L'$-linear combinations of motivated classes on $M$.

We make the basic but important observation that, if $L \nsubseteq K$, although one may consider 
$\HH^j_{\dRB}(X,\Q(k)) \otimes L$ as an object in $\mathscr C_{\dRB,K,L}$ for $X$ smooth projective over $K$ and although there is indeed a Betti realization functor $\omega_{\B}\colon \mathsf{M}^{\And}_{K} \otimes L \to\mathrm{Vec}_L$,
there is a priori no linear functor $\mathsf{M}^{\And}_{K} \otimes L \to \mathrm{Vec}_K$ 
and hence no de\,Rham--Betti realization functor $\mathsf{M}^{\And}_{K} \otimes L \to \mathscr C_{\dRB,K,L}$.
For instance, if $E$ is an elliptic curve over $\Q$ with CM by a field~$L$, 
then $\h^1(E)\otimes L$ splits as the direct sum of two non-trivial motives, whereas the $K$-vector space $\HH^1_\dR(E)$ does not admit any non-trivial motivic splitting.

However, if $L\subseteq K$, we do have a well-defined de\,Rham--Betti realization functor
$$\rho_{\dRB} \colon 
\mathsf M_K^\And \otimes L \to \mathscr C_{\dRB,K,L} $$
defined as above.
Given an Andr\'e motive $M \in M_K^\And \otimes L  $, we also denote abusively by $M$ its de\,Rham--Betti realization in $\mathscr C_{\dRB,K,L} $.

Note from  Example~\ref{Ex:dRB-Qbar}  that it is a priori not clear that $L'$-de\,Rham--Betti classes on $M$ are $L'$-linear combinations of de\,Rham--Betti classes on $M$.
Nonetheless, we have the following fullness conjecture\,:
\begin{conj}\label{conj:dRB-mot}
	The realization functor $\rho_{\dRB} : \mathsf M_K^\And\otimes L \to \mathscr C_{\dRB,K,L}$ is full.
	In other words, de\,Rham--Betti classes on Andr\'e motives over $K$ with $L$-coefficients are motivated.
\end{conj}
As we will see in Proposition~\ref{P:torsor-fullyfaithful}, Conjecture~\ref{conj:dRB-mot} is a particular instance
of the Grothendieck Period Conjecture~\ref{C:GPC} below. 
For now,
the following easy (non-Tannakian) lemma reduces, in particular, Conjecture~\ref{conj:dRB-mot} to the case $K=\overline \Q$ (see also Lemma~\ref{L:torsor-descent}) and in fact to the case $L=K=\overline \Q$\,:

\begin{lem}\label{L:descent}
	Let $L\subseteq K\subseteq \overline{\Q}$ be any fields, and let $M\in \mathsf M_K^\And\otimes L $ be an Andr\'e motive over~$K$ with $L$-coefficients. Let $K\subseteq K'\subseteq \overline{\Q}$ and $L\subseteq L'\subseteq L'' \subseteq \overline{\Q}$ be any field extensions. 
	If every $L''$-de\,Rham--Betti class on $M_{K'}$ is $L''$-motivated, then  every $L'$-de\,Rham--Betti class on $M$ is $L'$-motivated.
\end{lem}
\begin{proof}
	First it is clear that 	if every $L''$-de\,Rham--Betti class on $M_{K'}$ is $L''$-motivated, then  every $L'$-de\,Rham--Betti class on $M_{K'}$ is $L'$-motivated.
    (Note that $L''$-de\,Rham--Betti or $L''$-motivated classes on $M$ make sense for any $L\subseteq L'' \subseteq \overline \Q$ and that we are not requiring that $L''\subseteq K'$).
	Second consider a $L'$-de\,Rham--Betti class  $\alpha = (\alpha_{\dR}, \alpha_{\B})$  on $M$. 
	Its base-change to $M_{K'}$ is $L'$-motivated by assumption and so its base-change to $M_{\overline \Q}$ is also $L'$-motivated, i.e. the class of an $L'$-motivated cycle $z$ on $M_{\overline \Q}$. 
	We have to show that $z$ is defined over $K$. 
	Recall from \cite[Scolie~p.17]{AndreIHES} that  the Galois group $\Gal(\overline \Q/K)$ acts naturally on the space of motivated cycles on~$M_{\overline \Q}$ and
	the space of motivated cycles on $M$ is exactly the space of $\Gal(\overline \Q/K)$-invariant motivated cycles on~$M_{\overline \Q}$.
	Since $L'$-motivated cycles are simply $L'$-linear combinations of motivated cycles, the above also holds for $L'$-motivated cycles.
	 Let $g\in \Gal(\overline \Q /K)$, then the $L'$-de\,Rham--Betti class associated with $g(z)$ has de\,Rham component equal to $g(\alpha_{\dR}) = \alpha_{\dR}$. 
	 This implies that the cohomology class of $g(z)$ is the same as that of $z$ and so $g(z)=z$.
\end{proof}

\subsection{The Grothendieck period conjecture}\label{SS:GPC}
We generalize the Tannakian treatments of  the Grothendieck Period Conjecture given in \cite[\S 7.5]{AndreBook} and \cite[\S 2.2.2]{BostCharles} to motives with coefficients.

\begin{defn}[Motivated Galois group and motivated torsor of periods]
	Let  $M \in \mathsf{M}^\And_{K}\otimes L$ be an Andr\'e motive over $K$ with $L$-coefficients.
	The \emph{motivated Galois group} $G_{\And}(M)$ of $M$ is the Tannakian fundamental group
	$$G_{\And}(M) =_{\mathrm{def}} \operatorname{Aut}^\otimes(\omega_{\B}|_{\langle M\rangle}).$$
	If $L\subseteq K$, the \emph{motivated torsor of periods} of $M$ is the Tannakian torsor
	$$\Omega^\And_M =_{\mathrm{def}} \operatorname{Iso}^{\otimes}(\omega_{\dR}|_{\langle M\rangle}, \omega_{\B}|_{\langle M\rangle} \otimes_{L} K)\,;$$
	 it is a torsor under $\operatorname{Aut}^\otimes(\omega_{\B}|_{\langle M\rangle}\otimes_L K)$, which coincides with  $G_{\And}(M)_{K}$ by \cite[Rmk.~3.12]{DeligneMilne}.
\end{defn}

Since the neutral Tannakian category $\mathsf{M}^\And_K\otimes L$ is semi-simple, the motivated Galois group $G_{\And}(M)$ is reductive. 
In addition, $G_{\And}(M)$ is the closed subgroup of $\mathrm{GL}(\omega_B(M))$ that fixes motivated classes inside tensor spaces $ \bigoplus_{\mathrm{finite}} M^{\otimes n_i}\otimes (M^\vee)^{\otimes m_i}$.
Thus, for a field extension $L\subseteq L'$, since $L'$-motivated classes are by definition $L'$-linear combinations of motivated classes, the natural closed embedding 
$G_\And(M\otimes_L L') \subseteq (G_\And(M))_{L'}$ is an equality.
Likewise, if $L\subseteq K$, the motivated torsor of periods $\Omega^\And_M$ has the following description in terms of invariants\,: it is the intersection of the $\Omega_{\alpha}$ as in~\eqref{eq:Omega_alpha}, where now $\alpha$ runs through the motivated classes on tensor spaces $M^{\otimes n} \otimes (M^\vee)^{\otimes m}$.  Assuming $L\subseteq L'\subseteq K$, the natural closed embedding $\Omega^\And_{M\otimes_L L'} \subseteq \Omega^\And_M$ is an equality.
\medskip

 A homological motive $M$ over $K$ with $L$-coefficients is an object of the form $M=(X,p,n)$ with $X$ smooth projective over $K$ of pure dimension $d_X$, $p$ an idempotent in $\operatorname{im}(\CH^{d_X}(X\times X)\otimes L \to \HH^{2d_X}_B((X\times X)^\an_\C,L(d_X))$ and $n$ an integer.
 
 \begin{defn}[Motivic torsor of periods]
 	\label{D:torsormotperiods}
Assume $L\subseteq K$.
The \emph{motivic torsor of periods} $\Omega^\mot_M$ of a homological motive $M$ over $K$ with $L$-coefficients is defined as the intersection of the $\Omega_{\alpha}$ as in~\eqref{eq:Omega_alpha}, where $\alpha$ runs through the algebraic classes on tensor spaces $M^{\otimes n} \otimes (M^\vee)^{\otimes m}$. 
 \end{defn}
We note that this torsor has a Tannakian description in case $X$ satisfies  Grothendieck's standard conjectures\,; see \cite[\S 7.5.2]{AndreBook}.
 \medskip

Recall that given an Andr\'e motive $M$ we abusively denote by $M$ its de\,Rham--Betti realization, so that $\Omega_M^\dRB$ means $\Omega^\dRB_{\rho_{\dRB}(M)}$.
From the general theory of neutral Tannakian categories,
 or more simply
from the descriptions above, we have closed immersions $\Omega_M^\dRB \subseteq \Omega^\And_M$ for $M$ an Andr\'e motive over $K$ with $L$-coefficients and  $\Omega^\And_M \subseteq \Omega^\mot_M$ for $M$ a homological motive over $K$ with $L$-coefficients.

\begin{conj}[Grothendieck Period Conjecture]\label{C:GPC}
	Let $L\subseteq K \subseteq \overline \Q$ be fields.
\begin{enumerate}[(i)]
\item Let $M $ be an Andr\'e motive over~$K$ with $L$-coefficients.
We say $M$ satisfies the \emph{motivated version of the Grothendieck period conjecture} if the inclusions 
$$Z_M \subseteq \Omega_M \subseteq \Omega^\dRB_M \subseteq \Omega^\And_M $$ 
are equalities.
\item Let $M$ be a homological motive over~$K$  with $L$-coefficients. 
 We say  $M$ satisfies the \emph{Grothendieck period conjecture} if the inclusions 
$$Z_M \subseteq \Omega_M \subseteq  \Omega^\dRB_M \subseteq \Omega^\And_M \subseteq \Omega^\mot_M $$ are equalities.
\end{enumerate}
\end{conj}

Note that Conjecture~\ref{C:GPC} predicts that $\Omega_M^{\And}$ is connected for an Andr\'e motive $M$ over $K$  with $L$-coefficients and that $G_\And(M)$ is connected in case $K=\overline \Q$.
Note also that, for a field extension $L\subseteq L'\subseteq K$, since $Z_M = Z_{M\otimes_LL'}$ 
 and   $\Omega^\And_M = \Omega^\And_{M\otimes_LL'}$, we have the equivalence $Z_M = \Omega_M^\And \Leftrightarrow Z_{M\otimes_LL'} = \Omega^\And_{M\otimes_LL'}$, and similarly for the motivic torsor in place of the motivated torsor.

\begin{ex}[Torsor of periods of Artin motives] 
	Let $L=\Q$ and let $K\subseteq F \subseteq \overline \Q$ be a finite extension of~$K$. 
	Consider the \emph{Artin motive} $M:=\h(\Spec F)$ over $K$. 
	In that case all three $K$-torsors introduced above agree and we have $\Omega_M = \Omega_M^\And = \Omega_M^\mot = \Spec F^g$ as $K$-torsors under the constant group scheme $\operatorname{Gal}(F^{g}/K)$, 
	where $F^g$ denotes the Galois closure of $F$ inside $\overline \Q$. Moreover $c_M \in \Omega_M(\C) = \Hom_K(F,\C)$ is the canonical element. As such, $M$ satisfies the Grothendieck Period Conjecture.
\end{ex}

\subsection{The de\,Rham--Betti Conjecture}
In this work, we will address the following special instance of  Conjecture~\ref{C:GPC}\,:

\begin{conj}[De\,Rham--Betti Conjecture]\label{C:dRB}
Let $L\subseteq K \subseteq \overline \Q$ be fields.
	\begin{enumerate}[(i)]
		\item Let $M$ be an Andr\'e motive over~$K$ with $L$-coefficients. 
		  We say that $M$ satisfies the \emph{motivated de\,Rham--Betti conjecture} if the inclusion
		$$\Omega^\dRB_M \subseteq \Omega^\And_M $$ 
		is an equality.
		\item Let $M$ be a homological motive over~$K$  with $L$-coefficients. 
		 We say that $M$ satisfies the \emph{de\,Rham--Betti conjecture} if the inclusions 
		$$ \Omega^\dRB_M \subseteq \Omega^\And_M \subseteq \Omega^\mot_M $$ are equalities.
	\end{enumerate}
\end{conj}

Conjecture~\ref{C:dRB} is analogous to the conjectural injectivity of (4.4) in \cite{Brown}.
The following proposition
shows that Conjecture~\ref{C:dRB}$(i)$ for an Andr\'e motive $M$ over $K$ is a strengthening of Conjecture~\ref{conj:dRB-mot} restricted to~$\langle M \rangle$.

\begin{prop}\label{P:torsor-fullyfaithful}
	 Let $M$ be an Andr\'e motive over~$K$ with $L$-coefficients. 
The following statements are equivalent\,:
	\begin{enumerate}[(i)]
		\item $M$ satisfies the motivated de\,Rham--Betti conjecture, i.e.,	$\Omega^\dRB_{M} = \Omega^\And_{M}$\,;
		\item The functor $(\rho_{\dRB})|_{\langle M\rangle} $ is full and $G_\dRB(M)$ is reductive.
	\end{enumerate}
	In particular, if $M$ satisfies the motivated de\,Rham--Betti conjecture, then any de\,Rham--Betti class on a tensor space $M^{\otimes n}\otimes (M^\vee)^{\otimes m}$ is motivated, and if $M=\h(X)$, then any de\,Rham--Betti class in $\HH_{\dRB}^j(X^n,\Q(k))$ is motivated and in particular zero if $j\neq 2k$.
\end{prop}
\begin{proof}
	The equivalence of $(i)$ and $(ii)$  is a direct consequence of basic facts concerning Tannakian categories 
	and the fact that  $G_\And(M)$ is reductive. 
	Assume now that  $(\rho_{\dRB})|_{\langle M \rangle} $ is full. Then
	any de\,Rham--Betti classes on  a tensor space $ M^{\otimes n}\otimes (M^\vee)^{\otimes m}$ is motivated.
	If now $M=\h(X)$, we note that $\mathds{1}(-1)$ is a direct summand of $\h(X)$, so that $\h(X^n)(k)$ is a direct summand of $\h(X)^{\otimes r} \otimes (\h(X)^\vee)^{\otimes s}$ for some $r,s\geq 0$. Hence any de\,Rham--Betti class
	 on $\h^j(X^n)(k) \otimes L$
	 is motivated. 
	That de\,Rham--Betti classes
	 on $\h^j(X^n)(k) \otimes L$
	 are zero for $j\neq 2k$ follows at once from the fact that an Andr\'e motive with no graded piece of degree zero does not support any non-zero motivated class.
\end{proof}

\begin{rmk}
    The conditions of Proposition~\ref{P:torsor-fullyfaithful} are further equivalent to 
    \begin{enumerate}[(i)]
		\item[$(iii)$] The functor $(\rho_{\dRB})|_{\langle M\rangle} $ is full and $G_\dRB(M)$ is observable in $G_\And(M)$.
	\end{enumerate}
    As asked in \cite[Rmk~3.2]{Andre-observability}, it would be interesting to establish in general the observability of $G_\dRB(M)$ in $G_\And(M)$, as was done in the $\ell$-adic setting by Andr\'e in \cite[Thm.~2.1]{Andre-observability}.
\end{rmk}

The motivated de\,Rham--Betti conjecture~\ref{C:dRB}$(i)$ is stable under direct summand\,:

\begin{prop}\label{P:directfactor}
	Let $N$ be an Andr\'e motive over $K$ with $L$-coefficients, 
	and let $M$ be an object in $\langle N \rangle$. 
	If the motivated de\,Rham--Betti conjecture~\ref{C:dRB}$(i)$ holds for $N$, then it holds for $M$. In other words,
	$$ \Omega^\dRB_N = \Omega^\And_N \ \Longrightarrow \  \Omega^\dRB_M = \Omega^\And_M.$$
\end{prop}
\begin{proof}
	By Proposition~\ref{P:torsor-fullyfaithful}, $\Omega^\dRB_N = \Omega^\And_N$ if and only if $G_\dRB(N)$ is reductive and every de\,Rham--Betti class on tensor spaces $N^{\otimes n} \otimes (N^\vee)^{\otimes m}$ is motivated. Now if $M$ belongs to $\langle N \rangle$, then $G_\dRB(M)$ is a quotient of $G_\dRB(N)$, hence is reductive, and every de\,Rham--Betti class on tensor spaces $M^{\otimes n} \otimes (M^\vee)^{\otimes m}$ is motivated. By Proposition~\ref{P:torsor-fullyfaithful} again, we conclude that $\Omega^\dRB_M = \Omega^\And_M$.
\end{proof}

\subsection{Descent properties} 
\label{SS:descent}
Let $L' \subseteq K' \subseteq \overline \Q$ be respective fields extensions of $L\subseteq K \subseteq \overline \Q$, and let $M \in \mathsf M^\And_K \otimes L$ be an Andr\'e motive over $K$ with $L$-coefficients.
From \S \ref{S:dRB-barQ} and \S \ref{SS:GPC}, we have the following chains of closed immersions
\begin{equation}\label{eq:base-change}
\begin{tikzcd}
Z_{M_{K'}\otimes L'} \arrow[r, hook] \arrow[d, equal] & \Omega_{M_{K'}\otimes L'} \arrow[r, hook] \arrow[d, equal] & \Omega^\dRB_{M_{K'}\otimes L'} \arrow[r, hook] \arrow[d, hook] & \Omega^\And_{M_{K'}\otimes L'} \arrow[d, equal] \\
Z_{M_{K'}} \arrow[d, hook] \arrow[r, hook]                         & \Omega_{M_{K'}} \arrow[d, hook] \arrow[r, hook]                          & \Omega_{M_{K'}}^\dRB \arrow[d, hook] \arrow[r, hook]           & \Omega_{M_{K'}}^\And \arrow[d, hook]                          \\
(Z_M)_{K'} \arrow[r, hook]                                         & (\Omega_M)_{K'} \arrow[r, hook]                                               & (\Omega^\dRB_M)_{K'} \arrow[r, hook]                           & (\Omega^\And_M)_{K'}                                         
\end{tikzcd}
\end{equation}
where the inclusion $Z_{M_{K'}} \hookrightarrow (Z_M)_{K'}$ is the inclusion of an irreducible component.

\begin{prop}\label{P:GPC-LdRB}
Let $M\in M_K^\And\otimes L$ be an Andr\'e motive over $K$ with $L$-coefficients and let $L\subseteq L'\subseteq \overline \Q$ be a field extension.
If the Grothendieck Period Conjecture~\ref{C:GPC}$(i)$
holds for $M_{\overline \Q}$, i.e., if $Z_{M_{\overline \Q}} = \Omega^\And_{M_{\overline \Q}}$, 
then $L'$-de\,Rham--Betti classes on $M$ are $L'$-motivated, and hence $L'$-linear combinations of $\Q$-de\,Rham--Betti classes.
\end{prop}
\begin{proof}
Since $Z_{M_{\overline \Q}} = \Omega_{M_{\overline \Q}}^\And \Leftrightarrow Z_{M_{\overline \Q}\otimes_L L'} = \Omega_{M_{\overline \Q}\otimes_LL'}^\And$, 
we find by Proposition~\ref{P:torsor-fullyfaithful} that $L'$-de\,Rham--Betti classes on $M_{\overline \Q}$ are $L'$-motivated, and then by Lemma~\ref{L:descent} that 
$L'$-de\,Rham--Betti classes on $M$ are $L'$-motivated.
\end{proof}

The following lemma reduces the Grothendieck Period Conjecture or the motivated de\,Rham--Betti conjecture for $M$ to that for $M_{\overline \Q}$, and even to that for $M_{\overline \Q}\otimes_L \overline \Q$ (and a similar statement holds for homological motives and motivic torsors of periods, in place of Andr\'e motives and motivated torsors of periods, provided the standard conjectures hold for $M$)\,:

\begin{lem}\label{L:torsor-descent}
	Let $M\in M_K^\And\otimes L$ be an Andr\'e 
	motive over $K$ with $L$-coefficients, and let $L\subseteq L'\subseteq \overline{\Q}$ be a field extension.
	The following implications hold\,:
\begin{align*}
Z_{M_{\overline \Q}\otimes L'} = \Omega^\And_{M_{\overline \Q}\otimes L'} \  & \Longrightarrow Z_M = \Omega^\And_M\\
\Omega_{M_{\overline \Q}\otimes L'}  = \Omega^\And_{M_{\overline \Q}\otimes L'} \ & \Longrightarrow \ \Omega_M = \Omega^\And_M\\
\Omega^\dRB_{M_{\overline \Q}\otimes L'} = \Omega^\And_{M_{\overline \Q}\otimes L'}  \ & \Longrightarrow \ \Omega^\dRB_M = \Omega^\And_M
\end{align*}
\end{lem}
\begin{proof}
 Let $\Theta_M$ be any of $Z_M$, $\Omega_M$ or $\Omega^\dRB_M$.
 By assumption, we have inclusions $\Omega^\And_{M_{\overline \Q}} \subseteq (\Theta_{M})_{\overline \Q} \subseteq (\Omega^\And_{M})_{\overline \Q}$.
		By \cite[\S4.6]{AndreIHES}, there is a short exact sequence
		$$ 1 \to G_\And(M_{\overline \Q}) \to G_\And(M) \to \Gal(F/K)\to 1 $$
		for some finite Galois extension $F$ of $K$.
		This implies that 
		\begin{equation}\label{decompositiontorsor} 
		(\Omega^\And_{M})_{\overline \Q}= \coprod_{\Gal(F/K)} \Omega^\And_{M_{\overline \Q}}. 
		\end{equation}
		The inclusion $\Theta_M \subseteq \Omega^\And_M$ is defined over $K$, and therefore the action of $\mathrm{Gal}(\overline \Q/K)$ preserves the subvariety $(\Theta_{M})_{\overline \Q} \subseteq (\Omega^\And_{M})_{\overline \Q}$.
		Since $\mathrm{Gal}(\overline \Q/K)$ permutes the components on the right-hand side of (\ref{decompositiontorsor}) and $\Omega^\And_{M_{\overline \Q}} \subseteq (\Theta_{M})_{\overline \Q}$, we conclude that $(\Theta_{M})_{\overline \Q} = (\Omega^\And_{M})_{\overline \Q}$.
\end{proof}

\begin{rmk}\label{rmk:torsor_notalgclosed}
	Let $N \in \mathscr C_{\dRB,K,\Q}$ be a de\,Rham--Betti object over $K$.
	If $N = \HH^*_{\dRB}(M)$ is the de\,Rham--Betti realization of an Andr\'e motive $M$ over $K$ and $\Omega_{M_{\overline \Q}} = \Omega^\And_{M_{\overline \Q}}$, 
	then Lemma \ref{L:torsor-descent} implies that $\Omega_N = \Omega^{\dRB}_{N \otimes K}$.
	We want to emphasize that this last equality is not true for a general de\,Rham--Betti object $N$ over $K$.
	For example, let $a \in \overline \Q$ be such that $a^n \notin K$ for all $n \ge 1$.
	Then we can define $$N := ( K, \Q, c_N ),$$ where $c_N$ is multiplication by $a$.
	It is easy to see that $\Omega_N = \Spec F^g$, where $F^g$ denotes the Galois closure of $K(a)$.
	But $(N\otimes K)^{\otimes m}$ does not admit a de\,Rham--Betti class for $m >0$, and therefore $\Omega^{\dRB}_{N \otimes K}$ is a torsor under $\mathds{G}_{m,K}$.
	Note that this is in contrast to Proposition \ref{P:barQ-torsor}, which holds for arbitrary de\,Rham--Betti objects.
\end{rmk}

\subsection{Shimura periods of CM abelian varieties and $\overline{\Q}$-de\,Rham--Betti classes}
	\label{SS:Shimura}

Let $E$ be a CM field of degree $2g$ and let $A$ be an abelian variety (of dimension~$g$) over $\overline{\Q}$ with CM by $E$.
The de\,Rham cohomology $\HH^1_\dR(A)$ has an $E$-eigenbasis $\{\omega_{1}, \dots, \omega_{g}, \eta_{1}, \dots, \eta_{g}\}$ with $\omega_{j}$ holomorphic one-forms on~$A$ and $\eta_{j}= \bar{\omega}_{j}$ anti-holomorphic one-forms on $A$.
The Betti cohomology $\HH^1_B(A_,\overline{\Q})$ has an $E$-eigenbasis $\{\phi_{1}, \dots, \phi_{g}, \bar{\phi}_{1}, \dots, \bar{\phi}_{g}\}$ with the property that the comparison isomorphism diagonalizes with respect to those bases, with diagonal entries denoted by $(\theta_{1}, \dots, \theta_{g}, \bar{\theta}_{1}, \dots, \bar{\theta}_{g})$
and satisfying
$\theta_{j} \bar{\theta}_{j} = 2\pi i$\,; see \cite[Rmk.~6.5]{GUY}.
The complex numbers $\theta_{j}$ and $\bar{\theta}_{j}$, which are defined up to multiplication by $\overline \Q^\times$, are called the \emph{Shimura periods} of $A$.

The purpose of this paragraph is simply to note that  monomial relations among  the Shimura periods $\theta_{j}$, $\bar\theta_{j}$ are exactly induced by $\overline{\Q}$-de\,Rham--Betti classes in $\h^1(A)^{\otimes n}\otimes \mathds{1}(k)$ for $n>0$ and $k\in \Z$.
To see this, first note that the relations $\theta_{j} \bar{\theta}_{j} = 2\pi i$ imply that any monomial relation among the $\theta_{j}, \bar{\theta}_{j} $ can be written as a monomial in the $\theta_{j}, \bar{\theta}_{j} $ equals a $\overline \Q$-multiple of a power of $2\pi i$.
Let then $P(X_i,Y_i) = \prod_{i=1}^g X_{i}^{a_{i}} Y_{i}^{b_{i}}$ 
be a monomial of total degree $n$ 
and assume that $P(\theta_{i},\bar{\theta}_{j})$
is a $\overline{\Q}$-multiple of $(2\pi i)^k$.
This relation is induced by integrating 
$\varpi = \bigotimes_{i=1}^g \omega_{i}^{\otimes a_{i}} \otimes \eta_{i}^{\otimes b_{i}}$  
along $\gamma = (2\pi i)^{-k} \bigotimes_{i=1}^g \gamma_{i}^{\otimes a_{i}} \otimes  \bar\gamma_{i}^{\otimes b_{i}}$ 
where $\{\gamma_{1}, \dots, \gamma_{g}, \bar{\gamma}_{1}, \dots, \bar{\gamma}_{g}\}$ is the dual basis of $\{\phi_{1}, \dots, \phi_{g}, \bar{\phi}_{1}, \dots, \bar{\phi}_{g}\}$.
The latter also exactly means that $\varpi$ defines a $\overline{\Q}$-de\,Rham--Betti class on 	$\h^1(A)^{\otimes n}\otimes \mathds{1}(k)$.
Conversely, any de\,Rham class on 	$\h^1(A)^{\otimes n}\otimes \mathds{1}(k)$ that extends to a $\overline{\Q}$-de\,Rham--Betti class provides monomial relations among the Shimura periods by integrating  along the various  classes $(2\pi i)^{-k} \bigotimes_{i=1}^g \gamma_{i}^{\otimes a_{i}} \otimes  \bar\gamma_{i}^{\otimes b_{i}}$ satisfying $\sum_i (a_i + b_i) = n$.

Likewise, one sees that monomial relations among the Shimura periods of all CM abelian varieties over $\overline \Q$ are exactly induced by $\overline \Q$-de\,Rham--Betti classes on products of CM abelian varieties.  Denote $\langle Ab^{\mathrm{cm}} \rangle$ the Tannakian subcategory of $\mathsf{M}^\And_{\overline \Q}\otimes \overline \Q$ generated by the motives of CM abelian varieties over $\overline \Q$. 
	Its motivated Galois group is the base-change to $\overline \Q$ of the so-called Serre torus. As such, any algebraic subquotient is reductive, and we get, as in \cite[Prop.~24.6.3.1(1)]{AndreBook}, that the realization functor $\langle Ab^{\mathrm{cm}} \rangle \to \mathscr C_{\dRB,\overline \Q, \overline \Q}$ is full if and only if the  \emph{monomial Shimura relations} of \cite[\S 24.4.4]{AndreBook} generate all monomial relations among the Shimura periods of CM abelian varieties. 
This in fact corrects \cite[Prop.~24.6.3.1(1)]{AndreBook}, where $\overline \Q$-coefficients were overlooked and where only the realization to $\mathscr C_{\dRB,\overline \Q, \Q}$ is considered. 
Moreover, contrary to what is claimed just before Corollary~2 in \cite{kahn}, 
the question of whether all monomial relations among Shimura periods of CM elliptic curves are generated by the monomial Shimura relations remains open in general.

\section{De\,Rham--Betti classes on abelian varieties}\label{S:ab}

The aim of this section is to generalize W\"ustholz's analytic subgroup theorem  by allowing $L$-coefficients\,; see Proposition~\ref{Qbarsubgroupmotives}.
As a consequence we establish Theorem~\ref{T:LdRB-abelian} giving that $L$-de\,Rham--Betti classes in $\h^2(A)(1)$ are $L$-algebraic for any abelian variety $A$ over $K$, and then proceed to proving Theorem~\ref{T:main-elliptic} concerning products of elliptic curves, but also Theorem~\ref{T:abelian-surfaces} concerning powers of abelian surfaces.

\subsection{Consequences of W\"ustholz's analytic subgroup theorem}

Let $M \in \mathscr C_{\dRB, \overline \Q_{\dR}, L_{\B}}$ and $\gamma \in M_{\B}^{\vee}$. The comparison $c_M: M_{\dR} \otimes_{\overline \Q} \C \stackrel{\simeq}{\longrightarrow} M_{\B} \otimes_L \C $ defines a $\overline \Q$-linear map
$$\int_{\gamma}: M_{\dR} \to \mathbb{C}, \,\,\, \omega \mapsto \gamma_{\C}(c_M(\omega_\C)).$$
As in Definition~\ref{D:ann}, we set
	 the \emph{annihilator} of $\gamma \in M_{\B}^{\vee}$ to be $\mathrm{Ann}(\gamma) =_{\mathrm{def}} \ker \int_{\gamma} \subseteq M_{\dR}$.\medskip

We denote by $\mathcal{AB} \subset \mathsf{M}^{\And}_{\overline \Q}$ the full abelian subcategory of the category of Andr\'e motives over~$\overline \Q$ generated by the motives $\h^1(A)$, where $A$ is an abelian variety over $\overline \Q$\,; it is equivalent to the category of abelian varieties over $\overline \Q$ up to isogeny.  \medskip

The following formulation of W\"ustholz's analytic subgroup theorem \cite{Wuestholz} is derived  from the more general formulation for 1-motives taken from \cite[Thm.~9.7]{HW}. 
It has the advantage that it admits a natural extension to motives with $\overline \Q$-coefficients (see \S \ref{SS:WuestholzQbar}) and will allow for applications regarding $\overline \Q$-de\,Rham--Betti classes.

\begin{prop}[Analytic subgroup theorem for abelian motives\,; {\cite[Thm.~9.7]{HW}}] \label{subgroupmotives}
	Let $M \in \mathcal{AB}$ and $\gamma \in \HH_{\B}(M)^{\vee}$.
	There exists a decomposition $$ M = M' \oplus M''$$ in the category $\mathcal{AB}$ such that $\gamma \in \HH_{\B}(M'')^{\vee}$ and $\mathrm{Ann}(\gamma) = \HH_{\dR}(M')$.
\end{prop}

\begin{proof}
	This is \cite[Thm.~9.7]{HW}, where a stronger version for 1--motives is proved.
	More precisely, we obtain a short exact sequence
	$$0 \to M' \to M \to M'' \to 0 $$ in the category of 1--motives
	such that $\gamma \in \HH_{\B}(M'')^{\vee}$ and $\mathrm{Ann}(\gamma) = \HH_{\dR}(M')$.
	Since the category $\mathcal{AB}$ is a full subcategory of the category of 1--motives which is stable under taking subquotients, we know that $M', M'' \in \mathcal{AB}$. The short exact sequence is split since the category $\mathcal{AB}$ is semi-simple.
\end{proof}

\begin{rmk}
	Proposition~\ref{subgroupmotives} can also be derived without reference to 1-motives by applying the more classical version of the analytic subgroup theorem \cite[Thm.~6.2]{HW} to the universal vector extensions of abelian varieties.
\end{rmk}

\begin{rmk}
	Let $M \in \mathcal{AB}$. By applying the de\,Rham--Betti realization, Proposition \ref{subgroupmotives} tells us that for every $\gamma \in H_{\B}(M)^{\vee}$, there is a decomposition $H_{\dRB}(M) = N' \oplus N''$ in the category $\mathscr C_{\dRB}$ such that $\gamma \in (N''_{\B})^{\vee}$ and $\mathrm{Ann}(\gamma) = N'_{\dR}$.
	This is not true for an arbitrary de\,Rham--Betti object $ N \in \mathscr C_{\dRB}$ and therefore forces a strong restriction on the possible de\,Rham--Betti objects which arise as realizations of motives in $\mathcal{AB}$.
	As an example, consider the de\,Rham--Betti object
	\[
	N:=\left(N_\dR=\overline \Q^{\oplus 2}, N_\B = \Q^{\oplus 2}, c_N\right), \quad \text{with }c_N = \begin{pmatrix}
	\alpha & \beta \\ - \beta & \gamma
	\end{pmatrix},
	\quad \mathrm{degtr}_\Q \Q(\alpha, \beta, \gamma) = 3.
	\]
	Denote by $\gamma_1, \gamma_2$ the dual basis of the natural basis $e_1,e_2$ of $N_\B$.
	Similarly, denote by $\omega_1, \omega_2$ the natural basis of $N_{\dR}$.
	If we let $\gamma:= (\gamma_1, \gamma_2) \in (N_B^{\oplus 2})^{\vee}$, then $\mathrm{Ann}(\gamma) \subset N_{\dR}^{\oplus 2}$ is the one-dimensional subspace generated by $(\omega_2, \omega_1)$.
	But one checks that the de\,Rham--Betti object $N$ is simple, hence there does not exist a subobject $N' \subset N^{\oplus 2}$ such that $\mathrm{Ann}(\gamma) = N'_{\dR}$.
\end{rmk}

\begin{thm}[Andr\'e, Bost, W\"ustholz] \label{T:dRB1A}
	The functor $\mathcal{AB} \to \mathscr C_{\dRB}$ is fully faithful and its image is closed under taking subobjects. In particular, the essential image forms a full abelian subcategory of $\mathscr C_{\dRB}$ which is semi-simple.
\end{thm}
\begin{proof}
	We first prove that the image is closed under taking subobjects.
	Let $M \in \mathcal{AB}$ and suppose that $0 \subsetneq N \subseteq \HH_{\dRB}(M)$ is a subobject in the category $\mathscr C_{\dRB}$ of de\,Rham--Betti objects.
	We may assume that $N$ is not contained in the de\,Rham--Betti realization of a submotive $M' \subsetneq M$, otherwise we replace $M$ by $M'$. We aim to show that $N = \HH_{\dRB}(M)$.
	To get a contradiction, assume that $N \subsetneq \HH_{\dRB}(M)$.
	Then we denote by $N' \in \mathscr C_{\dRB}$ the quotient of $\HH_{\dRB}(M)$ by $N$.
	If we choose $0 \not= \gamma \in {N'_{\B}}^{\vee}$, then Proposition \ref{subgroupmotives} gives a decomposition
	$M = M' \oplus M''$ in $\mathcal{AB}$ such that $0 \not= \gamma \in \HH_{\B}(M'')^{\vee}$ and $\mathrm{Ann}(\gamma) = \HH_{\dR}(M')$.
	But then $N_{\dR} \subseteq \mathrm{Ann}(\gamma) = \HH_{\dR}(M')$.
	This contradicts the assumption that $N$ is not contained in the realization of a proper submotive $M' \subsetneq M$.
	
	We now prove the full faithfulness. Let $M, N \in \mathcal{AB}$ and let $(f_{\dR}, f_{\B}): \HH_{\dRB}(M) \to \HH_{\dRB}(N)$ be a morphism in $\mathscr C_{\dRB}$. We denote by $\Gamma_{f_{\dR}} \subset \HH_{\dR}(M) \oplus \HH_{\dR}(N) $ and $\Gamma_{f_{\B}} \subset \HH_{\B}(M) \oplus \HH_{\B}(N)$ the graphs of $f_{\dR}$ and $f_{\B}$, respectively. Then $(\Gamma_{f_{\dR}}, \Gamma_{f_{\B}})$ defines a subobject of $\HH_{\dRB}(M) \oplus \HH_{\dRB}(N)$ in $\mathscr C_{\dRB}$.
	By the first part of the proof, this subobject is the realization of a submotive $\Gamma \subset M \oplus N$.
	The composition 
	$$\xymatrix{ r\colon \Gamma \subset M \oplus N \ar@{->>}[r]^{\qquad \mathrm{pr}_1} & M}$$ is an isomorphism.
	Then 
	$$\xymatrix{f\colon M \ar[r]^{r^{-1}\quad }& \Gamma \subset M \oplus N \ar@{->>}[r]^{\qquad \mathrm{pr}_2} &N}$$ 
	defines a morphism in $\mathcal{AB}$ whose de\,Rham--Betti realization is $(f_{\dR}, f_{\B})$.
\end{proof}

\begin{rmk} Theorem \ref{T:dRB1A} is stated in \cite[\S 7.5.3]{AndreBook} and established as a consequence of a weaker form of W\"ustholz's analytic subgroup theorem.
	 Bost \cite[Thm.~5.1 \&~5.3]{Bost} also
	 provides a proof of the full faithfulness based on the older transcendence theorems of Schneider and Lang.
	However, in the next subsection we will need the stronger version of the analytic subgroup theorem as formulated in Proposition \ref{subgroupmotives} to prove an analog of Theorem \ref{T:dRB1A} with $\overline \Q$--coefficients.
\end{rmk}

An easy consequence of Theorem~\ref{T:dRB1A} is the following\,:

\begin{thm}\label{T:H1}
	Let $X$ be a smooth projective variety over $K$. Then the de\,Rham--Betti object $\HH^1_\dRB(X,\Q)$ over $K$ does not have
	any odd-dimensional subobject. 
	In particular,  for any $k\in \Z$, any de\,Rham--Betti class on $\HH_\dRB^1(X,\Q(k))$ is zero.
\end{thm}
\begin{proof} It is enough to prove the theorem after base-changing $K$ to $\overline \Q$, and we thus assume $K=\overline \Q$.
 The Poincar\'e bundle induces an isomorphism $\h^1(X)\simeq \h^1(\mathrm{Pic}^0_X)$ of Andr\'e motives. It follows from Theorem~\ref{T:dRB1A} that any subobject of $\HH^1_\dRB(X,\Q)$ is isomorphic to the de\,Rham--Betti realization of $\h^1(A)$ for some abelian variety $A/\overline \Q$ and hence is even-dimensional.
 We have the identification
 $$\Hom_{\dRB}(\mathds 1, \HH_\dRB^1(X,\Q(k))) = \Hom_{\dRB}(\mathds 1(-k), \HH_\dRB^1(X,\Q)),$$
 which then shows that  any de\,Rham--Betti class on $\HH_\dRB^1(X,\Q(k))$ is zero.
\end{proof}

\begin{rmk}
	That $\HH_\dRB^1(X,\Q(k))$ does not support any non-zero de\,Rham--Betti class  appears for the cases $k=0$ and $k=1$ in \cite[Thms.~4.1 \& 4.2]{BostCharles}.
Note that  the proof of \cite[Thm.~4.2]{BostCharles} (the case $k=1$)  relies on \cite[Thm.~3.3]{BostCharles}. It appears that the statement of \cite[Thm.~3.3]{BostCharles} contains a typo as can be seen by considering the case $G=\mathds G_{m,\overline \Q}$ therein\,; the condition $\exp_{G_\C}(\C v) \cap G(\overline \Q) \neq \varnothing$ should be replaced with $\exp_{G_\C}(v) \in G(\overline \Q)$.
Note also that the functor $\omega$ of \cite[\S 7.5.3]{AndreBook} is not full nor is a subobject in the image of $\omega$ the image of a subobject, as can be seen with $0 = \Hom(\mathds G_a, \mathds G_m)$ but $\Hom(\omega(\mathds{G}_a), \omega(\mathds{G}_m)) = \overline \Q$.
\end{rmk}

\subsection{A version of W\"ustholz's analytic subgroup theorem with coefficients}
\label{SS:WuestholzQbar}
We write $\mathcal{AB} \otimes L\subset \mathsf{M}^{\And}_{\overline \Q} \otimes L$ for the full abelian subcategory generated by the motives $\h^1(A)\otimes L$, where $A$ is an abelian variety over $\overline \Q$.

\begin{prop}[Analytic subgroup theorem for abelian motives with $L$--coefficients] \label{Qbarsubgroupmotives}
	Let $M \in \mathcal{AB}\otimes L$ and $\gamma \in \HH_{\B}(M)^{\vee}$.
	There exists a decomposition $$ M = M' \oplus M''$$ in the category $\mathcal{AB} \otimes L$ such that $\gamma \in \HH_{\B}(M'')^{\vee}$ and $\mathrm{Ann}(\gamma) = \HH_{\dR}(M')$.
\end{prop}
\begin{proof}
	We first handle the case where $M = \h^1(A) \otimes L$ is the motive of an abelian variety.
	Write $$\gamma = \sum_{i=1}^n \lambda_i \gamma_i \in \HH_{\B}^1(A, \Q)^\vee \otimes L, \quad \text{with } \lambda_i \in L \ \text{and}\ \gamma_i \in \HH_{\B}^1(A, \Q)^\vee.$$
	Inspired by the proof of \cite[Thm.~9.10]{HW}, we apply Proposition \ref{subgroupmotives} to $\underline{\gamma}:= (\gamma_1, ..., \gamma_n) \in \HH_{\B}^1(A, \Q)^{\oplus n}$.
	This gives a decomposition $ \h^1(A)^{\oplus n} = N' \oplus N''$ in $\mathcal{AB}$ such that $\underline{\gamma} \in \HH_{\B}(N'')^{\vee}$ and $\mathrm{Ann}(\underline{\gamma}) = \HH_{\dR}(N')$.
	The motive $N'$ is the kernel of the second projection $\mathrm{pr}_2: \h^1(A)^{\oplus n} \to N''$.
	Via the inclusion $L \subset \mathrm{End}(A) \otimes L$, we can view $\lambda_i \in L $ as an endomorphism of $\h^1(A) \otimes L$.
	We can thus define the morphism 
	$$\xymatrix{ f\colon \h^1(A) \otimes L  \ar[rr]^{(\lambda_1, ..., \lambda_n)} && \h^1(A)^{\oplus n} \otimes L \ar[r]^{\quad \mathrm{pr}_2} & N'' \otimes L}$$
	in $\mathcal{AB} \otimes L$.
	We let $M':= \ker f \subset M = \h^1(A) \otimes L $ and write $M = M' \oplus M''$ for some motive $M'' \in \mathcal{AB}\otimes L$.
	We have to show that $\gamma \in \HH_{\B}(M'')^{\vee}$ and $\mathrm{Ann}(\gamma) = \HH_{\dR}(M')$.
	For the former, note that by construction the dual morphism
	$$ (\lambda_1, ..., \lambda_n)^{\vee}\colon (\HH_{\B}^1(A, \Q)^{\vee})^{\oplus n} \otimes L \longrightarrow \HH_{\B}^1(A, \Q)^{\vee} \otimes L $$
	maps $\underline{\gamma}$ to $\gamma$ and the subspace $\HH_{\B}(N'')^{\vee}\otimes L$ to $\HH_{\B}(M'')^{\vee}$.
	For the latter, one computes that $\mathrm{Ann}(\gamma) \subset \HH^1_{\dR}(A) \otimes L$ is precisely the preimage of $\mathrm{Ann}(\underline{\gamma})$ under 
	$$(\lambda_1, ..., \lambda_n)\colon \HH^1_{\dR}(A) \otimes L \longrightarrow \HH^1_{\dR}(A)^{\oplus n} \otimes L.$$
	By construction of $N'$, we have $\mathrm{Ann}(\underline{\gamma}) = \HH_{\dR}(N') = \ker \mathrm{pr}_{2, \dR}$.
	This proves that $ \mathrm{Ann}(\gamma) = \ker f_{\dR} = \HH_{\dR}(M')$ and concludes the proof for $M = \h^1(A) \otimes L$.
	
	In general, $M \in \mathcal{AB} \otimes L $ will be a direct summand of $\h^1(A) \otimes L$ for some abelian variety~$A$.
	Let $\gamma \in \HH_{\B}(M)^{\vee}$. Via the projection, this gives an element $\tilde{\gamma} \in \HH_{\B}(A)^{\vee}\otimes L$. Applying the proposition to the motive $\h^1(A) \otimes L$, we get a decomposition $\h^1(A) \otimes L= N' \oplus N''$ such that $\tilde{\gamma} \in \HH_{\B}(N'')^{\vee}$ and $\mathrm{Ann}(\tilde{\gamma}) = \HH_{\dR}(N')$.
	Setting $M' := M \cap N'$ and $M''= M / M'$, we get $\gamma \in \HH_{\B}(M'')^{\vee}$ and $$\mathrm{Ann}(\gamma) = \mathrm{Ann}(\tilde{\gamma}) \cap \HH_{\dR}(M) = \HH_{\dR}(M'),$$
which concludes the proof of the proposition.
\end{proof}

\begin{rmk}[A version of the analytic subgroup theorem for $1$-motives with $L$-coefficients]
Using the arguments in the proof of Proposition~\ref{Qbarsubgroupmotives}, one can more generally establish a version with $L$-coefficients of the analytic subgroup theorem for 1-motives~\cite[Thm.~9.7]{HW}.
\end{rmk}

\begin{thm}\label{T:dRB1A-barQ}
	The functor $\mathcal{AB}\otimes L \to \mathscr C_{\dRB, \overline \Q, L}$ is fully faithful and its image is closed under taking subobjects. 
	In particular, the essential image forms a full abelian subcategory of $\mathscr C_{\dRB, \overline \Q, L}$ which is semi-simple.
\end{thm}
\begin{proof}
	The proof is the same as in Theorem \ref{T:dRB1A}, with essential input the version of the analytic subgroup theorem for abelian motives with $L$-coefficients (Proposition \ref{Qbarsubgroupmotives}).
\end{proof}

\begin{cor}\label{cor:LdRB-abelian}
	Let $K, L \subseteq \overline \Q$ be any fields.
	Let $A, A'$ be  abelian varieties over $K$.
	Then\,:
	\begin{enumerate}[(i)]
\item  any $L$-de\,Rham--Betti class on $\h^1(A) \otimes_{\Q} \h^1(A')^\vee$ is $L$-algebraic.
\item $\HH^1_\dRB(A,L(k))$ is a semi-simple object in $\mathscr C_{\dRB,K,L}$ for all $k\in \Z$.
	\end{enumerate}
\end{cor}
\begin{proof}
By  Lemma~\ref{L:descent}, in order to prove $(i)$, it suffices to show that every $L$-de\,Rham--Betti class on $\h^1(A_{\overline \Q}) \otimes_{\Q} \h^1(A'_{\overline \Q})^\vee$ is $L$-motivated.
The latter follows from Theorem~\ref{T:dRB1A-barQ}.
Regarding $(ii)$, let $N\subseteq \HH^1_\dRB(A,L(k))$ be a subobject. By Theorem~\ref{T:dRB1A-barQ}, there exists an $L$-motivated class on $\h^1(A_{\overline \Q}) \otimes_{\Q} \h^1(A_{\overline \Q})^\vee$ whose action on $\HH^1_\dRB(A_{\overline \Q},L(k))$ is a projector with image $N_{\overline \Q}$. 
The Galois group $\Gal(K)$ acts on $p$ via a finite quotient $G$, and 
it is easy to check that the motivated class over $K$ with $L$-coefficients 
$\frac{1}{|G|} \sum_{g\in G} g\circ p \circ g^{-1}$ acts on 
$\HH^1_\dRB(A,L(k))$ as a projector with image~$N$.
\end{proof}

\subsection{The de\,Rham--Betti group of an abelian variety.}

Let $A$ be an abelian variety over $\overline \Q$.
Its Mumford--Tate group is denoted by $\operatorname{MT}(A)$.

\begin{thm}\label{T:abelian-dRB}
Let $A/\overline\Q$ be an abelian variety of positive dimension. The de\,Rham--Betti group $G_\dRB(\h(A)\otimes L)$ has the following properties.
\begin{enumerate}[(i)]
	\item $G_{\dRB}(\h(A)\otimes L) $ is a connected reductive subgroup of $\mathrm{MT}(A)_L$.
	\item $\End_{G_{\dRB}(\h(A)\otimes L)}(\HH^1_{\B}(A,L)) = \End(A)_{\Q} \otimes_{\Q} L$.
	\item $\det: G_{\dRB}(\h(A)\otimes L) \rightarrow \mathds{G}_m$ is surjective.
	\item $A$ has complex multiplication if and only if $G_\dRB(\h(A)\otimes L)$ is a torus. 
\end{enumerate}
\end{thm}
\begin{proof}
That $G_{\dRB}(A)$ lies in  $\mathrm{MT}(A)_L$ follows from the inclusion  $G_{\dRB}(\h(A)\otimes L) \subseteq (G_{\dRB}(A))_L$  and from \cite[Thm.~0.6.2]{AndreIHES} giving that
the inclusion $\mathrm{MT}(A) \subseteq G_\And(A)$ is an equality. 
That $G_{\dRB}(A)$ is connected is Theorem~\ref{T:connectedness}.
As a consequence of 
Theorem~\ref{T:dRB1A-barQ}, the de\,Rham--Betti cohomology group $\HH^1_{\dRB}(A,\Q)$ is semi-simple as an object in~$\mathscr C_{\dRB}$.
It follows that,
since we are working in characteristic zero, $G_\dRB(A)$ is reductive.
Statement~$(ii)$ is Theorem~\ref{T:dRB1A-barQ}. 
Regarding $(iii)$,  the image of $\det$ is connected. 
Assume it is trivial. Then $G_\dRB(\h(A)\otimes L)$ acts
trivially on $\det(\HH^1_\B(A, L)) = \HH^{2\dim A}_\B(A, L)$. But $\h^{2\dim A}(A) \simeq \mathds{1}(-\dim A)$. 
This is a contradiction since $2\pi i$ is transcendental.
For $(iv)$, suppose $A$ has complex multiplication. 
Since $G_{\dRB}(\h(A)\otimes L)$ is a reductive and connected subgroup of the torus $\mathrm{MT}(A)_L$, it has to be a torus.
Conversely, assume that $G_{\dRB}(\h(A)\otimes L)$ is a torus. 
We recall the following classical argument (usually used for the Mumford--Tate group)\,:
$G_{\dRB}(\h(A)\otimes L)$ is contained in a maximal torus $T \subseteq \operatorname{GL}(\HH^1_\B(A, L))$. 
Then $$\End_{T}(\HH^1_{\B}(A,L)) \subseteq \End_{G_{\dRB}(\h(A)\otimes L)}(\HH^1_{\B}(A,L)) = \End(A)_{L}.$$
But $\End_{T}(\HH^1_{\B}(A,L)) $ is a commutative $L$-algebra of dimension $2g$, as can be seen after extending scalars to an algebraically closed field. It follows
 that $A$ has complex multiplication.
\end{proof}

\subsection{The de\,Rham--Betti conjecture for products of elliptic curves}
We prove the following stronger (but equivalent, by the motivic analogue of Proposition~\ref{P:torsor-fullyfaithful}) version of Theorem~\ref{T:main-elliptic} in the case of products of elliptic curves\,; see also Remark~\ref{rmk:elliptic} below.

\begin{thm}\label{T:dRB-elliptic}
	Let $E_1,\cdots, E_s$ be pairwise non-isogenous elliptic curves over $\overline \Q$ and let $A$ be an abelian variety over~$K$ such that $A_{\overline \Q}$ is isogenous to  $E_1^{n_1} \times \cdots \times E_s^{n_s}$. 
	\begin{enumerate}[(i)]
\item The de\,Rham--Betti Conjecture~\ref{C:dRB}$(ii)$ holds for $\h(A)$, i.e.,
$$\Omega^\dRB_{A} = \Omega^\mot_{A}.$$
In particular, by Proposition~\ref{P:torsor-fullyfaithful}, for any $n\geq 0$ and any $k\in \Z$, any de\,Rham--Betti class on $\h(A^n)(k)$ is algebraic.
\item If $L$ contains at most one of the CM fields associated with the CM elliptic curves among the elliptic curves $E_i$,
then  the de\,Rham--Betti Conjecture~\ref{C:dRB}$(ii)$ holds for $\h(A_{\overline \Q})\otimes L$, i.e.,
$$\Omega_{\h(A_{\overline \Q})\otimes L} = \Omega^\mot_{\h(A_{\overline \Q})\otimes L}.$$
In particular, by Proposition~\ref{P:torsor-fullyfaithful} and Lemma~\ref{L:descent}, for any $n\geq 0$ and any $k\in \Z$, any $L$-de\,Rham--Betti class on $\h(A^n)(k)$ is an $L$-linear combination of algebraic classes.
	\end{enumerate}
	
\end{thm}

We start by recalling that the Hodge conjecture is known for products of complex elliptic curves. Recall that the \emph{Hodge group} $\mathrm{Hdg}(H)$ of a rational Hodge structure $H$ is the smallest algebraic $\Q$-subgroup of 
  $\mathrm{GL}(H)$ that contains the image of $\mathrm{U}_{\C/\R} \hookrightarrow \mathrm{Res}_{\C/\R} \mathds{G}_m \to \mathrm{GL}(H)_{\R}$, where the right arrow is the morphism defining the Hodge structure on $H$.  The group $\mathrm{Hdg}(H)$ can be characterized as being the largest subgroup of   $\mathrm{GL}(H)$ that fixes all Hodge classes in tensor spaces associated with $H$. 
  In case $H$ is of pure weight $0$, then the Hodge group of $H$ agrees with its Mumford--Tate group, while in case $H$ is of pure weight $n\neq 0$, its Mumford--Tate group $\mathrm{MT}(H)$ is the image of the multiplication map
  $\mathds{G}_m \times \operatorname{Hdg}(H) \to \mathrm{GL}(H)$.
 If $A$ is a complex abelian variety, the Hodge group $\operatorname{Hdg}(A)$ (resp.\ the Mumford--Tate group $\operatorname{MT}(A)$) of $A$ is the Hodge group (resp.\ the Mumford--Tate group) of the Hodge structure $\HH_\B^1(A,\Q)$.
 
 For a CM field $F$, denote by $\bar \cdot$ the complex conjugation and by $\mathrm{U}_{F}$ the unitary torus given on $R$-points by
 $$\mathrm{U}_{F}(R) = \{g \in (F \otimes_{\Q} R)^{\times} \ \vert \ g \bar{g} = 1 \}.$$

 \begin{prop}\label{P:hodge-elliptic}
 	Let $E_1,\cdots, E_s$ be pairwise non-isogenous complex elliptic curves and let $A$ be an abelian variety isogenous to $E_1^{n_1} \times \cdots \times E_s^{n_s}$. Denote $V_i := \HH^1(E_i,\Q)$ and $K_i := \End(E_i)\otimes \Q$. Then the Hodge group of $A$ is
 	$$\operatorname{Hdg}(A) =\operatorname{Hdg(E_1) \times \cdots \times \operatorname{Hdg}(E_s)},\quad
 	\mbox{where} \ \mathrm{Hdg}(E_i) = 	\left \{
 	\begin{array}{ll}
 	\mathrm{U}_{K_i} & \mbox{if $E_i$ has CM}\,; \\
 	\mathrm{SL}(V_i) & 	 \mbox{if $E_i$ is without CM,}
 	\end{array} \right.$$ 
 	In particular, the subspace of Hodge classes on $A$ is generated by Hodge classes in $\HH^2(A,\Q)$ and, consequently, every Hodge class on $A$ is algebraic.
 \end{prop}
 \begin{proof}
 	This is classical and we refer to \cite[Part III]{Hodge-survey} for a proof.
 \end{proof}

\begin{proof}[Proof of Theorem \ref{T:dRB-elliptic}$(i)$] 
The proof goes similarly as for the Mumford--Tate group,
with the notable difference, laid out in Step~3 below, that there is no theory of weights available, meaning that it is a priori not clear if $G_{\dRB}(A) $ contains the scalar matrices.

By Lemma~\ref{L:torsor-descent}, we may and do assume that $K=\overline \Q$.
By Proposition~\ref{P:hodge-elliptic}, we only need to show that $G_{\dRB}(A) = G_{\And}(A)$. 
From Andr\'e~\cite[Thm.~0.6.2]{AndreIHES}, we have $G_{\And}(A) = \mathrm{MT}(A)$. Hence we have to show that the inclusion $G_{\dRB}(A)\subseteq\mathrm{MT}(A)$ is an equality, which is equivalent to showing that the inclusion $G^1_{\dRB}(A)\subseteq \mathrm{Hdg}(A)$ is an equality, where $G^1_{\dRB}(A)$ denotes the connected component of the kernel of $\det : G_\dRB(A)\rightarrow \mathds{G}_m$. 
\smallskip

\textit{Step 1}\,: $A=E$ is an elliptic curve. If $E$ has complex multiplication, then $\mathrm{MT}(E) = \mathrm{Res}_{K/\Q}\mathds{G}_{m,K}$, where $K=\End(E)_{\Q}$. In this case we know from Theorem~\ref{T:abelian-dRB} that $G_{\dRB}(E)\subseteq \mathrm{MT}(E)$ is a subtorus. 
If $G_{\dRB}(E)$ is a strict subtorus, then either $G_{\dRB}(E)\subseteq \mathrm{U}_{K}$ or $G_{\dRB}(E)\subseteq \mathds{G}_m$. Note that the case $G_{\dRB}(E)\subseteq \mathrm{U}_{K}$ violates the surjectivity of the determinant and the case $G_{\dRB}(E)\subseteq \mathds{G}_m$ violates the fact that $\End_{G_{\dRB}(E)}(\HH^1_{\B}(E,\Q)) = K$. 
Hence we see that $G_{\dRB}(E) = \mathrm{MT}(E)$. 
If $E$ does not have  CM, then $\mathrm{MT}(E)=\mathrm{GL}(V)$, where $V=\HH^1_\B(E,\Q)$. In this case, we have $\mathrm{Hdg}(E) = \mathrm{SL}(V)$ and $G^1_{\dRB}(E)\subseteq \mathrm{SL}(V)$ is a connected reductive subgroup. 
It follows that up to conjugation $G^1_{\dRB}(E)\subseteq \mathrm{SL}(V)$ is either contained in a maximal torus or equal to $\mathrm{SL}(V)$.
In the first case, there is a space of invariants of dimension at least 2 in $\End(V)$ which violates the assumption that $\End(E)_\Q=\Q$. 
Hence $G^1_{\dRB}(E) = \mathrm{SL}(V)$ and thus $G_{\dRB}(E) = \mathrm{GL}(V) = \mathrm{MT}(E)$.
\smallskip

\textit{Step 2}\,: $A=E_1^{n_1}\times E_2^{n_2}\times \cdots \times E_r^{n_r}$, for pairwise non-isogenous CM elliptic curves $E_i$. The fact that $\det(\HH^1_\dRB(E_i,\Q))\simeq \Q(-1)$ for all $i$ implies that
\begin{equation}\label{eq:G1-product}
G^1_\dRB(A) \subseteq G^1_{\dRB}(E_1)\times G^1_{\dRB}(E_2)\times \cdots \times G^1_{\dRB}(E_r)
\end{equation}
and that $G^1_\dRB(A)$ surjects onto each factor $G^1_\dRB(E_i)=\mathrm{U}_{K_i}$ with $K_i=\End(E_i)_{\Q}$. By Goursat's lemma \cite[Prop.~B.71.1]{Hodge-survey} and the fact that the $\mathrm{U}_{K_i}$ are pairwise non-isogenous,
 we conclude that the above inclusion is an identity. It follows from Proposition \ref{P:hodge-elliptic} that $G^1_\dRB(A)=\mathrm{Hdg}(A)$.
 
\smallskip
\textit{Step 3}\,: $A=E_1^{n_1}\times E_2^{n_2}$, where $E_1$ and $E_2$ are non-isogenous elliptic curves without CM. 
Due to the lack of a theory of weights in the de\,Rham--Betti setting, we are unable to use the usual argument for the Mumford--Tate group as in \cite[\S3]{Hodge-survey}. 
We may assume that $n_1=n_2=1$. Let $V_i:=\HH_\B^1(E_i,\Q)$, $i=1,2$. 
Then we know that $\mathrm{MT}(E_i) = \mathrm{GL}(V_i)$, $i=1,2$. Consider the inclusion $G_\dRB(A)\subseteq \mathrm{MT}(A)$ at the level of Lie algebras. We get
\[
	\mathfrak{g}=\mathfrak{g}^{ss}\oplus \mathfrak{z}(\mathfrak{g}) \hookrightarrow \mathfrak{mt} = \mathfrak{sl}_2 \oplus \mathfrak{sl}_2 \oplus \mathfrak{c},
\]
where $\mathfrak{g}:= \mathrm{Lie}\; G_\dRB(A)$ is the Lie algebra of $G_\dRB(A)$, $\mathfrak{z}(\mathfrak{g})$ corresponds to the center of $G_\dRB(A)$ and $\mathfrak{c}$ is a one-dimensional commutative Lie algebra. 
The fact that $G_\dRB(A)$ surjects onto $G_\dRB(E_i)$ implies that $\mathfrak{g}^{ss}\hookrightarrow \mathfrak{sl}_2 \oplus \mathfrak{sl}_2$ surjects onto both factors. 
This implies $\mathfrak{z}(\mathfrak{g}) \hookrightarrow \mathfrak{c}$. 
Furthermore, by Goursat's lemma~\cite[Prop.~B.71.2]{Hodge-survey}, we have either $\mathfrak{g}^{ss}\simeq\mathfrak{sl}_2$, embedded in $\mathfrak{sl}_2 \oplus \mathfrak{sl}_2$ as the graph of an isomorphism, or $\mathfrak{g}^{ss} = \mathfrak{sl}_2 \oplus \mathfrak{sl}_2$. 
If $\mathfrak{g}^{ss}\simeq \mathfrak{sl}_2$, then $\HH^1_\B(A,\Q) = V_1\oplus V_2$ becomes the direct sum of two copies of the standard representation of $\mathfrak{sl}_2$. 
As a consequence, the space of $G_\dRB(A)$-invariants of $\End(\HH^1_\B(A,\Q))$ 
is of dimension~4, 
which contradicts the assumption that $E_1$ and $E_2$ are not isogenous. 
Hence $\mathfrak{g}^{ss} = \mathfrak{sl}_2 \oplus \mathfrak{sl}_2$. 
The fact that $G_\dRB(A)$ surjects onto $\mathds{G}_m$ implies that $\mathfrak{z}(\mathfrak{g}) = \mathfrak{c}$. 
It follows that $G_\dRB(A) = \mathrm{MT}(A)$.

\smallskip

\textit{Step 4}\,: $A=E_1^{n_1}\times E_2^{n_2}\times \cdots \times E_r^{n_r}$ for pairwise non-isogenous non-CM elliptic curves $E_i$. In this case, by Step 3,  the inclusion \eqref{eq:G1-product} surjects onto the product of each pair of factors. Then one can apply a form of Goursat's lemma \cite[Prop.~B.71.3]{Hodge-survey}.
\smallskip

\textit{Step 5}\,: $A=B\times C$ where $B$ is a product of non-CM elliptic curves and $C$ is a product of CM elliptic curves. In this case, we apply Goursat's lemma to $G^1_{\dRB}(A)\subset G^1_\dRB(B)\times G^1_\dRB(C)$ and note that $G^1_\dRB(B)$ is semi-simple and $G^1_\dRB(C)$ is a torus.
\end{proof}

\begin{proof}[Proof of Theorem~\ref{T:dRB-elliptic}$(ii)$] 
	By Lemma~\ref{L:torsor-descent}, we may and do assume that $K=\overline \Q$.
	By Proposition~\ref{P:hodge-elliptic}, we only need to show that $G_\dRB(\h(A)\otimes  L) = G_{\And}(A)_{L}$. 
	This holds for $A$ a non-CM elliptic curve by the exact same argument as in the proof of Theorem~\ref{T:dRB-elliptic}$(i)$. For $A$ a CM elliptic curve, this holds 
	thanks to Chudnovsky's theorem~\cite{Chudnovsky}\,;
        more precisely, denoting $F = \mathrm{End}(A)_\Q$, we have $G_\dRB(\h(A)\otimes  L) = \mathds{G}_{m,L}^2$ if $L\supseteq F$ and $G_\dRB(\h(A)\otimes  L) = (\mathrm{Res}_{F/\Q}(\mathds{G}_m))_L$ otherwise.
        In particular, Step 2 of the proof of Theorem~\ref{T:dRB-elliptic}$(i)$ carries over under the condition that $L$ contains at most one of the CM fields $K_i = \mathrm{End}(E_i)_\Q$.
	Finally it is easy to see that
	Steps 3, 4 and 5 of the proof of Theorem~\ref{T:dRB-elliptic}$(i)$ carry over to the setting of $L$-de\,Rham--Betti classes.
\end{proof}

\begin{rmk}\label{rmk:elliptic}
The de\,Rham--Betti conjecture remains open for $\h(A) \otimes \overline \Q$ if contains two non-isogenous CM elliptic curves as factors. 
In that situation, the arguments in Step 2 of the proof of Theorem~\ref{T:dRB-elliptic}$(i)$ do not carry over as tori over $\overline \Q$ are merely classified by their rank. Moreover, we would like to point out that we are not aware of a proof of the de\,Rham--Betti Conjecture~\ref{C:dRB}$(ii)$ for $\h(E)\otimes \overline \Q$ with $E$ a CM elliptic curve that does not employ Chudnovsky's theorem~\cite{Chudnovsky}.
\end{rmk}

\subsection{The case of abelian surfaces}
We prove the following equivalent (by the motivic analogue of Proposition~\ref{P:torsor-fullyfaithful}) version of Theorem~\ref{T:main-elliptic} in the case of powers of abelian surfaces with non-trivial endomorphism ring.
For a smooth projective variety $X$ over $K$, we write $\Omega^?_X$ for $\Omega^?_{\h(X)}$ and $G_?(X)$ for $G_?(\h(X))$, where $?$ is any of $\dRB$, $\And$ or $\mot$, and where $\h(X) \in \mathsf{M}^\And_K$ is the Andr\'e motive of $X$. If $K=\overline \Q$, we write $\Omega_X$ for $\Omega_{\h(X)\otimes \overline \Q} = \Omega^\dRB_{\h(X)\otimes \overline \Q}$ (Proposition~\ref{P:barQ-torsor}).

\begin{thm}\label{T:abelian-surfaces}
	Let $A$ be an abelian surface over $K$ such that $\operatorname{End}(A_{\overline \Q})\otimes \Q \neq \Q$. 
	\begin{enumerate}[(i)]
\item The de\,Rham--Betti conjecture (Conjecture~\ref{C:dRB}(ii)) holds for $\h(A)$, i.e.,
$$\Omega^\dRB_{A} = \Omega^\mot_{A}.$$
In particular, by Proposition~\ref{P:torsor-fullyfaithful}, for any $n\geq 0$ and any $k\in \Z$, any de\,Rham--Betti class on $\h(A^n)(k)$ is algebraic.
\item If in addition $A$ is simple without CM, then the de\,Rham--Betti conjecture (Conjecture~\ref{C:dRB}(ii)) holds for $\h(A)\otimes \overline \Q$, i.e.,
$$\Omega_{A} = \Omega^\mot_{A}.$$
In particular, by Proposition~\ref{P:torsor-fullyfaithful}, for any $n\geq 0$ and any $k\in \Z$, any $\overline \Q$-de\,Rham--Betti class on $\h(A_{\overline \Q}^n)(k)$ is a $\overline \Q$-linear combination of algebraic classes.
	\end{enumerate}
\end{thm}
\begin{proof}[Proof of Theorem~\ref{T:abelian-surfaces}$(i)$] 
By Lemma~\ref{L:torsor-descent}, we may and do assume that $K=\overline \Q$.
The case where $A$ is isogenous to a product of elliptic curves is handled by Theorem \ref{T:dRB-elliptic}. We therefore assume that $A$ is a simple abelian surface over $\overline \Q$.
We recall, e.g.\ from \cite[(2.2)]{MoonenZarhin}, that there are the following four possibilities for the endomorphism ring $\End(A)_{\Q}$\,: 
\begin{enumerate}[(a)]
	\item
	$\End(A)_{\Q} = \Q$;
	\item
	$\End(A)_{\Q} $ is a real quadratic extension of $\Q$;
	\item
	$\End(A)_{\Q} $ is a quaternion algebra with center $\Q$ which splits over $\R$;
	\item
	$\End(A)_{\Q} $ is a CM field of degree $4$ over $\Q$ which does not contain an imaginary quadratic subfield.
\end{enumerate}
We prove that the inclusion $G_{\dRB}(A)\subseteq\mathrm{MT}(A)$ is an equality when $\End(A)_{\Q} \not= \Q$.
We distinguish the different cases for $\End(A)_{\Q}$.
\smallskip

\textit{Case (b)\,:}
Suppose $k := \End(A)_{\Q}$ is a real quadratic field. Then
$\mathrm{MT}(A) \subseteq \mathrm{Res}_{k / \Q} \operatorname{GL}_{2,k}$ is the subgroup with $R$-points given by
$$\mathrm{MT}(A)(R) = \{g \in \operatorname{GL}_{2}(k \otimes_{\Q} R) \, | \, \mathrm{det} \, g \in R^{\times}\}.$$
Denote by $\mathfrak{g}$ the Lie algebra of $G_{\dRB}(A)$.
Write the reductive Lie algebra $ \mathfrak{g} = \mathfrak{g}^{ss} \oplus \mathfrak{z}(\mathfrak{g})$ as the sum of a semi-simple Lie algebra and an abelian Lie algebra. 
On the other hand, the Lie algebra of $\mathrm{MT}(A)_k$ is $\mathfrak{mt}_k = \mathfrak{sl}_{2,k} \oplus \mathfrak{sl}_{2,k} \oplus \mathfrak{c}$, where $\mathfrak{c}$ is a one-dimensional abelian Lie algebra.
Then $\mathfrak{g}_k^{ss} :=\mathfrak{g}^{ss} \otimes_{\Q} k \subseteq \mathfrak{sl}_{2,k} \oplus \mathfrak{sl}_{2,k}$ and the semi-simple part is non-trivial since otherwise $G_{\dRB}(A)$ would be a torus, contradicting Theorem \ref{T:abelian-dRB}$(iv)$.
Since $\mathfrak{g}^{ss}$ is defined over $\Q$, the $k$-Lie algebra $\mathfrak{g}^{ss}_k$ is invariant under the non-trivial automorphism $\sigma \in \mathrm{Gal}(k /\Q)$ which acts on $\mathfrak{sl}_{2,k} \oplus \mathfrak{sl}_{2,k}$ by conjugating and swapping both factors. Consequently, $\mathfrak{g}^{ss}_k$ surjects onto both factors.
Thus either 
\begin{eqnarray*} \mathfrak{g}_k^{ss} = \mathfrak{sl}_{2,k} \oplus \mathfrak{sl}_{2,k} & \mathrm{or} & \mathfrak{g}_k^{ss} \simeq \mathfrak{sl}_{2,k} \subset \mathfrak{sl}_{2,k} \oplus \mathfrak{sl}_{2,k} \end{eqnarray*} 
is the graph of an automorphism of $\mathfrak{sl}_{2,k}$. We have to exclude the latter.
Using Theorem~\ref{T:abelian-dRB}$(iii)$ it follows that $\mathfrak{z}(\mathfrak{g}) = \mathfrak{c}$ and thus $\mathfrak{g}_k \simeq \mathfrak{gl}_{2,k}$, where the action on $\HH^1_{\B}(A,\Q) \otimes_{\Q} k$ is isomorphic to the direct sum of two copies of the standard representation.
But then the dimension of $\End_{\mathfrak{g}_k}(\HH^1_{\B}(A,\Q) \otimes_{\Q} k)$ is four, which is different from $\dim_{\Q} \End(A)_{\Q} = 2$.
It follows that $\mathfrak{g}_k^{ss} = \mathfrak{sl}_{2,k} \oplus \mathfrak{sl}_{2,k}$.
By using Theorem~\ref{T:abelian-dRB}$(iii)$ once more we conclude $\mathfrak{g}_k = \mathfrak{mt}_k$.
\smallskip
	
\textit{Case (c)\,:}
In this case, $G_{\dRB}(A)_{\C} \subseteq \mathrm{MT}(A)_{\C} \simeq \operatorname{GL}_{2, \C}$.
	Since $G_{\dRB}(A)$ is not a torus by Theorem~\ref{T:abelian-dRB}$(iv)$,
	the group $G_{\dRB}(A)_{\C}$ is either $\mathrm{SL}_{2, \C}$ or $\operatorname{GL}_{2,\C}$.
	But $G_{\dRB}(A)$ admits a non-trivial map to $\mathds{G}_m$ by Theorem \ref{T:abelian-dRB}$(iii)$, so we have $G_{\dRB}(A)_{\C} = \operatorname{GL}_{2, \C}$.
	\smallskip
	
	\textit{Case (d)\,:}
	Suppose $F := \End(A)_{\Q}$ is a CM field of degree $4$ over $\Q$, which does not contain an imaginary quadratic subfield.
	In this case, $T:= \mathrm{MT}(A) \subseteq \mathrm{Res}_{F/\Q}\mathbb{G}_m$ is the torus whose points in a $\Q$-algebra $R$ are given by
	$$ T(R) = \{g \in (F \otimes_{\Q} R)^{\times} \ \vert \ g \bar{g} \in R^{\times}\}.$$
	The torus $T$ is determined up to isogeny by the rational vector space of ($\overline \Q$-)cochar\-acters $X_{*}(T)_{\Q}$ equipped with its $\Gal(\overline \Q/\Q)$-action.
	We have the short exact sequence
	$$1 \to \mathrm{U}_F \to T \to \mathbb{G}_m \to 1,$$
	where $\mathrm{U}_F$ is the unitary torus given on $R$-points by
	$\mathrm{U}_F(R) = \{g \in (F \otimes_{\Q} R)^{\times} \ \vert \ g \bar{g} = 1 \}.$
	This gives rise to the short exact sequence of $\Q$-vector spaces with $\mathrm{Gal}(\overline \Q/\Q)$-action
	$$ 0 \to X_{*}(\mathrm{U}_{F})_{\Q} \to X_{*}(T)_{\Q} \to X_{*}(\mathbb{G}_m)_{\Q} = \Q \to 0,$$
	where the Galois action on the rightmost term is trivial.
	
	Let $T_{\dRB} := G_{\dRB}(A) \subseteq T$, which is a torus by Theorem \ref{T:abelian-dRB}$(iv)$.
	It is enough to prove that the inclusion $X_{*}(T_{\dRB})_{\Q} \subseteq X_{*}(T)_{\Q}$ is an equality.
	By Theorem \ref{T:abelian-dRB}$(iii)$, the group $X_{*}(T_{\dRB})_{\Q}$ surjects onto $X_{*}(\mathbb{G}_m)_{\Q}$.
	The kernel is a Galois-stable subspace of $X_{*}(\mathrm{U}_F)_{\Q}$.
	
	\begin{claim}
		There is no non-trivial Galois-stable subspace of $X_*(\mathrm{U}_{F})_{\Q}$.
	\end{claim}
	\begin{proof}[Proof of Claim]
		
		Let $\Sigma_F := \Hom(F, \C) = \{\phi_1, \bar{\phi}_1, \phi_2, \bar{\phi}_2\}$ be the set of embeddings of the CM field~$F$ into the complex numbers.
		Then $X^*(\mathrm{Res}_{F/\Q}\mathbb{G}_m)_{\Q}$ is naturally identified with the $\Q$-vector space with basis $\Sigma_F$ with its natural $\mathrm{Gal}(\overline \Q/\Q)$-action, and we denote by $\{\phi_1^{\vee}, \bar{\phi}_1^{\vee}, \phi_2^{\vee}, \bar{\phi}_2^{\vee}\}$ the associated dual basis of $X_*(\mathrm{Res}_{F/\Q}\mathbb{G}_m)_{\Q}$.
		
		With this notation, 
		\begin{equation} \label{generators} X_*(\mathrm{U}_{F})_{\Q} = \langle \phi_1^{\vee} - \bar{\phi}_1^{\vee}, \phi_2^{\vee} - \bar{\phi}_2^{\vee} \rangle_{\Q}. \end{equation}
		Assume there exists a one-dimensional Galois-stable subspace $L \subseteq X_*(\mathrm{U}_{F})_{\Q}$.
		As $\mathrm{Gal}(\overline \Q/ \Q)$ acts transitively on $\Sigma_F$, we know that
		\begin{eqnarray}\label{notL} L \not= \langle \phi_1^{\vee} - \bar{\phi}_1^{\vee} \rangle_{\Q} & \mathrm{and} & L \not= \langle \phi_2^{\vee} - \bar{\phi}_2^{\vee} \rangle_{\Q}. \end{eqnarray}
		Denote by $F^g$ the Galois closure of $F$.
		Then $\mathrm{Gal}(\overline \Q/ \Q)$ acts on $L$ through a character $$ \chi: \mathrm{Gal}(F^g/\Q) \to \Z / 2 \Z.$$
		Let $K := (F^g)^{\mathrm{ker \,} \chi}$.
		As $F$ is a CM field, the Galois group $\mathrm{Gal}(F^g / F)$ fixes one pair of complex embeddings $(\phi_i, \bar{\phi}_i)$.
		Together with (\ref{notL}), we get that $\mathrm{Gal}(F^g / F)$ acts trivially on $L$, and hence $K \subseteq F$.
		
		Since $F$ does not contain a totally imaginary quadratic subfield, $K \subset F$ is totally real.
		Consequently, complex conjugation acts trivially on the target of
		$$\Sigma_F = \Hom(F, \C) \twoheadrightarrow \Hom(K, \C).$$
		It follows that there exist $\sigma_1, \sigma_2 \in \mathrm{Gal}(F^g/K)$ such that
		$ \bar{\phi}_1 = \sigma_1 \phi_1$ and $ \bar{\phi}_2 = \sigma_2 \phi_2$.
		Applying these relations to the generators in (\ref{generators}), we see that there is no non-zero element of $X_*(\mathrm{U}_{F})_{\Q}$ on which both $\sigma_1$ and $\sigma_2$ act trivially.
	\end{proof}
	
	The claim shows that either $X_{*}(T_{\dRB})_{\Q} = X_{*}(T)_{\Q}$, in which case we are done, or $$X_{*}(T_{\dRB})_{\Q} \simeq X_{*}(\mathbb{G}_m)_{\Q}.$$
	As $\mathrm{Gal}(\overline \Q/\Q)$ acts transitively on $\Sigma_F$, there is a unique one-dimensional subspace of $X_{*}(T)_{\Q}$ on which $\mathrm{Gal}(\overline \Q/\Q)$ acts trivially.
	It corresponds to the scalar matrices $\mathbb{G}_m \subseteq T$.
	Since by Theorem~\ref{T:abelian-dRB}$(ii)$, $T_{\dRB}$ cannot just consist of scalar matrices, we conclude that $X_{*}(T_{\dRB})_{\Q} = X_{*}(T)_{\Q}$.
	This shows $T_{\dRB} = T$, as desired.
\end{proof}

\begin{proof}[Proof of Theorem~\ref{T:abelian-surfaces}$(ii)$]
By Lemma~\ref{L:torsor-descent}, we may and do assume that $K=\overline \Q$. We have to deal with cases $(b)$
 and $(c)$. 
In case $(c)$ one proves that the inclusion $G_\dRB(\h(A)\otimes \overline \Q) \subseteq \mathrm{MT}(A)_{\overline \Q}$ is an equality  as in the proof of Theorem~\ref{T:abelian-surfaces}$(i)$ by using the second part of Theorem~\ref{T:abelian-dRB} instead of the first. 
As for case $(b)$, denote by $\mathfrak g$ the Lie algebra of $G_\dRB(\h(A)\otimes \overline \Q)$\,; its semi-simple part $\mathfrak g^{ss}$ is a Lie subalgebra of $\mathfrak{mt}_{\overline \Q}^{ss} = \mathfrak{sl}_{2,\overline \Q} \oplus \mathfrak{sl}_{2,\overline \Q}$. 
First we show that $\mathfrak g^{ss}$ surjects onto both factors. Here we may no longer use the Galois argument as in the proof of Theorem~\ref{T:abelian-surfaces}$(i)$. We argue instead as follows. We have the inclusion $$G_\dRB(\h(A)\otimes \overline \Q) \subseteq \operatorname{MT}(A)_{\overline \Q} = \mathds{G}_{m,\overline \Q} \cdot \big(\operatorname{SL}_{2,\overline \Q} \times \operatorname{SL}_{2,\overline \Q} \big) \subset \operatorname{GL}_{2,\overline \Q} \times \operatorname{GL}_{2,\overline \Q}.$$
Here, the action of $G_\dRB(\h(A)\otimes \overline \Q)$ on $\HH^1_\B(A,\overline \Q)$  is via the action of $\operatorname{GL}_{2,\overline \Q} \times \operatorname{GL}_{2,\overline \Q}$ on two copies of the standard representation. 
We claim that the projections of $G_\dRB(\h(A)\otimes \overline \Q)$  on both $\operatorname{GL}_{2,\overline \Q}$-factors contains $\operatorname{SL}_{2,\overline \Q}$. 
If not, since  $G_\dRB(\h(A)\otimes \overline \Q)$ is connected and reductive, then up to switching the factors we would have an inclusion $G_\dRB(\h(A)\otimes \overline \Q) \subseteq \operatorname{GL}_{2,\overline \Q} \times \operatorname{T}$ for some maximal torus $T$ inside $\operatorname{GL}_{2,\overline \Q}$. But then the space of invariants in $\End(\HH^1_\B(A,\overline \Q))$ under $G_\dRB(\h(A)\otimes \overline \Q)$ would have dimension at least 3, which contradicts Theorem~\ref{T:abelian-dRB}$(ii')$ and $\End(A)\otimes \overline \Q = \overline \Q \oplus \overline \Q$.
We conclude that  $\mathfrak g^{ss}$ surjects onto both $\mathfrak{sl}_{2,\overline \Q} $-factors.
We can now argue as in the proof of Theorem~\ref{T:abelian-surfaces}$(i)$ by using  the second part of Theorem~\ref{T:abelian-dRB} instead of the first and conclude that $\mathfrak g = \mathfrak{mt}_{\overline \Q}$.
\end{proof}

\begin{rmk}
	It seems that, in the case of an abelian surface $A$ over $\overline \Q$ with $\End(A)_{\Q} = \Q$, neither the properties established in Theorem \ref{T:abelian-dRB} nor Proposition~\ref{P:dRB-tensor2} below applied to the Kummer surface associated with $A$ are  sufficient to determine $G_{\dRB}(A)$ uniquely.
	In this case, the Mumford--Tate group is $\mathrm{MT}(A) = \mathrm{GSp}_4$.
	We do not know how to exclude the possibility that the semi-simple part $\mathfrak{g}^{ss} \subseteq \mathfrak{sp}_4$ of the Lie algebra of $G_{\dRB}(A)$ is $\mathfrak{sl}_2$, embedded via the third symmetric power of the standard representation.
\end{rmk}

\begin{rmk} 
Zekun Ji~\cite{Ji-Zekun} has recently established the de\,Rham--Betti Conjecture~\ref{C:dRB}$(ii)$ for simple CM abelian fourfolds.
\end{rmk}

\section{De\,Rham--Betti classes on hyper-K\"ahler varieties}\label{S:cod2}

We prove Theorems~\ref{T:main-dRBisometry}, \ref{T:dRB-Torelli}, 
\ref{T:main-dRB} and~\ref{T:main-GPC-CM} of the introduction. 
We will repeatedly use Andr\'e's result that the Kuga--Satake correspondence is motivated and defined over a finite field extension\,:

\begin{thm}[Andr\'e {\cite[Prop.~6.4.3]{AndreK3}}]
	\label{T:simpleKS}
	Let $X$ be a hyper-K\"ahler variety defined over a field $K\subseteq \C$. Then there is a finite field extension $K'/K$ such that the  \emph{Kuga--Satake variety} $A$ attached to $X_\C$ is defined over $K'$ and such that the Andr\'e motive $\h^2(X_{K'})(1)$ is a direct summand of $\h^1(A)\otimes \h^1(A)^\vee$. In particular, by \cite[Thm.~0.6.2]{AndreIHES}, if $Y$ is a product of hyper-K\"ahler varieties and of abelian varieties over $\overline \Q$, then Hodge cycles on $Y$ are motivated and the natural inclusion $\mathrm{MT}(Y) \subseteq G_\And(Y)$ is an equality.
\end{thm}

The condition that hyper-K\"ahler varieties have second Betti number larger than 3, which holds for all known deformation families of hyper-K\"ahler varieties, ensures that the deformation space of a hyper-K\"ahler variety is big enough and is crucial in Andr\'e's work~\cite{AndreK3}.

\subsection{First observations} Let $K,L\subseteq \overline \Q$ be two fields.
Via Andr\'e's Theorem~\ref{T:simpleKS}, Corollary~\ref{cor:LdRB-abelian} and Theorem~\ref{T:abelian-dRB} admit the following two consequences\,:

\begin{prop}
	\label{P:BC}
	Let $X$ be a hyper-K\"ahler variety over $K$. Then any $L$-de\,Rham--Betti class on $\h^2(X)(1)$ is an $L$-linear combination of algebraic classes.
\end{prop}
\begin{proof} We first note that the proposition in case $L=\Q$ and $K=\overline \Q$ is \cite[Thm.~5.6]{BostCharles}.
By Theorem~\ref{T:simpleKS}, there exists an abelian variety $A$ over $\overline \Q$ such that the Andr\'e motive  $\h^2(X_{\overline \Q})(1)$ is a direct summand of $\h^1(A)\otimes \h^1(A)^\vee$. 
We may then conclude with Corollary~\ref{cor:LdRB-abelian}.
\end{proof}

\begin{prop}\label{P:redX}
	Let $X$ be a hyper-K\"ahler variety over $\overline \Q$.
	Then  the de\,Rham--Betti group $G_{\dRB}(\h^2(X)\otimes L)$ 
	is reductive.
\end{prop}
\begin{proof}
	By Theorem~\ref{T:simpleKS}, there exists an abelian variety $A$ over $\overline \Q$ such that  
	$\h^2(X)$ is a direct summand of $\h(A)$.
	Hence  $G_{\dRB}(\h^2(X)\otimes L)$  is a quotient
	of $G_\dRB(\h(A)\otimes L)$. It follows from Theorem~\ref{T:abelian-dRB}$(i)$ that both groups are reductive.
\end{proof}

\subsection{De\,Rham--Betti isometries between hyper-K\"ahler varieties are motivated}
For a polarized hyper-K\"ahler variety $X$ over $K$, we denote $\mathfrak{p}^2(X)$ its degree-2 primitive Andr\'e motive and $\mathfrak{t}^2(X)$ its degree-2 transcendental motive. Their de\,Rham--Betti realizations are respectively given by $\mathrm{P}_\dRB^2(X,\Q)$ and $\mathrm{T}_\dRB^2(X,\Q)$. The former is the orthogonal complement of the polarization in  $\HH^2_{\dRB}(X,\Q)$, while the latter is the orthogonal complement of the subspace of $\HH^2_{\dRB}(X,\Q)$ spanned by the classes of divisors defined over $K$.

\begin{thm}\label{P:GdRBisometry}
	Let $X$ and $X'$ be hyper-K\"ahler varieties over $K$ and let $L\subseteq \overline \Q$
	be a subfield. 
	Then 	
\begin{enumerate}[(i)]
		\item any $L$--de\,Rham--Betti isometry 
		$\mathrm{P}^2_\dRB(X,L) \stackrel{\sim}{\longrightarrow} \mathrm{P}^2_\dRB(X',L)$
		is $L$-motivated.
		\item any $L$--de\,Rham--Betti isometry 
		$ \mathrm{T}^2_\dRB(X,L) \stackrel{\sim}{\longrightarrow} \mathrm{T}^2_{\dRB}(X',L)$
				is $L$-motivated.
	  \item	any $L$--de\,Rham--Betti  isometry 
		$\mathrm{H}^2_\dRB(X,L) \stackrel{\sim}{\longrightarrow} \mathrm{H}^2_{\dRB}(X',L)$
				is $L$-motivated.
	\end{enumerate} 
   Here, 
    the isometries are with respect to either the polarization or the Beauville--Bogomolov form 
   (since
    those agree  on $\mathrm{P}^2_B(X,\Q)$, up to a rational multiple).
\end{thm}

\begin{proof} 
By Proposition~\ref{P:BC}, if $X$ is a hyper-K\"ahler variety, then $\mathfrak{t}^2(X)(1)$ does not support any non-zero $\overline \Q$-de\,Rham--Betti class. 
	Hence $(ii) \Rightarrow (iii)$ and it thus suffices to show items $(i)$ and~$(ii)$.
	By Lemma~\ref{L:descent}, we may and do assume that $K=L=\overline \Q$.
 Let $X$ and $X'$ be two hyper-K\"ahler varieties over $\overline \Q$ with respective polarizations $\eta$ and $\eta'$. Let $V:=\mathrm{P}_\B^2(X,\Q(1))$ (resp.\ $V:=\mathrm{T}_\B^2(X,\Q(1))$) and $V':=\mathrm{P}_\B^2(X',\Q(1))$ (resp.\ $V':=\mathrm{T}_\B^2(X',\Q(1))$) be the primitive cohomology (resp.\ transcendental cohomology). 
 Assume that we have a $G_{\dRB,\overline \Q, \overline \Q}$-invariant isometry $$i: V\otimes \overline \Q \stackrel{\sim}{\longrightarrow}  V'\otimes \overline \Q.$$ We will show that $i$ is $\overline{\Q}$-motivated.

Consider the Clifford algebras $\Lambda=C(V)$ and $\Lambda' = C(V')$ and view them respectively as $C^+(V)$- and $C^+(V')$-modules with their natural integral structures, where $C^+(V)$ and $C^+(V')$ are the corresponding even Clifford algebras.
Let $A$ and $A'$ be the corresponding Kuga--Satake varieties, which are defined over $\overline \Q$ by~\cite{AndreK3}.
We have identifications $\Lambda = \HH_\B^1(A,\Q)$ and $\Lambda' = \HH_\B^1(A',\Q)$. 
We fix a $C^+(V\otimes \overline \Q)$-basis $\{1,\alpha\}$ of $\Lambda\otimes \overline \Q$ with $\alpha\in V\otimes \overline \Q$,  
as well as the corresponding basis $\{1,\alpha'\}$ of $\Lambda'\otimes \overline \Q$, where $\alpha'=i(\alpha)\in V'\otimes \overline \Q$. 
This induces 
identifications of $\overline\Q$-algebras
\begin{align*}
C^+\otimes \overline \Q & :=(\End_{C^+(V)}\Lambda)^\op \otimes \overline \Q= \mathrm{M}_{2}(C^+(V\otimes \overline \Q))\quad \mbox{and} \\ 
{C'}^+\otimes \overline \Q & :=(\End_{C^+(V')}\Lambda')^\op \otimes \overline \Q = \mathrm{M}_{2}(C^+(V'\otimes \overline \Q))
\end{align*}
The isometry $i: V\otimes \overline \Q \longrightarrow V'\otimes \overline \Q$ induces $\overline\Q$-algebra isomorphisms 
$$\xymatrix{
 C^+\otimes \overline \Q \ar[rr]^{j_0} \ar@{=}[d]& & {C'}^+\otimes \overline \Q \ar@{=}[d] \\ 
 \mathrm{M}_2(C^+(V\otimes \overline \Q)) \ar[rr]^{\mathrm{M}_2(C^+(i))} && \mathrm{M}_2(C^+(V'\otimes \overline \Q))\,.}
$$
On the other hand, the $G_{\dRB,\overline \Q, \overline \Q}$-invariant isometry $i: V\otimes\overline\Q \longrightarrow V'\otimes\overline\Q$ induces a $G_{\dRB,\overline \Q, \overline \Q}$-invariant algebra isomorphism $C^+(i) : C^+(V\otimes \overline \Q) \longrightarrow C^+(V'\otimes \overline \Q)$. 
 By
\cite[Prop.~6.2.1 and Lem.~6.5.1]{AndreK3}  we obtain a $G_{\dRB,\overline \Q, \overline \Q}$-invariant algebra isomorphism
$$\xymatrix{
 (\End_{C^+}\Lambda)\otimes \overline \Q \ar[rr]^j \ar@{}[d]^{\rotatebox{90}{$\simeq$}} && (\End_{{C'}^+}\Lambda')\otimes \overline \Q \ar@{}[d]^{\rotatebox{90}{$\simeq$}} \\
C^+(V\otimes \overline \Q) \ar[rr]^{C^+(i)}&& C^+(V'\otimes \overline \Q)\,.
}$$
Now, still by \cite{AndreK3}, 
we have a canonical algebra isomorphism
$\End(\Lambda) \simeq (\End_{C^+}\Lambda) \otimes (C^+)^\op$ which is motivated.
We then get a $G_{\dRB,\overline \Q, \overline \Q}$-invariant algebra isomorphism
$$\xymatrix{
	 \End(\Lambda\otimes \overline \Q)  \ar[rr]^J \ar@{}[d]^{\rotatebox{90}{$\simeq$}} && \End(\Lambda'\otimes \overline \Q) \ar@{}[d]^{\rotatebox{90}{$\simeq$}} \\
((\End_{C^+}\Lambda)\otimes \overline \Q) \otimes (C^+\otimes \overline \Q)^\op \ar[rr]^{j\otimes j_0}&&((\End_{{C'}^+}\Lambda')\otimes \overline \Q) \otimes ({C'}^+\otimes \overline \Q)^\op\,,
}$$
 where $C^+\otimes \overline \Q$ and ${C'}^+\otimes \overline \Q$ are viewed as representations of $G_{\dRB,\overline \Q, \overline \Q}$ with trivial action.

 Our aim is to show that $J$ is motivated.
 By Lemma~\ref{lem:Endomorphism} below, there exists an isomorphism
 $\nu: \Lambda\otimes \overline \Q \rightarrow \Lambda'\otimes \overline \Q$ such that
$$
J(f)=\nu\circ f\circ \nu^{-1}
$$
for 
all $f\in \End_{\overline{\Q}}(\Lambda\otimes{ \overline\Q})$. The isomorphism 
$\nu$ is unique up to scaling by a scalar in $\overline\Q^*$.
 
 \begin{lem}\label{lem:Endomorphism}
 Let $W$ and $W'$ be two vector spaces over a field $k$ and let $\Phi: \End_k(W)\rightarrow \End_k(W')$ be an isomorphism of $k$-algebras. Then there exists a $k$-linear isomorphism $\phi: W\rightarrow W'$ such that $\Phi(f) = \phi\circ f\circ \phi^{-1}$ for all $f\in \End_k(W)$. The isomorphism $\phi$ is unique up to a scalar in $k^*$. 
 \end{lem}
 \begin{proof}
 Pick a basis $\{w_1, w_2, \ldots, w_n\}$ of $W$. Let $pr_1\in \End_k(W)$ be the projector onto the subspace generated by $w_1$. Let $pr'_1:=\Phi(pr_1)\in \End_k(W')$, which is nonzero. 
 Pick some $w'\in W'$ such that $w'_1:=pr'_1(w')$ is nonzero. Let $f_i\in\End_k(W)$, $i=2,3,\ldots,n$, be defined by $f_i(w_1)=w_i$ and $f_i(w_j)=0$ for all $j>1$.  Let $f'_i:=\Phi(f_i)$ and set $w'_i := f'_i(w'_1)$. 
 In this way, we get a basis $\{w'_1, w'_2,\ldots, w'_n\}$ of $W'$ and the isomorphism $\phi: W\rightarrow W'$, $\phi(w_i) = w'_i$, satisfies the required condition. 
 If $\phi'$ is another such isomorphism, then $\phi^{-1}\circ \phi'$ sits in the center of $\End_k(W)$ and hence $\phi'=c\phi$ for some $c\in k^*$.
 \end{proof}
 
The condition that $J$ is a $G_{\dRB,\overline \Q, \overline \Q}$-invariant isomorphism of algebras means
$$
\nu\circ g\circ f \circ g^{-1}\circ \nu^{-1} = g\circ \nu \circ f \circ \nu^{-1}\circ g^{-1}, \quad \forall g\in G_{\dRB,\overline \Q, \overline \Q},\; f\in \End(\Lambda\otimes \overline \Q).
$$
As a consequence $\lambda(g):=g^{-1}\circ \nu^{-1}\circ g \circ \nu$ commutes with all $f\in \End(\Lambda\otimes \overline \Q)$ and hence $\lambda(g)\in \overline\Q^*$. 
After choosing a $\overline\Q$-basis for $\Lambda\otimes {\overline \Q}$ and $\Lambda'\otimes {\overline \Q}$, we can write $\lambda(g) = YZY^{-1}Z^{-1}$ with~$Y, Z\in \mathrm{M}_{r\times r}(\overline\Q)$. 
By taking determinant, we see that $\lambda(g)\in\mu_r$. Since by Theorem~\ref{T:connectedness} the $\overline \Q$-de\,Rham--Betti group $G_{\dRB,\overline \Q, \overline \Q}$  
is connected, 
we conclude that $\lambda(g)=1$ for all $g\in G_{\dRB,\overline \Q, \overline \Q}$. As a consequence $\nu$ is $G_{\dRB,\overline \Q, \overline \Q}$-invariant.
It follows from Theorem~\ref{T:dRB1A-barQ} that $\nu$ is a $\overline \Q$-linear combination of classes of algebraic cycles. As a consequence, $J$ is $\overline \Q$-motivated. Note that the Kuga--Satake correspondence, by the definition of $J$, fits into the following commutative diagram
\[
\xymatrix{
V\otimes \overline \Q \ar[rr]^{i}\ar[d] &&V'\otimes \overline \Q \ar[d]\\
\End(\Lambda\otimes \overline \Q)\ar[rr]^{J} &&\End(\Lambda'\otimes \overline \Q)
}
\]
We conclude from
the fact that the Kuga--Satake correspondence $V \to \End(\Lambda)$ is motivated 
\cite{AndreK3} 
and from the semi-simplicity
that $i: V\otimes \overline \Q \rightarrow V'\otimes \overline \Q$ is $\overline \Q$-motivated.
\end{proof}

\begin{cor}
	If $X$ and $X'$ are hyper-K\"ahler varieties over $K$ of K3${[n]}$-deformation type, 
	then any de\,Rham--Betti isometry 
	$\HH^2_\dRB(X,\Q) \stackrel{\sim}{\longrightarrow} \HH^2_\dRB(X',\Q)$
 is algebraic.
\end{cor}
\begin{proof} By Markman~\cite{Markman-K3n}, any Hodge isometry $\HH^2_\B(X_\C^\an,\Q) \stackrel{\sim}{\longrightarrow} \HH^2_\B((X')_\C^\an,\Q)$ is algebraic. (The case where $X$ and $X'$ are K3 surfaces is due to Buskin~\cite{Buskin}).
The corollary then follows directly from Theorem~\ref{P:GdRBisometry}$(iii)$.
\end{proof}

\subsection{A global de\,Rham--Betti Torelli theorem for K3 surfaces over $\overline \Q$}
We derive from Theorem~\ref{P:GdRBisometry} the following result of independent interest\,:

\begin{thm}[A global de\,Rham--Betti Torelli theorem for K3 surfaces over $\overline \Q$]
	\label{T2:dRB-Torelli}
	Let $S$ and $S'$ be two K3 surfaces over $\overline{\Q}$. If there is an integral de\,Rham--Betti isometry
	$$i: \HH^2_\dRB(S,\Z) \stackrel{\sim}{\longrightarrow} \HH^2_\dRB(S',\Z),$$
	\emph{i.e.}, an isometry $i: \HH^2_\B(S_\C^\an,\Z) \stackrel{\sim}{\longrightarrow} \HH^2_\B(S'^\an_\C,\Z)$ that becomes de\,Rham--Betti after base-change to~$\Q$,
	then $S$ and $S'$ are isomorphic.
\end{thm}
\begin{proof}
	Let $\eta$ be an ample divisor class on $S$. 
	Then $\eta'=i(\eta)$ is a de\,Rham--Betti class and hence algebraic by \cite[Thm.~5.6]{BostCharles}\,; see also Proposition \ref{P:BC}. Thus either $\eta'$ or $-\eta'$ is a positive class on $S'$. 
	Without loss of generality, we assume that $\eta'$ is a positive class. 
	Then $\eta'$ can be moved to an ample class by a series of reflections along ($-2$)-classes\,; see \cite[Ch.~8, Cor.~2.9]{HuybrechtsK3book}. 
	Such reflections are all algebraic isometries of $\HH^2_\dRB(S',\Z)$ since each ($-2$)-class is represented by a rational curve on $S'$. By composing $i$ with these reflections, we may assume that $\eta'$ is an ample class. Thus $i$ restricts to a de\,Rham--Betti isometry
	$$
	i': \mathrm{P}^2_\dRB(S,\Z) \rightarrow \mathrm{P}^2_{\dRB}(S',\Z),
	$$
	where $\mathrm{P}^2_\dRB(S,\Z)$ (resp.\ $\mathrm{P}^2_\dRB(S',\Z)$) is the orthogonal complement of $\eta$ (resp.\ $\eta'$). 
	It follows from Theorem \ref{P:GdRBisometry} that $i' \otimes 1_\Q$ is motivated and hence respects Hodge structures. 
	As a consequence, $i$ is a Hodge isometry and the usual Torelli theorem for K3 surfaces provides an isomorphism $S_\C \simeq S'_\C$, from which we obtain an isomorphism $S\simeq S'$.
\end{proof}

\subsection{Codimension-2 de\,Rham--Betti classes on hyper-K\"ahler varieties}
Recall  that for all Andr\'e motives $M$ over $\overline \Q$,  $G_\And(M)$ is reductive and that
we have $G_{\dRB}(M\otimes {\overline \Q}) \subseteq G_{\dRB}(M)_{\overline \Q} \subseteq G_\And(M)_{\overline \Q}$. 
Recall also from  Proposition~\ref{P:redX} that $G_{\dRB}(\h^2(X)\otimes \overline \Q)$ is reductive if $X$ is a hyper-K\"ahler variety over $\overline \Q$.
The following proposition is analogous to \cite[Lem.~7.4.1]{AndreK3}.

\begin{prop}\label{P:dRB-tensor2}
	Let $X$ be a hyper-K\"ahler variety over ${\overline \Q}$. Then the inclusion 
	$$\End_{G_{\And}(\h^2(X))} (\HH^2_\B(X^\an_\C,{\overline \Q})) \subseteq \End_{G_{\dRB}(\h^2(X)\otimes \overline \Q)} (\HH^2_\B(X^\an_\C,{\overline \Q}))$$ 
	is an equality.
Equivalently, any ${\overline \Q}$-de\,Rham--Betti class on $\h^2(X)(1) \otimes \h^2(X)(1)$ is ${\overline \Q}$-motivated. 
In particular, if $X$ is  a hyper-K\"ahler variety over $K$, then any $L$-de\,Rham--Betti class on $\h^2(X)(1) \otimes \h^2(X)(1)$ is $L$-motivated. 
\end{prop}
\begin{proof} 
	The two statements in the proposition are equivalent  since any choice of polarization on~$X$ induces an isomorphism of Andr\'e motives $\h^2(X)(1)^\vee \simeq \h^2(X)(1)$, while the statement about $L$-de\,Rham--Betti classes follows from Lemma~\ref{L:descent}.

Since $G_\And(\h^2(X))$ and $G_{\dRB}(\h^2(X)\otimes \overline \Q)$ are reductive, it is enough
to show that any simple $G_{\dRB,\overline \Q, \overline \Q}$-submodule $T$ of $\HH^2_\B(X^\an_\C,{\overline \Q})$ is a  $(G_{\And})_{\overline \Q}$-submodule.
Since the intersection form on $\HH^2_\B(X^\an_\C,\Q)$ is motivated,
$T\cap T^\perp$ is a $G_{\dRB,\overline \Q, \overline \Q}$-submodule of $T$ and it is therefore either $\{0\}$ or~$T$.

If $T\cap T^\perp = \{0\}$, then  $\HH^2_\B(X^\an_\C,\overline \Q) = T\oplus^\perp T^\perp$ and $(-\mathrm{id}_T, \mathrm{id}_{T^\perp})$ 
defines a $G_{\dRB,\overline \Q, \overline \Q}$-equivariant isometry of $\HH^2_\B(X^\an_\C,\overline \Q)$. By Theorem~\ref{P:GdRBisometry} 
it is motivated and hence $T$ is a $(G_{\And})_{\overline \Q}$-submodule.

If $T\cap T^\perp = T$, i.e. if $T$ is totally isotropic, then choose a  $G_{\dRB,\overline \Q, \overline \Q}$-submodule $T'$ of $\HH^2_\B(X^\an_\C,\overline \Q)$ such that $\HH^2_\B(X^\an_\C,\overline \Q) = T^\perp \oplus T'$. 
The restriction of the quadratic form to the  $G_{\dRB,\overline \Q, \overline \Q}$-submodule $T\oplus T'$ is then nondegenerate. 
Let $\psi : T\oplus T' \to T^\vee \oplus T'^\vee$ be the induced $G_{\dRB,\overline \Q, \overline \Q}$-equivariant map and let $\psi_T : T' \to T^\vee$ and $\psi_{T'} : T' \to T'^\vee$ be the induced  $G_{\dRB,\overline \Q, \overline \Q}$-equivariant maps. 
Since $T$ is totally isotropic, the map $\psi_T$ is an isomorphism. Define the $G_{\dRB,\overline \Q, \overline \Q}$-submodule $T'' := \mathrm{im}\big(-\frac{1}{2}(\psi_T^{-1})^\vee\circ \psi_{T'} \oplus \mathrm{id}_{T'} : T' \to T \oplus T' \big)$. 
Then $T\oplus T' = T\oplus T''$ is a decomposition of the nondegenerate quadratic space $T\oplus T'$ into $G_{\dRB,\overline \Q, \overline \Q}$-submodules such that both $T$ and $T''$ are totally isotropic.
It follows that $(2\mathrm{id}_T, \frac{1}{2}\mathrm{id}_{T''}, \mathrm{id}_{(T\oplus T'')^\perp})$ defines a $G_{\dRB,\overline \Q, \overline \Q}$-equivariant isometry of $\HH^2_\B(X^\an_\C,\overline \Q)$. By Theorem~\ref{P:GdRBisometry} it is ${\overline \Q}$-motivated and hence $T$ is a $G_{\And,{\overline \Q}}$-submodule.
\end{proof}

\begin{thm}
	\label{T:dRB-cod2}
	Let $X$ be a hyper-K\"ahler variety over $K$ of known deformation type and let $n$ be a positive integer. Then any $L$-de\,Rham--Betti class 
on $\h^4(X^n)(2)$ is $L$-motivated.
\end{thm}
\begin{proof}
By Lemma~\ref{L:descent}, 
we can assume $K={\overline \Q}$. By~\cite[Prop.~2.35(ii) \& Prop.~2.38(i)]{GKLR}, for any hyper-K\"ahler variety $X$,
$\HH^4_\B(X^\an_\C,\Q)$ is a sub-Hodge structure of $\big(\HH_\B^2(X^\an_\C,\Q)\otimes \HH_\B^2(X^\an_\C,\Q) \oplus \HH_\B^2(X^\an_\C,\Q(-1)) \oplus \Q(-2) \big)^{\oplus N}$ for some~$N$. 
Since $\HH_{\B}^1(X_\C^\an,\Q)=0$ and since $\Q(-1)$ is a direct summand of $\HH_\B^2(X^\an_\C,\Q) $, we see that for all $n\geq 1$,
 $\HH^4_\B((X^\an_\C)^n,\Q)$ is a sub-Hodge structure of $\big(\HH_\B^2(X^\an_\C,\Q)\otimes \HH_\B^2(X^\an_\C,\Q)\big)^{\oplus N}$ for some~$N$. 
 From \cite[Cor.~1.17]{FFZ} and \cite[Cor.~1.2]{Soldatenkov}, if $X$ is of known deformation type, any Hodge class on a power of $X_\C$ is motivated\,;
it follows that the Andr\'e motive $\h^4(X^n)$ is a direct summand of $(\h^2(X)\otimes \h^2(X))^{\oplus N}$. 
 Now, having proved in Proposition~\ref{P:dRB-tensor2} that any ${\overline \Q}$-de\,Rham--Betti class in $\h^2(X)(1)\otimes \h^2(X)(1)$ is ${\overline \Q}$-motivated, we deduce that any ${\overline \Q}$-de\,Rham--Betti class in $\h^4(X^n)(2)$ is ${\overline \Q}$-motivated.
\end{proof}

\begin{cor}
Let $X$ be a hyper-K\"ahler fourfold over $K$ of known deformation type and let $k$ be a non-negative integer. Then any $L$-de\,Rham--Betti class on $\h^{2k}(X)(k)$ is $L$-motivated.
\end{cor}
\begin{proof}
From Proposition~\ref{P:BC} and Theorem~\ref{T:dRB-cod2}, it remains to see that any $L$-de\,Rham--Betti class 
on $\h^6(X)(3)$
is $L$-motivated. This follows from Proposition~\ref{P:BC} together with the Lefschetz isomorphism of Andr\'e motives $\h^2(X)(1) \simeq \h^6(X)(3)$.
\end{proof}

\subsection{The de\,Rham--Betti group of a hyper-K\"ahler variety.}
	Let $X$ be a hyper-K\"ahler variety over $\overline\Q$ and denote by $\mathfrak t^2(X)$ its transcendental motive in degree 2. By Zarhin~\cite{Zarhin}, the endomorphism algebra $E:=\End(\mathfrak t^2(X))$ is a field of degree dividing $\dim_{\Q}  \mathrm{T}^2_\B(X,\Q)$. 
	We say that $X$ has complex multiplication if $[E:\Q]= \dim_\Q \mathrm{T}^2_\B(X,\Q)$ (this implies $E$ is a CM field).  
 With what we have established so far, we have the following analogue of Theorem~\ref{T:abelian-dRB}.

\begin{thm}\label{T:hK-dRB}
	Let $X/\overline\Q$ be a hyper-K\"ahler variety of positive dimension. The de\,Rham--Betti group $G_\dRB(\h^2(X)\otimes L)$ has the following properties.
	\begin{enumerate}[(i)]
		\item $G_{\dRB}(\h^2(X)\otimes L)$ is a connected reductive subgroup of $\mathrm{MT}(\h^2(X))_L$.
		\item $\End_{G_{\dRB}(\h^2(X)\otimes L)}(\HH^2_{\B}(X,L)) = \End_{\operatorname{MT}(\h^2(X))}(\HH^2_{\B}(X,\Q))\otimes L$.
		\item $\det\colon G_{\dRB}(\h^2(X)\otimes L) \rightarrow \mathds{G}_m$ is surjective.
		\item $X$ has complex multiplication if and only if $G_\dRB(\h^2(X)\otimes L)$ is a torus. 
	\end{enumerate}
\end{thm}
\begin{proof}
	That $G_{\dRB}(\h^2(X)\otimes L)$ lies in  $\mathrm{MT}(\h^2(X))_L$ follows from the inclusion  $G_{\dRB}(\h^2(X)\otimes L) \subseteq (G_{\dRB}(\h^2(X)))_L$  and from Andr\'e's Theorem \ref{T:simpleKS} stating that 
	the inclusion $\mathrm{MT}(\h^2(X)) \subseteq G_\And(\h^2(X))$ is an equality. 
	That $G_{\dRB}(\h^2(X)\otimes L)$ is connected is Theorem~\ref{T:connectedness} and that it is reductive is Proposition~\ref{P:redX}.
	Statement $(ii)$ is Proposition~\ref{P:dRB-tensor2}. Regarding $(iii)$,  the image of $\det$ is connected. 
	Assume it is trivial. Then $G_\dRB(\h^2(X)\otimes L)$ acts
	trivially on $\det(\HH^2_\B(X, L))$. But, by Andr\'e's Theorem \ref{T:simpleKS}, $\det \h^{2}(X) \simeq \mathds{1}(-\dim \HH^2_\B(X,\Q))$ as Andr\'e motives. 
	This is a contradiction since $2\pi i$ is transcendental.
	For $(iv)$, suppose that the transcendental motive $\mathfrak t^2(X)$ has complex multiplication. 
	Since $G_{\dRB}(\h^2(X)\otimes L)$ is a reductive and connected subgroup of  $\mathrm{MT}(\h^2(X))_L$, which is a torus, it has to be a torus.
	Conversely, assume that $G_{\dRB}(\h^2(X)\otimes L)$ is a torus. 
Similarly to the case of abelian varieties, if
	$G_{\dRB}(\h^2(X)\otimes L)$ is contained in a maximal torus $T \subseteq \operatorname{GL}(\mathrm{T}^2_\B(X,L))$, 
	then $$\End_{T}(\mathrm{T}^2_{\B}(X,L)) \subseteq \End_{G_{\dRB}(\h^2(X)\otimes L)}(\mathrm{T}^2_{\B}(X,L)) =\End_{\operatorname{MT}(\h^2(X))_L}(\mathrm{T}^2_{\B}(X,L)) .$$
	But $\End_{T}(\mathrm{T}^2_{\B}(X,L)) $ is a commutative $L$-algebra of dimension $\dim_L \mathrm{T}^2_{\B}(X,L)$, as can be seen after extending scalars to an algebraically closed field. It follows that $X$ has complex multiplication.
\end{proof}

\begin{cor}\label{C:T2}
	Let $X$ be a hyper-K\"ahler variety over $\overline \Q$ and let $k\in \Z$.
	Every de\,Rham--Betti class on $\mathfrak{t}^2(X)(k)$ is zero. In particular, if $X$ is a hyper-K\"ahler variety over $K$ and $k$ is an integer, then
	every de\,Rham--Betti class on $\mathfrak{h}^2(X)(k)$ is algebraic.
\end{cor}
\begin{proof}
	It follows from Theorem \ref{T:hK-dRB}$(ii)$ that 
	$$ \End_{G_{\dRB}(\h^2(X))}(\mathrm{T}_\B^{2}(X,\Q)) = \End_{\mathrm{MT}(\h^2(X))}(\mathrm{T}_\B^{2}(X,\Q)).$$
	In particular, since the Hodge structure $\mathrm{T}_\B^{2}(X,\Q)$ is irreducible, the de\,Rham--Betti object $\mathrm{T}^{2}_\dRB(X,\Q)$ is simple.
	Since $\dim \mathrm{T}^2_\B(X,\Q) \geq 2$, we get 
	 $\Hom_{\dRB}(\mathds 1(-k), \mathrm{T}^{2}_\dRB(X,\Q)) = 0,$
	 for all $k\in \Z$.
\end{proof}

\begin{rmk}
	We do not know how to prove that any $\overline \Q$-de\,Rham--Betti class on $\mathfrak{t}^2(X)(k)$ is zero for $X$ a hyper-K\"ahler variety over $\overline \Q$.
\end{rmk}

\subsection{The de\,Rham--Betti conjecture and hyper-K\"ahler varieties of large Picard rank} 
Let us introduce the following terminology\,:

\begin{defn}
	Let $X$ be a smooth projective variety over a subfield of $\C$.
	\begin{itemize}
		\item The \emph{Picard rank} of $X$ is $\rho(X) =_{\mathrm{def}} \rk (\CH^1(X_\C) \to \HH^2_\B(X_\C^\an,\Q(1)))$.
		\item The \emph{Picard corank} of $X$ is $\rho^c(X) =_{\mathrm{def}} \h^{1,1}(X_\C^\an) - \rho(X)$, where $\h^{1,1}(X_\C^\an) = \dim_\C \HH^{1}(X_\C^\an,\Omega_{X_\C^\an}^1)$.
	\end{itemize}
\end{defn}

We address the de\,Rham--Betti conjecture (Conjectures~\ref{C:dRB}) for the degree two motive  (possibly with $\overline \Q$-coefficients) of hyper-K\"ahler varieties of Picard corank $\leq 2$ and prove  Theorem~\ref{T:main-GPC-CM}.

	\begin{thm}\label{T:main-GPC-maxPic}
	Let $X$ be a hyper-K\"ahler variety over $K$.
	\begin{enumerate}[(i)]
		\item 	If $\rho^c(X)=0$, then $Z_{\h^2(X)} = \Omega^\And_{\h^2(X)}$.
		\item If $\rho^c(X)=1$, then $\Omega_{\h^2(X)} = \Omega^\And_{\h^2(X)}.$
		\item If $\rho^c(X)=2$, then $\Omega^\dRB_{\h^2(X)} = \Omega^\And_{\h^2(X)}.$
	\end{enumerate}
	Assume further that $X$ is of known deformation type.
	\begin{enumerate}[(i')]
		\item 	If $\rho^c(X)=0$, then $Z_{X} = \Omega^\And_{X}$.
		\item If $\rho^c(X)=1$, then $\Omega_{X} = \Omega^\And_{X}.$
		\item If $\rho^c(X)=2$, then $\Omega^\dRB_{X} = \Omega^\And_{X}.$
	\end{enumerate}
\end{thm}

\begin{proof}
Regarding $(i)$\,:
By Andr\'e's Theorem~\ref{T:simpleKS}, $G_\And(\h^2(X_{\overline \Q}))$ agrees with the Mumford--Tate group and hence is connected. By a Galois argument as in the proof of Lemma~\ref{L:torsor-descent},
we deduce that $\Omega^\And_{\h^2(X)}$ is connected. Therefore it suffices to show statement $(i)$ in the case $K=\overline \Q$.
Note that the Kuga--Satake variety associated with $X$ of maximal Picard rank is a CM elliptic curve, so that we have an isomorphism of Andr\'e motives  
$$\h^2(X) \simeq \big( \h^1(E)\otimes \h^1(E)\big) \oplus \mathds{1}(-1)^{\oplus b_2-4}.$$
To conclude it suffices to see that $\dim Z_{\h^2(X)} \geq 2$. This
follows at once from Chudnovsky's theorem~\cite{Chudnovsky} that the transcendence degree of periods of elliptic curves is at least 2. 
\smallskip

Regarding $(ii)$\,: By
Lemma~\ref{L:torsor-descent}, we may and do assume that $K=\overline \Q$.
 Note that if the Picard corank of $X$ is at most 1, then the Kuga--Satake variety $A$ associated with the transcendental motive $\mathfrak{t}^2(X)$ is either a CM elliptic curve or an abelian surface of Picard rank 3. 
Since $\mathfrak t^2(X)$ is  a direct summand of $\h(A\times A)$, we can conclude from Proposition~\ref{P:directfactor} together with Theorems~\ref{T:dRB-elliptic} and~\ref{T:abelian-surfaces}.
\smallskip

Regarding $(iii)$\,: 
By Lemma~\ref{L:torsor-descent}, we may and do assume that $K=\overline \Q$.
Let $E=\End(\mathfrak t^2(X))$. By Zarhin~\cite{Zarhin}, $E$ is a field and there are three possibilities for the Mumford--Tate group of $\h^2(X)(1)$\,: it is either $\mathrm{SO}_4$, or  $\mathrm{U}_2(E)$ for the CM quadratic extension $E/\Q$,  or	$\mathrm{Res}_{F/\Q}\mathrm{U}_1(F)$ for the CM quartic extension $E/F/\Q$, where $F$ is a real quadratic extension.
(Note that the case $\mathrm{Res}_{E/\Q}\mathrm{SO}_2$ for a real quadratic extension $E/\Q$ does not occur since in that case the  Hodge group of $\h^2(X)$ is commutative and hence $\h^2(X)$ would be CM.)

First we assume that $\operatorname{MT}(\h^2(X)(1)) = \operatorname{SO}_{4,\Q}$. In that case, we actually show the stronger statement that the inclusion $G_{\dRB}(\h^2(X)(1)\otimes \overline \Q) \subseteq G_\And(\h^2(X)(1))_{\overline \Q}  =  \operatorname{SO}_{4,\overline \Q}$ is an equality. 
Let $\mathfrak g$ be the Lie algebra of $G_{\dRB}(\h^2(X)(1)\otimes \overline \Q)$. 
We then have an inclusion 
$$\mathfrak g \subseteq \mathfrak{so}_{4,\overline \Q} = \mathfrak{sl}_{2,\overline \Q} \times \mathfrak{sl}_{2,\overline \Q},$$
 where  
the $\mathfrak g$-module structure  on $\mathrm{T}^2_\B(X,\overline \Q(1))$ is induced from the $\mathfrak{sl}_{2,\overline \Q} \times \mathfrak{sl}_{2,\overline \Q}$-module decomposition $\mathrm{T}^2_\B(X,\overline \Q(1)) = V_1 \otimes V_2$, where $V_i$ is the standard representation of the $i$-th factor of $\mathfrak{sl}_{2,\overline \Q} \times \mathfrak{sl}_{2,\overline \Q}$.
Now, by using Theorem~\ref{T:hK-dRB} instead of Theorem~\ref{T:abelian-dRB}, we may argue exactly as in the proof of Theorem~\ref{T:abelian-surfaces}$(ii)$ in the case of simple abelian surfaces with endomorphism algebra given by a real quadratic extension of $\Q$ to show that the inclusion $\mathfrak g \subseteq \mathfrak{so}_{4,\overline \Q} = \mathfrak{sl}_{2,\overline \Q} \times \mathfrak{sl}_{2,\overline \Q}$ is an equality.

Suppose $\mathrm{MT}(\h^2(X)(1)) = \mathrm{U}_2(E)$ for a CM quadratic extension $E$ of $\Q$, and let $\mathfrak{mt}$ denote the Lie algebra of $\mathrm{MT}(\h^2(X)(1))$.
Then $\mathfrak{mt}_{\overline \Q} = \mathfrak{gl}_{2, \overline \Q} = \mathfrak{sl}_{2,\overline \Q} \oplus \mathfrak c$ and the embedding
\begin{equation}\label{e:action} \mathfrak{mt}_{\overline \Q} \subset \mathfrak{sl}_{2, \overline \Q} \oplus \mathfrak{sl}_{2, \overline \Q} \end{equation} 
corresponding to the embedding $\mathrm{MT}(\h^2(X)(1)) \subset \mathrm{SO}_4$ is given by mapping $\mathfrak{sl}_{2, \overline \Q}$ to one of the factors and $\mathfrak{c}$ to the Lie algebra of a maximal torus in $\mathrm{SL}_{2,\overline \Q}$ in the other factor.
Here the action of $\mathfrak{sl}_{2,\overline \Q} \oplus \mathfrak{sl}_{2,\overline \Q}$ on $\mathrm{T}_\B^2(1)\otimes \overline \Q$ is via the tensor product of two copies of the standard representation.
Let $\mathfrak{g} \subseteq \mathfrak{sl}_{2,\overline \Q}\oplus \mathfrak{c}$ denote the Lie algebra of $G_{\overline \Q-\dRB}(\h^2(X)(1))$.
If $\mathfrak{g}$ was abelian, then $G_{\dRB}(\h^2(X)(1)\otimes \overline \Q)$ would be a torus, contradicting by Theorem~\ref{T:hK-dRB} the fact that $\End(\h^2(X)(1))=E$.
Hence $\mathfrak{g}$ has to contain~$\mathfrak{sl}_{2,\overline \Q}$.
We have to exclude the case $\mathfrak{g} = \mathfrak{sl}_{2,\overline \Q}$.
In this case the embedding (\ref{e:action}) shows that $\End_{G_{\dRB}(\mathfrak{h}^2(X)(1)\otimes \overline \Q)}(\mathrm{T}_\B^2(X,\overline \Q(1))$ has dimension~$4$,
 which contradicts Theorem~\ref{T:hK-dRB}$(ii')$.

Suppose $\mathrm{MT}(\h^2(X)(1)) = \mathrm{Res}_{F/\Q}\mathrm{U}_1(F)$ for a CM quartic extension $E/\Q$, where $F\subset E$ denotes the real quadratic subfield.
We distinguish two cases\,: 
if $E$ does not contain an imaginary quadratic subfield, then the proof of Theorem \ref{T:abelian-surfaces}$(i)$ in case $(d)$ shows that $\mathrm{MT}(\h^2(X)(1))$ does not contain any non-trivial subtori defined over $\Q$.
Since $G_{\dRB}(\h^2(X)(1))$ cannot be trivial by Proposition~\ref{P:dRB-tensor2}, we conclude that $G_{\dRB}(\h^2(X)(1)) = \mathrm{MT}(\h^2(X)(1))$.
Suppose now that $E$ contains an imaginary quadratic subfield.
Then $E$ is a biquadratic field, and hence $\Gal(E/\Q) = \Z/2\Z \times \Z/2\Z$.
As we have seen in the proof of Theorem \ref{T:abelian-surfaces}$(i)$ in case $(d)$,
the cocharacter lattice of $\mathrm{Res}_{F/\Q}\mathrm{U}_{1}(F)$ is given by
$X_*(\mathrm{Res}_{F/\Q}\mathrm{U}_{1}(F))_{\Q} = \langle \phi_1^{\vee} - \bar{\phi}_1^{\vee}, \phi_2^{\vee} - \bar{\phi}_2^{\vee} \rangle_{\Q}.$
Here $\Sigma_E := \Hom(E, \C) = \{\phi_1, \bar{\phi}_1, \phi_2, \bar{\phi}_2\}$ is the set of embeddings of the CM field $E$ into the complex numbers.
One computes that $X_*(\mathrm{Res}_{F/\Q}\mathrm{U}_{1}(F))_{\Q}$ has two one-dimensional Galois-stable subspaces, generated by $\phi_1^{\vee} + \phi_2^{\vee} - (\bar{\phi}_1^{\vee} + \bar{\phi}_2^{\vee})$ and $\phi_1^{\vee} - \phi_2^{\vee} - (\bar{\phi}_1^{\vee} - \bar{\phi}_2^{\vee})$, respectively.
The corresponding $\Q$-subtori of rank $1$ are precisely the tori of the form $\mathrm{U}_{1}(K)$ for $K \subset E$ an imaginary quadratic subfield.
We have to exclude the possibility that $G_{\dRB}(\h^2(X)(1)) = \mathrm{U}_{1}(K)$.
If this is the case, then $\End_{K}(\mathrm{T}^2_\B(X,\Q)) \subset \End_{G_{\dRB}(\h^2(X)(1))}(\mathrm{H}^2_\B(X,\Q))$.
Note that $\End_{K}(\mathrm{T}^2_\B(X,\Q))$ is of dimension~$4$ over~$K$, and hence of dimension~$8$ over $\Q$. This contradicts Theorem~\ref{T:hK-dRB}$(ii)$.

Now that we have established in all cases that $G_{\dRB}(\h^2(X)(1)) = \mathrm{MT}(\h^2(X)(1))$, let us prove $G_{\dRB}(\h^2(X))= \mathrm{MT}(\h^2(X))$.
We first show that $\mathrm{MT}(\h^2(X)(1)) \subset G_{\dRB}(\h^2(X))$. Namely, since we know that both groups are reductive, we can argue as follows\,:
Suppose $v \in \HH_\B^2(X)^{\otimes a} \otimes (\HH_\B^2(X)^\vee)^{ \otimes b}$ is fixed by $G_{\dRB}(\h^2(X))$.
Since $$\HH_\B^2(X)(1)^{\otimes a} \otimes (\HH_\B^2(X)(1)^\vee)^{ \otimes b} = \HH_\B^2(X)^{\otimes a} \otimes (\HH_\B^2(X)^\vee)^{ \otimes b} \otimes \Q(a-b),$$
we see that the one-dimensional subspace spanned by $v$ is preserved by $G_{\dRB}(\h^2(X)(1))$.
The equality $G_{\dRB}(\h^2(X)(1)) = \mathrm{MT}(\h^2(X)(1))$ then shows that this subspace is a one-dimensional sub-Hodge structure of $\HH_\B^2(X)(1)^{\otimes a} \otimes (\HH_\B^2(X)(1)^\vee)^{ \otimes b}$.
It therefore has to be of weight zero, and consequently $v$ is fixed by $\mathrm{MT}(\h^2(X)(1))$.
Since $\mathrm{MT}(\h^2(X))$ is generated by the scalars $\mathds{G}_m$ and $\mathrm{MT}(\h^2(X)(1))$,
and we know from Theorem~\ref{T:hK-dRB} that $G_{\dRB}(\mathfrak{h}^2(X))$ surjects onto the determinant, we conclude that $G_{\dRB}(\mathfrak{h}^2(X))= \mathrm{MT}(\mathfrak{h}^2(X))$.
\medskip

Finally, we show how to derive $(i')$, $(ii')$ and $(iii')$ from $(i)$, $(ii)$ and $(iii)$, respectively. As before, we may and do assume that $K=\overline \Q$.
Since $\h^2(X)$ is a direct summand of $\h(X)$ we have a commutative diagram with surjective horizontal arrows
$$\xymatrix{
	G_\And(X)  \ar@{->>}[r] & G_\And(\h^2(X))  \\
	\mathrm{MT}(X) \ar@{^(->}[u] \ar@{->>}[r] & \mathrm{MT}(\h^2(X))  \ar@{^(->}[u]
 }$$
By Theorem~\ref{T:simpleKS}, the right inclusion is an equality, while the assumption that $X$ is of known deformation type provides by \cite[Thm.~1.11 \& Cor.~1.16]{FFZ} that the left inclusion is an equality. 
On the other hand, by \cite[Prop.~6.4]{FFZ} the bottom horizontal arrow has finite kernel.

In particular, $G_\And(X)$ is connected and $\dim G_\And(X)  = \dim G_\And(\h^2(X))$. Since clearly $\dim Z_X \geq \dim Z_{\h^2(X)}$, we get the implication $(i)\Rightarrow (i')$.

On the other hand, we have commutative diagrams with surjective horizontal arrows
$$\xymatrix{G_\dRB(X) \ar@{^(->}[d] \ar@{->>}[r] & G_\dRB(\h^2(X))  \ar@{^(->}[d] 
	&& G_{\dRB}(\h(X)\otimes \overline \Q) \ar@{^(->}[d] \ar@{->>}[r] & G_{\dRB}(\h^2(X)\otimes \overline \Q)  \ar@{^(->}[d] \\
	G_\And(X)  \ar@{->>}[r] & G_\And(\h^2(X)) 
	&&	G_\And(X)_{\overline \Q}  \ar@{->>}[r] & G_\And(\h^2(X))_{\overline \Q} 
}$$
 and, as explained above, the surjection $	G_\And(X) \twoheadrightarrow G_\And(\h^2(X)) $ is an isogeny of connected algebraic groups. 
Therefore, if the inclusion $G_\dRB(\h^2(X)) \hookrightarrow G_\And(\h^2(X)) $ (resp.\ the inclusion  $G_\dRB(\h^2(X)\otimes \overline \Q) \hookrightarrow G_\And(\h^2(X))_{\overline \Q}$) is an equality, then the inclusion 
 $G_\dRB(X) \subseteq G_{\And}(X)$ (resp.\ the inclusion 
 $G_{\dRB}(\h(X)\otimes \overline \Q) \subseteq G_{\And}(X)_{\overline \Q}$) is an equality. 
 This establishes the implication $(iii)\Rightarrow (iii')$ (resp.\ the implication $(ii) \Rightarrow (ii')$).
\end{proof}

\begin{rmk} 
	Let $X$ be a hyper-K\"ahler variety over $K$. Assume that $X$ does not have CM and that  $\rho^c(X)=2$.
	The arguments of the proof of Theorem~\ref{T:main-GPC-maxPic} actually show that  $\Omega_{\h^2(X)} = \Omega^\And_{\h^2(X)}$. Moreover, if $X$ is of known deformation type, then $\Omega_{X} = \Omega^\And_{X}$.
\end{rmk}

Using the Shioda--Inose structure on K3 surface of Picard corank $\leq 1$ and the validity of the Hodge conjecture for powers of abelian surfaces, we obtain\,:

\begin{cor}\label{C:K3-2} 
	Let $S$ be a K3 surface over $K$.
	\begin{enumerate}[(i)]
		\item If $\rho^c(S)=0$, then $Z_S = \Omega^\mot_S$.
		\item If $\rho^c(S)=1$, then $\Omega_S = \Omega^\mot_S$.
	\end{enumerate}
	In particular, in both cases, for any $n\geq 0$ and any $k\in \Z$, any $\overline \Q$-de\,Rham--Betti class on $\h(S^n)(k)$ is a $\overline \Q$-linear combination of algebraic classes.
\end{cor}
\begin{proof} On the one hand,
Theorem~\ref{T:main-GPC-maxPic} provides $Z_S = \Omega^\And_S$ and $\Omega_{S} = \Omega^\And_{S}$ if $\rho^c(S)=0$ and $\rho^c(S)=1$, respectively. We are thus left to show that  $\Omega^\And_{S} = \Omega^\mot_{S}$ if $\rho^c(S)\leq 1$. Since the standard conjectures hold for surfaces in characteristic zero, we have to show that motivated classes on powers of $S$ are algebraic. For that purpose, we may and do assume that $K=\overline \Q$.
If $S$ has Picard corank $\leq 1$, then it has a Shioda--Inose structure and its transcendental (homological) motive is thus isomorphic to the transcendental motive of an abelian surface over $\overline \Q$. 
Since Hodge classes on powers of complex abelian surfaces are algebraic \cite{Tankeev}, it follows that motivated classes on powers of $S$ are algebraic. 
\end{proof}

\begin{rmk}
	Corollary~\ref{C:K3-2}$(i)$ is also established in
\cite{Kawabe} in the case of Kummer surfaces associated with squares of CM elliptic curves.
\end{rmk}


\bibliographystyle{amsalpha}
\bibliography{bib}

\providecommand{\bysame}{\leavevmode\hbox to3em{\hrulefill}\thinspace}
\providecommand{\MR}{\relax\ifhmode\unskip\space\fi MR }
\providecommand{\MRhref}[2]{%
  \href{http://www.ams.org/mathscinet-getitem?mr=#1}{#2}
}
\providecommand{\href}[2]{#2}
\begin{thebibliography}{GKLR22}

\bibitem[AF25]{Ancona-Fratila}
Giuseppe Ancona and {Fr\u a\c til\u a, Drago\c s}, \emph{Algebraic classes in
  mixed characteristic and {A}ndr\'e's {$p$}-adic periods}, J. Inst. Math.
  Jussieu \textbf{24} (2025), no.~4, 1093--1138. \MR{4922232}

\bibitem[Anc21]{Ancona}
Giuseppe Ancona, \emph{Standard conjectures for abelian fourfolds}, Invent.
  Math. \textbf{223} (2021), no.~1, 149--212. \MR{4199442}

\bibitem[And96a]{AndreK3}
Yves Andr\'{e}, \emph{On the {S}hafarevich and {T}ate conjectures for
  hyper-{K}\"{a}hler varieties}, Math. Ann. \textbf{305} (1996), no.~2,
  205--248. \MR{1391213}

\bibitem[And96b]{AndreIHES}
\bysame, \emph{Pour une th\'{e}orie inconditionnelle des motifs}, Inst. Hautes
  \'{E}tudes Sci. Publ. Math. (1996), no.~83, 5--49. \MR{1423019}

\bibitem[And04]{AndreBook}
\bysame, \emph{Une introduction aux motifs (motifs purs, motifs mixtes,
  p\'{e}riodes)}, Panoramas et Synth\`eses [Panoramas and Syntheses], vol.~17,
  Soci\'{e}t\'{e} Math\'{e}matique de France, Paris, 2004. \MR{2115000}

\bibitem[And09]{Andre}
\bysame, \emph{Galois theory, motives and transcendental numbers},
  Renormalization and {G}alois theories, IRMA Lect. Math. Theor. Phys.,
  vol.~15, Eur. Math. Soc., Z\"{u}rich, 2009, pp.~165--177. \MR{2588609}

\bibitem[And25]{Andre-observability}
Yves Andr\'e, \emph{On the observability of {G}alois representations and the
  {T}ate conjecture}, J. Number Theory \textbf{270} (2025), 6--16. \MR{4875524}

\bibitem[BC16]{BostCharles}
Jean-Beno\^{\i}t Bost and Fran\c{c}ois Charles, \emph{Some remarks concerning
  the {G}rothendieck period conjecture}, J. Reine Angew. Math. \textbf{714}
  (2016), 175--208. \MR{3491887}

\bibitem[Bos13]{Bost}
Jean-Beno\^{\i}t Bost, \emph{Algebraization, transcendence, and {$D$}-group
  schemes}, Notre Dame J. Form. Log. \textbf{54} (2013), no.~3-4, 377--434.
  \MR{3091663}

\bibitem[Bro17]{Brown}
Francis Brown, \emph{Notes on motivic periods}, Commun. Number Theory Phys.
  \textbf{11} (2017), no.~3, 557--655. \MR{3713352}

\bibitem[Bus19]{Buskin}
Nikolay Buskin, \emph{Every rational {H}odge isometry between two {$K3$}
  surfaces is algebraic}, J. Reine Angew. Math. \textbf{755} (2019), 127--150.
  \MR{4015230}

\bibitem[Chu80]{Chudnovsky}
G.~V. Chudnovsky, \emph{Algebraic independence of values of exponential and
  elliptic functions}, Proceedings of the {I}nternational {C}ongress of
  {M}athematicians ({H}elsinki, 1978), Acad. Sci. Fennica, Helsinki, 1980,
  pp.~339--350. \MR{562625}

\bibitem[Clo99]{Clozel}
L.~Clozel, \emph{Equivalence num\'erique et \'equivalence cohomologique pour
  les vari\'et\'es ab\'eliennes sur les corps finis}, Ann. of Math. (2)
  \textbf{150} (1999), no.~1, 151--163. \MR{1715322}

\bibitem[Del82]{Deligne}
Pierre Deligne, \emph{Hodge cycles on abelian varieties, in \emph{Hodge cycles,
  motives, and {S}himura varieties}}, Lecture Notes in Mathematics, vol. 900,
  Springer-Verlag, Berlin-New York, 1982. \MR{654325}

\bibitem[DM82]{DeligneMilne}
Pierre Deligne and James Milne, \emph{Tannakian categories, in \emph{Hodge
  cycles, motives, and {S}himura varieties}}, Lecture Notes in Mathematics,
  vol. 900, Springer-Verlag, Berlin-New York, 1982. \MR{654325}

\bibitem[FFZ21]{FFZ}
Salvatore Floccari, Lie Fu, and Ziyu Zhang, \emph{On the motive of
  {O}'{G}rady's ten-dimensional hyper-{K}\"{a}hler varieties}, Commun. Contemp.
  Math. \textbf{23} (2021), no.~4, Paper No. 2050034, 50. \MR{4237568}

\bibitem[GKLR22]{GKLR}
Mark Green, Yoon-Joo Kim, Radu Laza, and Colleen Robles, \emph{The {LLV}
  decomposition of hyper-{K}\"{a}hler cohomology (the known cases and the
  general conjectural behavior)}, Math. Ann. \textbf{382} (2022), no.~3-4,
  1517--1590. \MR{4403229}

\bibitem[Gor99]{Hodge-survey}
B.~Brent Gordon, \emph{{Appendix B to ``A survey of the {H}odge conjecture''}},
  second ed., CRM Monograph Series, vol.~10, American Mathematical Society,
  Providence, RI, 1999. \MR{1683216}

\bibitem[GUY25]{GUY}
Ziyang Gao, Emmanuel Ullmo, and Andrei Yafaev,
  \emph{{Bi-$\overline{\mathbb{Q}}$-structures on Hermitian symmetric spaces
  and quadratic relations between CM periods}}, 2025.

\bibitem[H{\"o}r21]{Hoermann}
Fritz H{\"o}rmann, \emph{A note on formal periods}, 2021.

\bibitem[Huy16]{HuybrechtsK3book}
D.~Huybrechts, \emph{Lectures on {K3} surfaces}, Cambridge studies in advanced
  mathematics, vol. 158, Cambridge University Press, 2016.

\bibitem[HW22]{HW}
Annette Huber and Gisbert W\"ustholz, \emph{Transcendence and linear relations
  of 1-periods}, Cambridge Tracts in Mathematics, Cambridge University Press,
  2022.

\bibitem[Ji25]{Ji-Zekun}
Zekun Ji, \emph{{De Rham-Betti Groups of Type IV Abelian Varieties}}, 2025.

\bibitem[Kah25]{kahn}
Bruno Kahn, \emph{The fullness conjectures for products of elliptic curves}, J.
  Reine Angew. Math. \textbf{819} (2025), 301--318, With an appendix by Cyril
  Demarche. \MR{4857002}

\bibitem[Kaw23]{Kawabe}
Daiki Kawabe, \emph{{Grothendieck's period conjecture for Kummer surfaces of
  self-product CM type}}, 2023.

\bibitem[Lom16]{Lombardo}
Davide Lombardo, \emph{On the {$\ell$}-adic {G}alois representations attached
  to nonsimple abelian varieties}, Ann. Inst. Fourier (Grenoble) \textbf{66}
  (2016), no.~3, 1217--1245. \MR{3494170}

\bibitem[{Mar}22]{Markman-K3n}
Eyal {Markman}, \emph{{Rational Hodge isometries of hyper-Kahler varieties of
  K3[n]-type are algebraic}}, arXiv e-prints (2022), arXiv:2204.00516.

\bibitem[MZ99]{MoonenZarhin}
B.~J.~J. Moonen and Yu.~G. Zarhin, \emph{Hodge classes on abelian varieties of
  low dimension}, Math. Ann. \textbf{315} (1999), no.~4, 711--733. \MR{1731466}

\bibitem[Ser68]{Serre}
Jean-Pierre Serre, \emph{Abelian {$l$}-adic representations and elliptic
  curves}, W. A. Benjamin, Inc., New York-Amsterdam, 1968, McGill University
  lecture notes written with the collaboration of Willem Kuyk and John Labute.
  \MR{0263823}

\bibitem[Sol21]{Soldatenkov}
Andrey Soldatenkov, \emph{{Deformation Principle and Andr\'e motives of
  Projective Hyperk\"ahler Manifolds}}, International Mathematics Research
  Notices (2021), rnab108.

\bibitem[Spi99]{Spiess}
Michael Spie\ss, \emph{Proof of the {T}ate conjecture for products of elliptic
  curves over finite fields}, Math. Ann. \textbf{314} (1999), no.~2, 285--290.
  \MR{1697446}

\bibitem[Tan82]{Tankeev}
S.~G. Tankeev, \emph{Cycles on simple abelian varieties of prime dimension},
  Izv. Akad. Nauk SSSR Ser. Mat. \textbf{46} (1982), no.~1, 155--170, 192.
  \MR{643899}

\bibitem[W{\"{u}}s84]{Wuestholz}
G.~W{\"{u}}stholz, \emph{Zum {P}eriodenproblem}, Invent. Math. \textbf{78}
  (1984), no.~3, 381--391. \MR{768986}

\bibitem[Zar83]{Zarhin}
Yu.~G. Zarhin, \emph{Hodge groups of {$K3$} surfaces}, J. Reine Angew. Math.
  \textbf{341} (1983), 193--220. \MR{697317}

\end{thebibliography}

\end{document}